\documentclass[11pt,a4paper]{article}

\pdfoutput=1 
\usepackage{epsf,epsfig,amsfonts,amsgen,amsmath,amstext,amsbsy,amsopn,amsthm,cases,listings,color
}
\usepackage{ebezier,eepic}
\usepackage{color}
\usepackage{multirow}
\usepackage{epstopdf}
\usepackage{graphicx}
\usepackage{pgf,tikz}
\usepackage{mathrsfs}
\usepackage[marginal]{footmisc}
\usepackage{enumitem}
\usepackage[titletoc]{appendix}
\usepackage{booktabs}
\usepackage{url}
\usepackage{mathtools}

\usepackage{pgfplots}
\usepackage{authblk}
\usepackage{amssymb}

\usepackage{wasysym}

\usepackage{empheq}

\usepackage{dsfont}
\usetikzlibrary{patterns}

\usepackage{subfigure}
\pgfplotsset{compat=1.18}
\usepackage{mathrsfs}
\usepackage{wasysym} 
\usetikzlibrary{arrows}
\usepackage{aligned-overset}
\usepackage{bm}
\usepackage[backref=page]{hyperref}

\allowdisplaybreaks[1]

\definecolor{uuuuuu}{rgb}{0.27,0.27,0.27}
\definecolor{sqsqsq}{rgb}{0.1255,0.1255,0.1255}

\newtheorem{definition}{Definition} [section]
\newtheorem{theorem}[definition]{Theorem}

\newtheorem{proposition}[definition]{Proposition}

\newtheorem{corollary}[definition]{Corollary}
\newtheorem{conjecture}[definition]{Conjecture}
\newtheorem{claim}[definition]{Claim}
\newtheorem{problem}[definition]{Problem}

\newtheorem{fact}[definition]{Fact}
\newenvironment{pf}{\noindent{\bf Proof}}

\begin{document}
\title{\bf\Large Uniquely colorable hypergraphs}
\date{\today}
\author[1]{Xizhi Liu\thanks{Research supported by ERC Advanced Grant 101020255. Email: \texttt{xizhi.liu.ac@gmail.com}}}
\author[2]{Jie Ma\thanks{Research supported by National Key Research and Development Program of China 2023YFA1010201 and National Natural Science Foundation of China grant 12125106. Email: \texttt{jiema@ustc.edu.cn}}}
\author[2]{Tianhen Wang\thanks{Research supported by Innovation Program for Quantum Science and Technology 2021ZD0302902. Email: \texttt{wth1115060377@mail.ustc.edu.cn}}}
\author[2]{Tianming Zhu\thanks{Research supported by Innovation Program for Quantum Science and Technology 2021ZD0302902. Email: \texttt{zhutianming@mail.ustc.edu.cn}}}

\affil[1]{Mathematics Institute and DIMAP,
            University of Warwick, 
            Coventry, CV4 7AL, UK}
\affil[2]{School of Mathematical Sciences,
            University of Science and Technology of China, 
            Hefei, Anhui, 230026, China}

\maketitle
\begin{abstract}
An $r$-uniform hypergraph is uniquely $k$-colorable if there exists exactly one partition of its vertex set into $k$ parts such that every edge contains at most one vertex from each part. 
For integers $k \ge r \ge 2$, let $\Phi_{k,r}$ denote the minimum real number such that every $n$-vertex $k$-partite $r$-uniform hypergraph with positive codegree greater than $\Phi_{k,r} \cdot n$ and no isolated vertices is uniquely $k$-colorable. 
A classic result by of Bollob\'{a}s~\cite{Bol78} established that $\Phi_{k,2} = \frac{3k-5}{3k-2}$ for every $k \ge 2$. 

We consider the uniquely colorable problem for hypergraphs. 
Our main result determines the precise value of $\Phi_{k,r}$ for all $k \ge r \ge 3$.
In particular, we show that $\Phi_{k,r}$ exhibits a phase transition at approximately $k = \frac{4r-2}{3}$, a phenomenon not seen in the graph case. 
As an application of the main result, combined with a classic theorem by Frankl--F\"{u}redi--Kalai, we derive general bounds for the analogous problem on minimum positive $i$-degrees for all $1\leq i<r$, which are tight for infinitely many cases.
\medskip

\end{abstract}
\section{Introduction}\label{SEC:Intorduction}
Given an integer $r\ge 2$, an \textbf{$r$-uniform hypergraph} (henceforth \textbf{$r$-graph}) $\mathcal{H}$ is a collection of $r$-subsets of some finite set $V$.
We identify a hypergraph $\mathcal{H}$ with its edge set and use $V(\mathcal{H})$ to denote its vertex set. 
The size of $V(\mathcal{H})$ is denoted by $v(\mathcal{H})$. 

Given an $r$-graph $\mathcal{H}$, the \textbf{shadow} of $\mathcal{H}$ is 
\begin{align*}
    \partial\mathcal{H}
    \coloneqq \left\{e\in \binom{V(\mathcal{H})}{r-1} \colon \text{exists $E\in \mathcal{H}$ such that $e\subseteq  E$}\right\}.
\end{align*}
For every $(r-1)$-set $S\subseteq  V(\mathcal{H})$, the \textbf{neighborhood} of $S$ in $\mathcal{H}$ is 
\begin{align*}
    N_{\mathcal{H}}(S) 
    \coloneqq \left\{v\in V(\mathcal{H}) \colon S\cup \{v\} \in \mathcal{H}\right\}, 
\end{align*}
and the \textbf{degree} of $S$ in $\mathcal{H}$ is $d_{\mathcal{H}}(S) \coloneqq |N_{\mathcal{H}}(S)|$.
Following the definition of Balogh--Lemons--Palmer~\cite{BLP21}, the \textbf{minimum positive codegree} of $\mathcal{H}$ is given by
\begin{align*}
    \delta_{r-1}^{+}(\mathcal{H})
    \coloneqq 
    \min\left\{d_{\mathcal{H}}(e) \colon e \in \partial \mathcal{H}\right\}.
\end{align*}
Given a vertex $v\in V(\mathcal{H})$, the \textbf{link} of $v$ in $\mathcal{H}$ is 
\begin{align*}
    L_{\mathcal{H}}(v)
    \coloneqq \left\{e\in \partial\mathcal{H} \colon e\cup \{v\} \in \mathcal{H}\right\},  
\end{align*}
and the \textbf{degree} of $v$ in $\mathcal{H}$ is $d_{\mathcal{H}}(v) \coloneqq |L_{\mathcal{H}}(v)|$.
The \textbf{minimum degree} of $\mathcal{H}$ is denoted by $\delta(\mathcal{H})$.
Note that $\delta^{+}_{1}(G) = \delta(G)$ for every graph $G$ without isolated vertices.

Given two $r$-graphs $\mathcal{H}$ and $\mathcal{G}$, a map $\psi \colon V(\mathcal{H})\to V(\mathcal{G})$ is a \textbf{homomorphism}\footnote{To avoid any ambiguity, every hypergraph considered in this paper is vertex-labeled, with each vertex having a unique label.} from $\mathcal{H}$ to $\mathcal{G}$ if $\psi(e) \in \mathcal{G}$ for all $e\in \mathcal{H}$. 
An \textbf{automorphism} of $\mathcal{H}$ is simply a surjective homomorphism from $\mathcal{H}$ to itself. 
We use $\mathrm{Hom}(\mathcal{H}, \mathcal{G})$ to denote the collection of all homomorphisms from $\mathcal{H}$ to $\mathcal{G}$, and use $\mathrm{Aut}(\mathcal{H})$ to denote the collection of all automorphisms of $\mathcal{H}$.
We say $\mathcal{H}$ is \textbf{$\mathcal{G}$-colorable} if $\mathrm{Hom}(\mathcal{H}, \mathcal{G}) \neq \emptyset$. 
Two homomorphisms $\psi_1, \psi_2 \in \mathrm{Hom}(\mathcal{H}, \mathcal{G})$ are \textbf{equivalent}, denoted by $\psi_1 \cong \psi_2$, if there exists an automorphism $\eta\in \mathrm{Aut}(\mathcal{G})$ such that $\eta \circ \psi_1 = \psi_2$.

Given integers $k \ge r \ge 2$, we use $K_{k}^{r}$ to denote the complete $r$-graph on $k$ vertices. For convenience, we always assume that the vertex set of $K_{k}^{r}$ is $[k]$.
We say an $r$-graph $\mathcal{H}$ is \textbf{$k$-colorable} if it is $K_{k}^{r}$-colorbale. 
In other words,  $\mathcal{H}$ is $k$-colorable if and only if it is $k$-partite. 
A $k$-colorable $r$-graph $\mathcal{H}$ is called \textbf{uniquely $k$-colorable} if $\psi_1 \cong \psi_2$ for all $\psi_1, \psi_2 \in \mathrm{Hom}(\mathcal{H},K_{k}^{r})$.

A classical theorem by Bollob\'{a}s~\cite{Bol78} from 1970s established the tight minimum degree (equivalently, the minimum positive codegree) bound that forces an $n$-vertex $k$-partite graph to be uniquely $k$-colorable.  
\begin{theorem}[Bollob\'{a}s~\cite{Bol78}]\label{THM:Bollobas-clique}
    Let $n\ge k\ge 2$ be integers. 
    Suppose that $G$ is a $k$-partite graph on $n$ vertices with 
    \begin{align*}
        \delta(G)
        >\frac{3k-5}{3k-2}n. 
    \end{align*}
    Then $G$ is uniquely $k$-colorable. 
    Moreover, the constant $\frac{3k-5}{3k-2}$ is optimal. 
\end{theorem}
%

In this work, we extend the theorem of Bollob\'{a}s to $r$-graphs for all $r \ge 3$. 
More specifically, for integers $k \ge r \ge 2$, we study the minimum real number $\Phi_{k,r}$ such that for $n \ge k$, every $k$-partite $r$-graph on $n$ vertices without isolated vertices and with $\delta_{r-1}^{+}(\mathcal{H}) > \Phi_{k,r}\cdot n$ is uniquely $k$-colorable. 
Note that Bollob\'{a}s' theorem can be restated as $\Phi_{k,2} = \frac{3k-5}{3k-2}$ for every $k \ge 2$.

In the following theorem, we determine the exact value of $\Phi_{k,r}$ for all integers $k \ge r \ge 3$. In particular, our results reveal an interesting phenomenon$\colon$ for each fixed $r \ge 3$, $\Phi_{k,r}$ as a function of $k$ exhibits a phase transition around $\frac{4r-2}{3}$ (see Figure~\ref{fig:Phi-k-6}), a feature not seen in the case of graphs.
\begin{theorem}\label{THM:main}
    Let $n \ge k \ge r \ge 2$ be integers. 
    Suppose that $\mathcal{H}$ is a $k$-partite $r$-graph on $n$ vertices with no isolated vertices, and 
    \begin{align}\label{equ:THM:main}
        \delta_{r-1}^{+}(\mathcal{H})
        > 
        \begin{cases}
            \frac{k-r+1}{k+2}n,   & \quad\text{if}\quad \ k <  \frac{4r-2}{3}, \\
            \frac{3k-3r+1}{3k-2}n,   & \quad\text{if}\quad \ k \ge  \frac{4r-2}{3}.
        \end{cases}
    \end{align}
    Then $\mathcal{H}$ is uniquely $k$-colorable. 
    Moreover, both constants $\frac{k-r+1}{k+2}$ and $\frac{3k-3r+1}{3k-2}$ are optimal. 
\end{theorem}
\textbf{Remarks.}
\begin{itemize}
    \item The assumption that $\mathcal{H}$ has no isolated vertices cannot be omitted, as there exist $n$-vertex $k$-partite $r$-graphs satisfying~\eqref{equ:THM:main} that contain isolated vertices, which are clearly not uniquely $k$-colorable.
    \item 
    Theorem~\ref{THM:main} can be restated as 
    \begin{align*}
        \Phi_{k,r}
        = \max\left\{\frac{k-r+1}{k+2},\ \frac{3k-3r+1}{3k-2}\right\}
        \quad\text{for all}\quad k \ge r \ge 3.
    \end{align*}
    The witnesses for the lower bounds come from the constructions $\mathcal{H}_{k,r}(1, m)$ and $\mathcal{H}_{k,r}(3, m)$, defined below. 
    \item The theorem of Bollob\'as (Theorem~\ref{THM:Bollobas-clique}) employs a straightforward induction on $k$. 
    Our proofs, however, differ significantly. 
    In addition to the induction on the uniformity $r$, our approach relies on several technical innovations, notably Propositions~\ref{PROP:Prelim-B} and~\ref{PROP:Final-compute}. Proposition~\ref{PROP:Prelim-B} asserts the existence of special edges containing specified vertices and colors, while Proposition~\ref{PROP:Final-compute} provides a useful certificate for the equivalence of two homomorphisms. 
    It is worth noting that Proposition~\ref{PROP:Final-compute}, together with some additional arguments, provides a different proof of Bollob\'as' theorem (see the remark after Proposition~\ref{PROP:case-all-small}).
    \item In general, beyond the uniquely $K_{k}^r$-colorable problem, one could also explore the uniquely $F$-colorable problem, where $F$ is a fixed $r$-graph (see Section~\ref{SEC:Remarks}). 
    This type of problem was studied in~\cite{Lai87a,Lai89b} for odd cycles and in~\cite{HLZ24} for general hypergraphs.
    However, for $r\ge 3$, the complete $r$-graphs $K_{k}^{r}$ are the only family for which the tight bound is known. 
\end{itemize}
\textbf{Construction $\mathcal{H}_{k,r}(\alpha, m) \colon$}
Let $k \ge r \ge 3, m \ge 1$ be integers, and $\alpha >0$ be a real number.
Let $V_1, \ldots, V_{k-2}, V_{k-1, 1}, V_{k-1,2}, V_{k,1}, V_{k,2}$ be pairwise disjoint sets with 
\begin{align*}
    |V_1| = \cdots = |V_{k-2}| = \lfloor\alpha m \rfloor
    \quad\text{and}\quad
    |V_{k-1, 1}| = |V_{k-1,2}| = |V_{k,1}| = |V_{k,2}| = m. 
\end{align*}
Let $V \coloneqq V_1\cup \cdots\cup V_{k-2}\cup V_{k-1, 1}\cup V_{k-1,2}\cup V_{k,1}\cup V_{k,2}$, $U_1 \coloneqq V\setminus(V_{k-1,2}\cup V_{k,2})$, and $U_2 \coloneqq V\setminus(V_{k-1,1}\cup V_{k,1})$. 
Let $\mathcal{H}_{k,r}(\alpha, m)$ denote the $r$-graph on $V$ whose edge set is the union of the following two sets$\colon$
\begin{align*}
    & \left\{e \in \binom{U_1}{r} \colon \text{$|e\cap V_i| \le 1$ for $i\in [k-2]$,\  $|e\cap V_{k-1,1}| \le 1$, \text{ and } $|e\cap V_{k,1}| \le 1$} \right\}, \\
    & \left\{e \in \binom{U_2}{r} \colon \text{$|e\cap V_i| \le 1$ for $i\in [k-2]$,\  $|e\cap V_{k-1,2}| \le 1$, \text{ and } $|e\cap V_{k,2}| \le 1$} \right\}.
\end{align*}
It is clear from the definition that $\mathcal{H}_{k,r}(\alpha, m)$ is a $k$-partite $r$-graph without isolated vertices, and it has two non-equivalent $k$-colorings as follows$\colon$
\begin{align*}
    \psi_1(V_i) = i \ \text{ for }\ i \in [k-2],\quad 
    \psi_1(V_{k-1, 1}) = \psi_1(V_{k-1, 2}) = k-1,\quad 
    \psi_1(V_{k, 1}) = \psi_1(V_{k, 2}) = k. \\
    \psi_2(V_i) = i \ \text{ for }\ i \in [k-2],\quad 
    \psi_2(V_{k-1, 1}) = \psi_2(V_{k, 2}) = k-1,\quad 
    \psi_2(V_{k-1, 2}) = \psi_2(V_{k, 1}) = k.
\end{align*}


\begin{figure}[htbp]
\centering
\begin{tikzpicture}[xscale=7,yscale=7]
\draw [line width=1pt, ->] (0,0)--(1.1,0);
\draw [line width=1pt, ->] (0,0)--(0,0.65);
%

\draw[line width=0.7pt,color=cyan, dash pattern=on 1pt off 1.2pt] (0., 0.125)--(0.02, 0.135802)--(0.04, 0.146341)--(0.06, 0.156627)--(0.08, 0.166667)--(0.1, 0.176471)--(0.12, 0.186047)--(0.14, 0.195402)--(0.16, 0.204545)--(0.18, 0.213483)--(0.2, 0.222222)--(0.22, 0.230769)--(0.24, 0.23913)--(0.26, 0.247312)--(0.28, 0.255319)--(0.3, 0.263158)--(0.32, 0.270833)--(0.34, 0.278351)--(0.36, 0.285714)--(0.38, 0.292929)--(0.4, 0.3);
\draw[line width=0.7pt,color=magenta,dash pattern=on 1pt off 1.2pt] (0.4, 0.318182)--(0.466667, 0.347826)--(0.533333, 0.375)--(0.6, 0.4)--(0.666667, 0.423077)--(0.733333, 0.444444)--(0.8, 0.464286)--(0.866667, 0.482759)--(0.933333, 0.5)--(1., 0.516129);

\begin{small}
\draw [fill=uuuuuu] (1/5,2/9) circle (0.2pt);
\draw [fill=uuuuuu] (2/5,7/22) circle (0.2pt);
\draw [fill=uuuuuu] (3/5,2/5) circle (0.2pt);
\draw [fill=uuuuuu] (4/5,13/28) circle (0.2pt);
\draw [fill=uuuuuu] (1,16/31) circle (0.2pt);
\draw [fill=uuuuuu] (0, 1/8) circle (0.2pt);
\draw[color=uuuuuu] (0-0.05,1/8) node {$\frac{1}{8}$};
%
\draw [fill=uuuuuu] (0,0) circle (0.2pt);
\draw[color=uuuuuu] (0,0-0.05) node {$6$};
\draw [fill=uuuuuu] (1/5,0) circle (0.2pt);
\draw[color=uuuuuu] (1/5,0-0.05) node {$7$};
\draw [fill=uuuuuu] (2/5,0) circle (0.2pt);
\draw[color=uuuuuu] (2/5,0-0.05) node {$8$};
\draw [fill=uuuuuu] (3/5,0) circle (0.2pt);
\draw[color=uuuuuu] (3/5,0-0.05) node {$9$};
\draw [fill=uuuuuu] (4/5,0) circle (0.2pt);
\draw[color=uuuuuu] (4/5,0-0.05) node {$10$};
\draw [fill=uuuuuu] (1,0) circle (0.2pt);
\draw[color=uuuuuu] (1,0-0.05) node {$11$};
%
\draw[color=uuuuuu] (1+0.1,0-0.05) node {$k$};
\draw[color=uuuuuu] (0-0.08,0.54+0.08) node {$\Phi_{k,6}$};
\end{small}
\end{tikzpicture}
\caption{$\Phi_{k,6}$ for $k\in [6,11]$.}
\label{fig:Phi-k-6}
\end{figure}
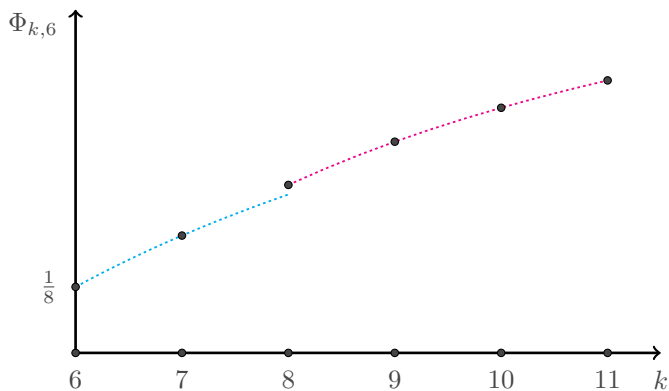

Given integers $r > i \ge 1$, the \textbf{$i$-th shadow} $\partial_{i}\mathcal{H}$ of an $r$-graph $\mathcal{H}$ is defined inductively by letting $\partial_{i}\mathcal{H} \coloneqq \partial\partial_{i-1}\mathcal{H}$ with $\partial_{1}\mathcal{H} \coloneqq \partial\mathcal{H}$.
For every $i$-set $S\subseteq  V(\mathcal{H})$, the \textbf{link} of $S$ in $\mathcal{H}$ is 
\begin{align*}
    L_{\mathcal{H}}(S)
    \coloneqq \left\{e\in \partial_{i}\mathcal{H} \colon S\cup e \in \mathcal{H}\right\}, 
\end{align*}
and the \textbf{degree} of $S$ is $d_{\mathcal{H}}(S) \coloneqq |L_{\mathcal{H}}(S)|$. 
Similar to the definition of $\delta^{+}_{r-1}(\mathcal{H})$, the \textbf{positive minimum $i$-degree of $\mathcal{H}$} is defined as 
\begin{align*}
    \delta^{+}_{i}(\mathcal{H})
    \coloneqq \min\left\{d_{\mathcal{H}}(e) \colon e\in \partial_{r-i}\mathcal{H}\right\}. 
\end{align*}
%
For integers $k \ge r> i \ge 1$, let $\Phi_{k,r,i}$ denote the minimum real number such that for $n \ge k$, every $n$-vertex $k$-partite $r$-graph $\mathcal{H}$ without isolated vertices and with $\delta^{+}_{i}(\mathcal{H}) > \Phi_{k,r,i} \cdot n^{r-i}$ is uniquely $k$-colorable. 

We establish the following inequality concerning $\Phi_{k,r,i}$ using a classical Kruskal--Katona-type theorem by Frankl--F\"{u}redi--Kalai~\cite{FFK88}.

\begin{theorem}\label{THM:s-degree-beta}
    Suppose that $k \ge r_1 \ge r_2 >i \ge 1$ are integers.
    Then
    \begin{align*}
        \left ( \frac{\Phi_{k,r_1,i}}{\binom{k-i}{r_1-i}} \right ) ^{1/(r_1-i)} \le  \left ( \frac{\Phi_{k,r_2,i}}{\binom{k-i}{r_2-i}} \right ) ^{1/(r_2-i)}.
    \end{align*}
\end{theorem}
As an application of the main theorem (i.e.  Theorem~\ref{THM:main}), we provide the following general upper bound for $\Phi_{k,r,i}$, which is tight when $k \le \frac{4i+2}{3}$.
\begin{corollary}\label{CORO:i-degree-tight-bound}
    Suppose that $k \ge r > i \ge 1$ are integers. 
    Then 
    \begin{align*}
        \Phi_{k,r,i}
        \le \binom{k-i}{r-i} \cdot \left(\frac{1}{k-i} \cdot \max \left \{ \frac{k-i}{k+2} ,\ \frac{3k-3i-2}{3k-2} \right \}\right)^{r-i}. 
    \end{align*}
    In particular, 
    \begin{align*}
        \Phi_{k,r,i}
        = \binom{k-i}{r-i} \left(\frac{1}{k+2}\right)^{r-i} 
        \quad\text{for all}\quad k \le \frac{4i+2}{3}.
    \end{align*}
\end{corollary}
The remainder of this paper is organized as follows$\colon$
The proof of Theorem~\ref{THM:main} is presented in Section~\ref{SEC:proof-r-1-degree}. 
Proofs for Theorem~\ref{THM:s-degree-beta} and Corollary~\ref{CORO:i-degree-tight-bound} are presented in Section~\ref{SEC:proof-i-degree}.
Section~\ref{SEC:Remarks} includes some remarks and open questions.

\section{Proof of Theorem~\ref{THM:main}}\label{SEC:proof-r-1-degree}
\subsection{Preliminaries}
In this subsection, we present several necessary definitions and preliminary results. 

Given an $r$-graph $\mathcal{H}$ and a vertex $v \in V(\mathcal{H})$, 
the \textbf{neighborhood} of $v$ is 
    \begin{align*}
        N_{\mathcal{H}}(v)
        \coloneqq 
        \left\{u \in V(\mathcal{H})\setminus\{v\} \colon \ \text{exists } e\in \mathcal{H} \text{ such that } \{u,v\} \subseteq  e\right\}.
    \end{align*}
Throughout this subsection, we will assume the following conditions$\colon$
\begin{enumerate}[label=(\textbf{\alph*})]
    \item\label{assume:1} $n, k, r$ are integers satisfying $n \ge k \ge r \ge 2$. 
    \item\label{assume:2} $\mathcal{H}$ is a $k$-colorable $n$-vertex $r$-graph without isolated vertices and satisfies
    \begin{align}\label{equ:assumption-mincodegree}
        \delta_{r-1}^{+}(\mathcal{H})
        > \max \left \{ \frac{3k-3r+1}{3k-2} n,\ \frac{k-r+1}{k+2} n \right \}.
    \end{align}
\end{enumerate}
Given a vertex set $S \subseteq  V(\mathcal{H})$, we classify $S$ as follows$\colon$
\begin{itemize}
    \item $S$ is \textbf{good} if $|S| \ge \frac{n}{k+2}$; otherwise, it is \textbf{bad}. 
    \item $S$ is \textbf{large} if $|S| \ge \frac{3 n}{3k-2}$; otherwise, it is \textbf{small}.  
\end{itemize}
For every $\varphi \in \mathrm{Hom}(\mathcal{H},K_{k}^{r})$, define 
\begin{align*}
    I_{\varphi} 
    \coloneqq \left\{j\in[k]\colon \varphi^{-1}(j)\ \text{is good}\right\}
    \quad\text{and}\quad 
    J_{\varphi} 
    \coloneqq \left\{j\in[k]\colon \varphi^{-1}(j)\ \text{is large}\right\}. 
\end{align*}
Note that, by definition, every large set is good, and hence, $J_{\varphi} \subseteq I_{\varphi}$.

For every vertex $v \in  V(\mathcal{H})$, let
\begin{align*}
    [v]_{\varphi}
    \coloneqq \left\{u \in V(\mathcal{H}) \colon \varphi(u) = \varphi(v)\right\}. 
\end{align*}
For a subset $A\subseteq V(\mathcal{H})$, we denote by $\varphi(A)$ the set of all colors $\varphi(v)$ for $v\in A$. 
It will be convenient later to set $\overline{S} \coloneqq [k]\setminus S$ for every $S \subseteq  [k]$. 
In particular, for every $\varphi \in \mathrm{Hom}(\mathcal{H}, K_{k}^{r})$, we have $\overline{I_{\varphi}} = [k]\setminus I_{\varphi}$, $\overline{J_{\varphi}} = [k]\setminus J_{\varphi}$, and $\overline{\varphi(e)} = [k]\setminus \varphi(e)$ for every $e \in \mathcal{H}$. 

\begin{proposition}\label{PROP:Prelim-A}
    The following statements hold for every $\varphi \in \mathrm{Hom}(\mathcal{H},K_{k}^{r})$. 
    \begin{enumerate}[label=(\roman*)]
        \item\label{PROP:Prelim-A-1} 
        $\left|\varphi^{-1}(J)\right| \ge \frac{|J|}{k-r+1} \cdot \delta_{r-1}^{+}(\mathcal{H})$ for every set $J \subseteq  [k]$ of size at least $k-r+1$.
        \item\label{PROP:Prelim-A-2} $\varphi$ is surjective. 
        \item\label{PROP:Prelim-A-3} $|J_{\varphi}| \le r-2$. 
    \end{enumerate}    
\end{proposition}
\begin{proof}[Proof of Proposition~\ref{PROP:Prelim-A}]
    First we prove Proposition~\ref{PROP:Prelim-A}~\ref{PROP:Prelim-A-1}. 
    Note that it suffices to show this for every $J \in \binom{[k]}{k-r+1}$, as the other cases follow from a standard averaging argument. 
    Fix a set $J \subseteq  [k]$ of size $k-r+1$. 
    Let $e \in \mathcal{H}$ be an edge such that $|\varphi(e) \cap \overline{J}|$ is maximized.
    Since $|\overline{J}|=r-1$, there exists a vertex $v\in e \setminus \varphi^{-1}(\overline{J})$. 
    Let $e' \coloneqq e\setminus \{v\}$.
    It follows from the maximality of $e$ that $N_{\mathcal{H}}(e') \subseteq  \varphi^{-1}(J)$. Therefore, $|\varphi^{-1}(J)| \ge d_{\mathcal{H}}(e') \ge \delta_{r-1}^{+}(\mathcal{H})$.

    Next, we prove Proposition~\ref{PROP:Prelim-A}~\ref{PROP:Prelim-A-2}. 
    The case $k = r$ is trivially true, so we may assume that $k\ge r+1$. 
    Suppose to the contrary that $\varphi$ is not surjective. 
    Then there exists a homomorphism $\psi \in \mathrm{Hom}(\mathcal{H}, K_{k-1}^{r})$. 
    Let $V_i \coloneqq \psi^{-1}(i)$ for $i \in [k-1]$. 
    Proposition~\ref{PROP:Prelim-A}~\ref{PROP:Prelim-A-1} applied to $\psi$ yields 
    \begin{align*}
        n 
        = \sum_{i\in [k-1]}|V_i|
        = |\psi^{-1}([k-1])|
        \ge \frac{k-1}{(k-1)-r+1} \cdot \delta^{+}_{r-1}(\mathcal{H}),
    \end{align*}
    which implies that $\delta_{r-1}^{+}(\mathcal{H}) \le \frac{k-r}{k-1} n$.
    Simple calculations show that $\frac{k-r}{k-1}<\frac{3k-3r+1}{3k-2}$, which contradicts Inequality~\eqref{equ:assumption-mincodegree}. 

    Now we prove Proposition~\ref{PROP:Prelim-A}~\ref{PROP:Prelim-A-3}.
    Suppose to the contrary that $|J_{\varphi}| \ge r-1$. 
    Fix an edge $e\in \mathcal{H}$ such that $|\varphi(e) \cap J_{\varphi}|$ is maximized. 
    
    Suppose that $|\varphi(e) \cap J_{\varphi}| \ge r-1$. 
    Then there exists a vertex $v \in e$ such that $\varphi(e\setminus \{v\}) \subseteq  J_{\varphi}$. 
    It follows from the definition of $J_{\varphi}$ that 
    \begin{align*}
        d_{\mathcal{H}}(e\setminus \{v\})
        \le n - \sum_{j \in \varphi(e\setminus \{v\})} |V_j|
        \le n - (r-1) \cdot \frac{3 n}{3k-2}
        = \frac{3k-3r+1}{3k-2} n, 
    \end{align*}
    contradicting Inequality~\eqref{equ:assumption-mincodegree}. 
    
    Suppose that $|\varphi(e) \cap J_{\varphi}| \le r-2$. 
    Then fix a vertex $v\in e$ such that $\varphi(v)\notin J_{\varphi}$. 
    Note that the set $M \coloneqq J_{\varphi}\cup \varphi(e\setminus \{v\})$ satisfies $|M| \ge  |\varphi(e\setminus \{v\})|+1 \ge  r$. 
    It follows from the maximality of $e$ and the definition of $J_{\varphi}$ that 
    \begin{align*}
        d_{\mathcal{H}}(e\setminus \{v\})
        = 
        \sum_{j \in \overline{M}} |\varphi^{-1}(j)|
        \le (k-r) \cdot \frac{3 n}{3k-2}
        < \frac{3k-3r+1}{3k-2} n, 
    \end{align*}
    contradicting Inequality~\eqref{equ:assumption-mincodegree}.
\end{proof}
The following simple but crucial proposition will be used extensively throughout the paper.
\begin{proposition}\label{PROP:Prelim-B}
    The following statements hold for every $\varphi \in \mathrm{Hom}(\mathcal{H},K_{k}^{r})$, $i\in [k]$, and $v \in \varphi^{-1}(i)$. 
    \begin{enumerate}[label=(\roman*)]
        \item\label{PROP:Prelim-B-1} Suppose that  $i\in J_{\varphi}$.
        Then for every $(r-|J_{\varphi}|)$-set $I \subseteq \overline{J_{\varphi}}$, there exists an edge $e \in \mathcal{H}$ containing $v$ such that  $\varphi(e)=J_{\varphi}\cup I$.
        \item\label{PROP:Prelim-B-2} Suppose that $i\in \overline{J_{\varphi}}$.
        Then for every $(r-1-|J_{\varphi}|)$-set $I \subseteq \overline{J_{\varphi}}$, there exists an edge $e \in \mathcal{H}$ containing $v$ such that  $\varphi(e)=J_{\varphi}\cup I\cup\{i\}$.
        \item\label{PROP:Prelim-B-3} For every $j \in [k] \setminus \{i\}$, there exists an edge $e\in \mathcal{H}$ containing $v$ such that $j \in \varphi(e)$ and 
        \begin{align*}
            \min\left\{|\varphi^{-1}(\ell)| \colon \ell \in \varphi(e) \setminus \{i,j\}\right\}
            \ge 
            \max\left\{|\varphi^{-1}(\ell)| \colon \ell \in \overline{\varphi(e)} \right\}.
        \end{align*}
        \item\label{PROP:Prelim-B-4} Suppose that $J_{\varphi} = \emptyset$.
        Then for every $(r-1)$-set $I \subseteq [k]$ there exists an edge $e\in \mathcal{H}$ containing $v$ such that $I \subseteq  \varphi(e)$.
    \end{enumerate}    
\end{proposition}
\begin{proof}[Proof of Proposition~\ref{PROP:Prelim-B}]
    We will present the proof for Proposition~\ref{PROP:Prelim-B}~\ref{PROP:Prelim-B-1} only, as the proof for Proposition~\ref{PROP:Prelim-B}~\ref{PROP:Prelim-B-2} is nearly identical. 
    Additionally, Proposition~\ref{PROP:Prelim-B}~\ref{PROP:Prelim-B-3} and~\ref{PROP:Prelim-B-4} follow directly from Proposition~\ref{PROP:Prelim-B}~\ref{PROP:Prelim-B-1}~\ref{PROP:Prelim-B-2} and the fact that $|J_{\varphi}|\le r-2$ (see Proposition~\ref{PROP:Prelim-A}~\ref{PROP:Prelim-A-3}).
    
    Fix $\varphi \in \mathrm{Hom}(\mathcal{H},K_{k}^{r})$, $i\in J_{\varphi}$, and $v \in \varphi^{-1}(i)$. 
    Fix an arbitrary set $I \subseteq  \overline{J_{\varphi}}$ of size $r - |J_{\varphi}|$. 
    First, we show that there exists an edge $e \in \mathcal{H}$ containing $\{v\}$ such that $J_{\varphi} \subseteq \varphi(e)$. 
    
    Let $e \in \mathcal{H}$ be an edge containing $v$ such that $|\varphi(e) \cap J_{\varphi}|$ is maximized. 
    Suppose to the contrary that $|\varphi(e) \cap J_{\varphi}|\le |J_{\varphi}|-1$. 
    Then it follows from Proposition~\ref{PROP:Prelim-A}~\ref{PROP:Prelim-A-3} that there exists a vertex $u\in e$ such that $[u]_{\varphi}$ is small. 
    Similar to the proof of Proposition~\ref{PROP:Prelim-A}~\ref{PROP:Prelim-A-3}, the set $M' \coloneqq J_{\varphi} \cup \varphi(e\setminus \{u\})$ has size at least $r$.
    Thus, by the maximality of $e$, we have 
    \begin{align*}
        d_{\mathcal{H}}(e\setminus\{u\})
        \le \sum_{j\in \overline{M'}} |\varphi^{-1}(j)|
        \le (k-r) \cdot \frac{3 n}{3k-2}
        < \frac{3k-3r+1}{3k-2} n,
    \end{align*}
    contradicting Inequality~\eqref{equ:assumption-mincodegree}.
    Therefore, we have $J_{\varphi} \subseteq  \varphi(e)$. 

    Let $\tilde{e} \in \mathcal{H}$ be an edge that contains $v$ and satisfies $J_{\varphi} \subseteq  \varphi(\tilde{e})$, such that $|\varphi(\tilde{e}) \cap I|$ is maximized among all such edges. 
    Suppose to the contrary that $|\varphi(\tilde{e}) \cap I| \le |I| - 1$. 
    Then there exists a vertex $w \in \tilde{e}$ such that $\varphi(w) \in \overline{I \cup J_{\varphi}}$. 
    In particular, $[w]_{\varphi}$ is small. 
    Similar to the argument above, the set $M \coloneqq I \cup \varphi(e\setminus \{w\})$ has size least $r$. 
    Thus, by the maximality of $\tilde{e}$, we have  
    \begin{align*}
        d_{\mathcal{H}}(e\setminus\{w\})
        \le \sum_{j\in \overline{M}} |\varphi^{-1}(j)|
        \le (k-r) \cdot \frac{3 n}{3k-2}
        < \frac{3k-3r+1}{3k-2} n,
    \end{align*}
    contradicting Inequality~\eqref{equ:assumption-mincodegree}.
\end{proof}
The following result provides a lower bound for the size of $I_{\varphi}$.
\begin{proposition}\label{PROP:Prelim-D}
    We have $|I_{\varphi}| \ge r$ for every $\varphi \in \mathrm{Hom}(\mathcal{H},K_{k}^{r})$.
\end{proposition}
\begin{proof}[Proof of Proposition~\ref{PROP:Prelim-D}]
    Suppose to the contrary that $|I_{\varphi}| \le r-1$ for some $\varphi \in \mathrm{Hom}(\mathcal{H},K_{k}^{r})$.
    Then it follows from Proposition~\ref{PROP:Prelim-B}~\ref{PROP:Prelim-B-1} (or~\ref{PROP:Prelim-B-2}) and the fact $J_{\varphi} \subseteq  I_{\varphi}$ that there exists an edge $e \in \mathcal{H}$ such that $I_{\varphi} \subseteq  \varphi(e)$. 
    In particular, $|\varphi^{-1}(j)| \le \frac{n}{k+2}$ for every $j \in \overline{\varphi(e)}$. 
    Since $|I_{\varphi}| \le r-1$, there exists a vertex $u \in e$ such that $\varphi(u) \in \overline{I_{\varphi}}$. 
    It follows that  
    \begin{align*}
        d_{\mathcal{H}}(e \setminus \{u\})
        \le \sum_{j\in \overline{\varphi(e\setminus \{u\})}} |\varphi^{-1}(j)|
        \le (k-r+1) \cdot \frac{n}{k+2},
    \end{align*}
    contradicting Inequality~\eqref{equ:assumption-mincodegree}.
\end{proof}
%
\subsection{Two key propositions}
In this subsection, we present two technical but crucial results that are necessary for the proof of Theorem~\ref{THM:main}. 
The proofs of these results are deferred to  Sections~\ref{SEC:proof-Prop-small} and~\ref{SEC:proof-final-compute}.

We continue to assume that Assumptions~\ref{assume:1} and~\ref{assume:2} hold throughout this subsection. 
Recall that two homomorphisms $\varphi,\vartheta \in \mathrm{Hom}(\mathcal{H},K_{k}^{r})$ are equivalent (denoted by $\vartheta\cong \varphi$) if there exists an automorphism $\eta\in \mathrm{Aut}(K_{k}^{r})$ such that $\eta \circ \varphi = \vartheta$.
%
\begin{proposition}\label{PROP:Final-compute}
    Let $\varphi,\vartheta \in \mathrm{Hom}(\mathcal{H},K_{k}^{r})$. 
    If 
    $|\vartheta(\varphi^{-1}(i))| \le 2$ for all $i\in [k]$, then $\vartheta\cong \varphi$. 
\end{proposition}
We say a homomorphism $\varphi \in \mathrm{Hom}(\mathcal{H},K_{k}^{r})$ is \textbf{small} if $\varphi^{-1}(i)$ is small (i.e. $|\varphi^{-1}(i)| < \frac{3n}{3k-2}$) for every $i\in [k]$.
Note that $\varphi$ is small if and only if $J_{\varphi} = \emptyset$. 
\begin{proposition}\label{PROP:case-all-small}
    Suppose that every homomorphism in $\mathrm{Hom}(\mathcal{H},K_{k}^{r})$ is small. 
    Then for every pair of homomorphisms $\vartheta, \varphi \in \mathrm{Hom}(\mathcal{H},K_{k}^{r})$ and for every $i\in [k]$, we have $|\vartheta(\varphi^{-1}(i))| \le 2$. 
    Consequently, $\mathcal{H}$ is uniquely $k$-colorable. 
\end{proposition}
\textbf{Remark.}
    Notice that in the case where $r=2$ (i.e. for graphs), for every $\varphi \in \mathrm{Hom}(\mathcal{H},K_{k}^{r})$ we have $|J_\varphi| \le  r-2 = 0$ (by Proposition~\ref{PROP:Prelim-A}~\ref{PROP:Prelim-A-3}), meaning that $\varphi$ is small. 
    Consequently, it follows from Proposition~\ref{PROP:case-all-small} that $\mathcal{H}$ is uniquely $k$-colorable. 
    This provides an alternative proof of Bollob\'{a}s' theorem (Theorem~\ref{THM:Bollobas-clique}). 

\subsection{Proof of Theorem~\ref{THM:main}}
We present the proof of Theorem~\ref{THM:main} in this subsection, addressing it in two separate cases$\colon$
\begin{proposition}\label{PROP:case-k-large}
    Theorem~\ref{THM:main} holds for $k \ge (4r-2)/3$.
\end{proposition}
\begin{proposition}\label{PROP:case-k-small}
    Theorem~\ref{THM:main} holds for $k < (4r-2)/3$.
\end{proposition}
First, we prove Theorem~\ref{THM:main} for the case $k \ge (4r-2)/3$.
\begin{proof}[Proof of Proposition~\ref{PROP:case-k-large}]
    By contradiction, let $r$ be the smallest integer for which Proposition~\ref{PROP:case-k-large} fails. 
    By Theorem~\ref{THM:Bollobas-clique}, we may assume that $r \ge 3$. 
    Let $\mathcal{H}$ be an $n$-vertex $k$-partite $r$-graph without isolated vertices, and with $\delta_{r-1}^{+}(\mathcal{H})>\frac{3k-3r+1}{3k-2}n$, which is not uniquely $k$-colorable.
    Let $V \coloneqq V(\mathcal{H})$. 

    Fix $\varphi \in \mathrm{Hom}(\mathcal{H},K_{k}^{r})$ such that $|J_{\varphi}|$ is maximized. 
    Let $V_{i} \coloneqq \varphi^{-1}(i)$ for $i\in [k]$. 
    Let $q \coloneqq |J_{\varphi}|$. 
    By relabelling vertices in $K_{k}^{r}$, we may assume that $J_{\varphi} = [q]$, i.e. $V_1, \ldots, V_{q}$ are large sets. 
    We may assume that $q \ge 1$, since otherwise, by Proposition~\ref{PROP:case-all-small}, we are done. 
    Recall from Proposition~\ref{PROP:Prelim-A}~\ref{PROP:Prelim-A-3} that we have $q \le r-2$ as well.

    Fix an arbitrary homomorphism $\vartheta \in \mathrm{Hom}(\mathcal{H},K_{k}^{r})$.

    \begin{claim}\label{CLAIM:k-large-induction}
        For every $i \in [q]$ and $v \in V_i$, the following statements hold$\colon$
        \begin{enumerate}[label=(\roman*)]
            \item\label{CLAIM:k-large-induction-1} $|\vartheta(N_{\mathcal{H}}(v) \cap V_{j})|=1$ for every $j\in [k]\setminus \{i\}$, 
            \item\label{CLAIM:k-large-induction-2} $\vartheta(N_{\mathcal{H}}(v) \cap V_{j}) \neq \vartheta(v)$ for every $j\in [k]\setminus \{i\}$, and 
            \item\label{CLAIM:k-large-induction-3} $\vartheta(N_{\mathcal{H}}(v) \cap V_{j}) \neq \vartheta(N_{\mathcal{H}}(v) \cap V_{j'})$ for all distinct $j,j' \in [k]\setminus \{i\}$. 
        \end{enumerate}    
    \end{claim}
    \begin{proof}[Proof of Claim~\ref{CLAIM:k-large-induction}] 
        Fix $i \in [q]$ and $v \in V_i$. 
        We may assume that $\vartheta(v)=i$; otherwise, we can replace $\vartheta$ with $\eta \circ \vartheta$, where $\eta \in \mathrm{Aut}(K_{k}^{r})$ satisfies $\eta \circ \vartheta(v) = i$. 
        Let $N \coloneqq N_{\mathcal{H}}(v)$, noting that $N \subseteq  V \setminus V_i$. 
        Since $i \in [q]$, we have $|N| \le n - |V_i| \le n-\frac{3n}{3k-2}$. 
        Additionally, it follows from the assumption $k \ge \frac{4r-2}{3}$ that $k-1 \ge \frac{4(r-1)-2}{3}$. 
        We consider the link $L_{\mathcal{H}}(v)$ as an $(r-1)$-graph on $N$, and, in particular, $L_{\mathcal{H}}(v)$ contains no isolated vertices. 
        Since $L_{\mathcal{H}}(v)$ satisfies  
        \begin{align*}
            \delta_{r-2}^{+}(L_{\mathcal{H}}(v)) 
            \ge \delta_{r-1}^{+}(\mathcal{H})
            > \frac{3k-3r+1}{3k-2}n
            & = \frac{3k-3r+1}{3k-5} \left(n-\frac{3n}{3k-2} \right) \\ 
            & \ge \frac{3k-3r+1}{3k-5}|N|, 
        \end{align*}
        it follows from the minimality of $r$ that $L_{\mathcal{H}}(v)$ is uniquely $(k-1)$-colorable.  
        Therefore, the induced maps $\left.\varphi\right|_{N}, \left.\vartheta\right|_{N}$ (viewed as homomorphisms from $L_{\mathcal{H}}(v)$ to $K_{k-1}^{r-1}$) are equivalent. 
        This implies Claim~\ref{CLAIM:k-large-induction}~\ref{CLAIM:k-large-induction-1},~\ref{CLAIM:k-large-induction-2}, and~\ref{CLAIM:k-large-induction-3}. 
    \end{proof}

    Our next step is to show that $|\vartheta(V_i)|\le  2$ for every $i \in [k]$. 
    By Proposition~\ref{PROP:Final-compute}, this will complete the proof of Proposition~\ref{PROP:case-k-large}. 
    We will achieve this in two cases, addressed in Claims~\ref{Claim:small-part} and~\ref{Claim:large-part}, respectively. 
    %
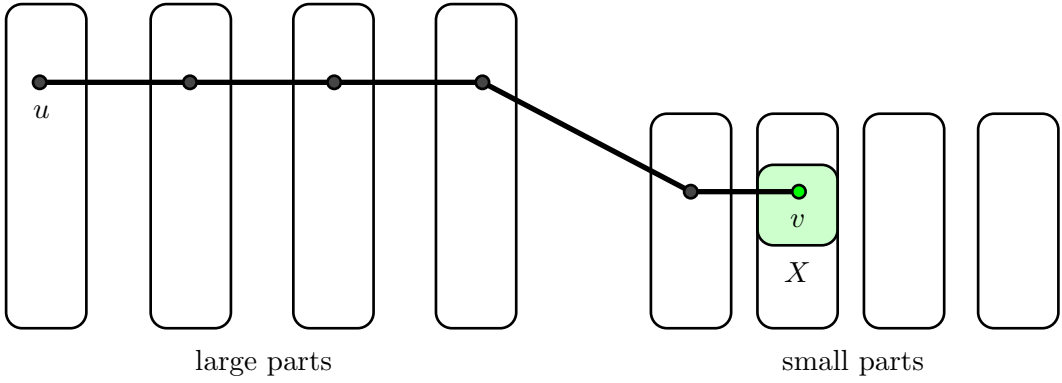
\begin{figure}[htbp]
\centering
\tikzset{every picture/.style={line width=1pt}} 

\begin{tikzpicture}[x=0.75pt,y=0.75pt,yscale=-1,xscale=1]

\draw   (68.24,48.66) .. controls (68.24,44.23) and (71.83,40.65) .. (76.25,40.65) -- (100.29,40.65) .. controls (104.71,40.65) and (108.3,44.23) .. (108.3,48.66) -- (108.3,195.68) .. controls (108.3,200.1) and (104.71,203.69) .. (100.29,203.69) -- (76.25,203.69) .. controls (71.83,203.69) and (68.24,200.1) .. (68.24,195.68) -- cycle ;
\draw   (211.61,48.66) .. controls (211.61,44.23) and (215.19,40.65) .. (219.62,40.65) -- (243.65,40.65) .. controls (248.08,40.65) and (251.67,44.23) .. (251.67,48.66) -- (251.67,195.68) .. controls (251.67,200.1) and (248.08,203.69) .. (243.65,203.69) -- (219.62,203.69) .. controls (215.19,203.69) and (211.61,200.1) .. (211.61,195.68) -- cycle ;
\draw   (140.4,48.66) .. controls (140.4,44.23) and (143.99,40.65) .. (148.41,40.65) -- (172.45,40.65) .. controls (176.87,40.65) and (180.46,44.23) .. (180.46,48.66) -- (180.46,195.68) .. controls (180.46,200.1) and (176.87,203.69) .. (172.45,203.69) -- (148.41,203.69) .. controls (143.99,203.69) and (140.4,200.1) .. (140.4,195.68) -- cycle ;
\draw   (390.1,103.8) .. controls (390.1,99.38) and (393.69,95.79) .. (398.11,95.79) -- (422.15,95.79) .. controls (426.57,95.79) and (430.16,99.38) .. (430.16,103.8) -- (430.16,195.68) .. controls (430.16,200.1) and (426.57,203.69) .. (422.15,203.69) -- (398.11,203.69) .. controls (393.69,203.69) and (390.1,200.1) .. (390.1,195.68) -- cycle ;
\draw   (443.27,103.8) .. controls (443.27,99.38) and (446.86,95.79) .. (451.28,95.79) -- (475.32,95.79) .. controls (479.74,95.79) and (483.33,99.38) .. (483.33,103.8) -- (483.33,195.68) .. controls (483.33,200.1) and (479.74,203.69) .. (475.32,203.69) -- (451.28,203.69) .. controls (446.86,203.69) and (443.27,200.1) .. (443.27,195.68) -- cycle ;
\draw   (496.44,103.8) .. controls (496.44,99.38) and (500.02,95.79) .. (504.45,95.79) -- (528.48,95.79) .. controls (532.91,95.79) and (536.5,99.38) .. (536.5,103.8) -- (536.5,195.68) .. controls (536.5,200.1) and (532.91,203.69) .. (528.48,203.69) -- (504.45,203.69) .. controls (500.02,203.69) and (496.44,200.1) .. (496.44,195.68) -- cycle ;
\draw   (282.81,48.66) .. controls (282.81,44.23) and (286.4,40.65) .. (290.83,40.65) -- (314.86,40.65) .. controls (319.29,40.65) and (322.87,44.23) .. (322.87,48.66) -- (322.87,195.68) .. controls (322.87,200.1) and (319.29,203.69) .. (314.86,203.69) -- (290.83,203.69) .. controls (286.4,203.69) and (282.81,200.1) .. (282.81,195.68) -- cycle ;
\draw   (553.44,103.8) .. controls (553.44,99.38) and (557.02,95.79) .. (561.45,95.79) -- (585.48,95.79) .. controls (589.91,95.79) and (593.5,99.38) .. (593.5,103.8) -- (593.5,195.68) .. controls (593.5,200.1) and (589.91,203.69) .. (585.48,203.69) -- (561.45,203.69) .. controls (557.02,203.69) and (553.44,200.1) .. (553.44,195.68) -- cycle ;
\draw [fill=green, fill opacity=0.2]  (443.27,129.51) .. controls (443.27,125.08) and (446.86,121.5) .. (451.28,121.5) -- (475.32,121.5) .. controls (479.74,121.5) and (483.33,125.08) .. (483.33,129.51) -- (483.33,153.97) .. controls (483.33,158.4) and (479.74,161.98) .. (475.32,161.98) -- (451.28,161.98) .. controls (446.86,161.98) and (443.27,158.4) .. (443.27,153.97) -- cycle ;
\draw  [line width=2pt]  (85,80) -- (160,80) -- (232,80) -- (306,80) -- (410,135) -- (464,135) ;
\draw [fill=uuuuuu] (85,80) circle (2.5pt);
\draw [fill=uuuuuu] (160,80) circle (2.5pt);
\draw [fill=uuuuuu] (232,80) circle (2.5pt);
\draw [fill=uuuuuu] (306,80) circle (2.5pt);
\draw [fill=uuuuuu] (410,135) circle (2.5pt);
\draw [fill=green] (464,135) circle (2.5pt);
\draw (454,213) node [anchor=north west][inner sep=0.75pt]   [align=left] {\text{small parts}};
\draw (161,213) node [anchor=north west][inner sep=0.75pt]   [align=left] {\text{large parts}};
\draw (458,145) node [anchor=north west][inner sep=0.75pt]   [align=left] {$v$};
\draw (80,90) node [anchor=north west][inner sep=0.75pt]   [align=left] {$u$};
\draw (455,169) node [anchor=north west][inner sep=0.75pt]   [align=left] {$X$};

\end{tikzpicture}

\caption{auxiliary figure for the proof of Claim~\ref{Claim:small-part}.} 
\label{Fig:Claim2-9}
\end{figure}
    %
    \begin{claim}\label{Claim:small-part}
        We have $|\vartheta(V_i)|\le  2$ for every $i \in [q+1, k]$. 
    \end{claim}
    \begin{proof}[Proof of Claim~\ref{Claim:small-part}]
        Fix $i \in [q+1, k]$ and $v \in V_i$. 
        By Proposition~\ref{PROP:Prelim-B}, there exists an edge $e \in \mathcal{H}$ containing $v$ such that  $[q] \subseteq  \varphi(e)$ (see Figure~\ref{Fig:Claim2-9}). 
        Let 
        \begin{align*}
            X 
            \coloneqq V_{i} \cap N_{\mathcal{H}}(e \setminus \{v\}) 
            = N_{\mathcal{H}}(e\setminus\{v\}) \setminus \bigcup_{j\in \overline{\varphi(e)}} V_{j}.
        \end{align*}
        Since $V_j$ is small for $j \in \overline{\varphi(e)} \subseteq [q+1, k]$, we have 
        \begin{align*}
            |X|
            \ge d_{\mathcal{H}}(e \setminus \{v\}) - \sum_{j \in \overline{\varphi(e)}}|V_{j}|
            > \frac{3k-3r+1}{3k-2}n-\left( k-r\right) \cdot \frac{3n}{3k-2}
            =\frac{n}{3k-2}. 
        \end{align*}
        Let $u$ denote the vertex in $e\cap V_1$, noting that $X \subseteq  N_{\mathcal{H}}(u)$. 
        Claim~\ref{CLAIM:k-large-induction}~\ref{CLAIM:k-large-induction-1} applied to $u$ shows that $|\vartheta(X)| = 1$, and hence, $X \subseteq  [v]_{\vartheta}$ (recall from the definition that $v\in X$).
        It follows that 
        \begin{align*}
            |[v]_{\vartheta} \cap V_i|
            \ge |X|
            > \frac{n}{3k-2} 
            \ge \frac{|V_i|}{3}.
        \end{align*}
        Since $v$ was chosen arbitrarily, we conclude that $|\vartheta(V_i)| < \frac{|V_i|}{|V_i|/3} = 3$. 
    \end{proof}
    %
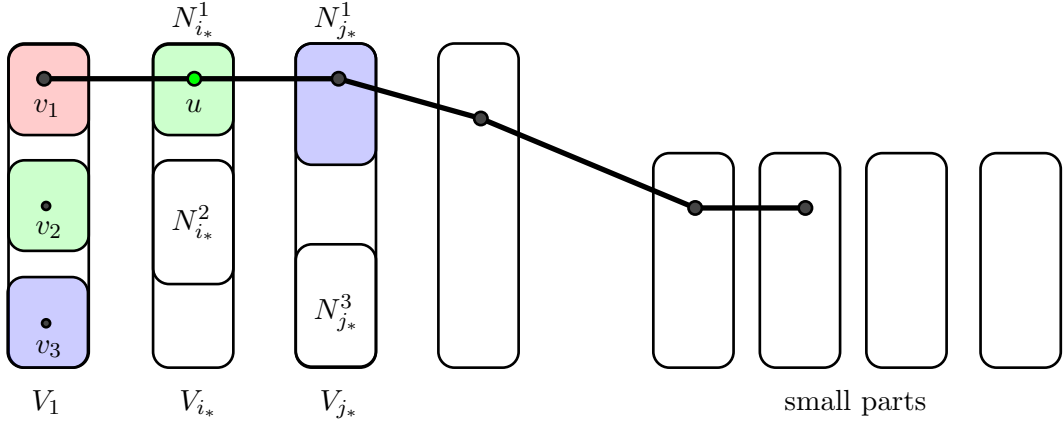
\begin{figure}[htbp]
\centering
\tikzset{every picture/.style={line width=1pt}} 

\begin{tikzpicture}[x=0.75pt,y=0.75pt,yscale=-1,xscale=1]

\draw   (97.24,30.29) .. controls (97.24,25.87) and (100.83,22.28) .. (105.25,22.28) -- (129.29,22.28) .. controls (133.71,22.28) and (137.3,25.87) .. (137.3,30.29) -- (137.3,177.31) .. controls (137.3,181.73) and (133.71,185.32) .. (129.29,185.32) -- (105.25,185.32) .. controls (100.83,185.32) and (97.24,181.73) .. (97.24,177.31) -- cycle ;
\draw   (240.61,30.29) .. controls (240.61,25.87) and (244.19,22.28) .. (248.62,22.28) -- (272.65,22.28) .. controls (277.08,22.28) and (280.67,25.87) .. (280.67,30.29) -- (280.67,177.31) .. controls (280.67,181.73) and (277.08,185.32) .. (272.65,185.32) -- (248.62,185.32) .. controls (244.19,185.32) and (240.61,181.73) .. (240.61,177.31) -- cycle ;
\draw   (169.4,30.29) .. controls (169.4,25.87) and (172.99,22.28) .. (177.41,22.28) -- (201.45,22.28) .. controls (205.87,22.28) and (209.46,25.87) .. (209.46,30.29) -- (209.46,177.31) .. controls (209.46,181.73) and (205.87,185.32) .. (201.45,185.32) -- (177.41,185.32) .. controls (172.99,185.32) and (169.4,181.73) .. (169.4,177.31) -- cycle ;
\draw   (419.1,85.44) .. controls (419.1,81.01) and (422.69,77.43) .. (427.11,77.43) -- (451.15,77.43) .. controls (455.57,77.43) and (459.16,81.01) .. (459.16,85.44) -- (459.16,177.31) .. controls (459.16,181.73) and (455.57,185.32) .. (451.15,185.32) -- (427.11,185.32) .. controls (422.69,185.32) and (419.1,181.73) .. (419.1,177.31) -- cycle ;
\draw   (472.27,85.44) .. controls (472.27,81.01) and (475.86,77.43) .. (480.28,77.43) -- (504.32,77.43) .. controls (508.74,77.43) and (512.33,81.01) .. (512.33,85.44) -- (512.33,177.31) .. controls (512.33,181.73) and (508.74,185.32) .. (504.32,185.32) -- (480.28,185.32) .. controls (475.86,185.32) and (472.27,181.73) .. (472.27,177.31) -- cycle ;
\draw   (525.44,85.44) .. controls (525.44,81.01) and (529.02,77.43) .. (533.45,77.43) -- (557.48,77.43) .. controls (561.91,77.43) and (565.5,81.01) .. (565.5,85.44) -- (565.5,177.31) .. controls (565.5,181.73) and (561.91,185.32) .. (557.48,185.32) -- (533.45,185.32) .. controls (529.02,185.32) and (525.44,181.73) .. (525.44,177.31) -- cycle ;
\draw   (311.81,30.29) .. controls (311.81,25.87) and (315.4,22.28) .. (319.83,22.28) -- (343.86,22.28) .. controls (348.29,22.28) and (351.87,25.87) .. (351.87,30.29) -- (351.87,177.31) .. controls (351.87,181.73) and (348.29,185.32) .. (343.86,185.32) -- (319.83,185.32) .. controls (315.4,185.32) and (311.81,181.73) .. (311.81,177.31) -- cycle ;
\draw  [fill=red, fill opacity=0.2] (97.24,30.78) .. controls (97.24,26.36) and (100.83,22.77) .. (105.25,22.77) -- (129.29,22.77) .. controls (133.71,22.77) and (137.3,26.36) .. (137.3,30.78) -- (137.3,60.32) .. controls (137.3,64.74) and (133.71,68.33) .. (129.29,68.33) -- (105.25,68.33) .. controls (100.83,68.33) and (97.24,64.74) .. (97.24,60.32) -- cycle ;
\draw  [fill=green, fill opacity=0.2] (97.24,89.03) .. controls (97.24,84.61) and (100.83,81.02) .. (105.25,81.02) -- (129.29,81.02) .. controls (133.71,81.02) and (137.3,84.61) .. (137.3,89.03) -- (137.3,118.57) .. controls (137.3,122.99) and (133.71,126.58) .. (129.29,126.58) -- (105.25,126.58) .. controls (100.83,126.58) and (97.24,122.99) .. (97.24,118.57) -- cycle ;
\draw  [fill=blue, fill opacity=0.2] (96.83,147.78) .. controls (96.83,143.35) and (100.42,139.77) .. (104.84,139.77) -- (128.88,139.77) .. controls (133.3,139.77) and (136.89,143.35) .. (136.89,147.78) -- (136.89,177.31) .. controls (136.89,181.73) and (133.3,185.32) .. (128.88,185.32) -- (104.84,185.32) .. controls (100.42,185.32) and (96.83,181.73) .. (96.83,177.31) -- cycle ;
\draw [fill=green, fill opacity=0.2]  (169.4,30.78) .. controls (169.4,26.36) and (172.99,22.77) .. (177.41,22.77) -- (201.45,22.77) .. controls (205.87,22.77) and (209.46,26.36) .. (209.46,30.78) -- (209.46,60.32) .. controls (209.46,64.74) and (205.87,68.33) .. (201.45,68.33) -- (177.41,68.33) .. controls (172.99,68.33) and (169.4,64.74) .. (169.4,60.32) -- cycle ;
\draw  [fill=blue, fill opacity=0.2] (240.61,30.29) .. controls (240.61,25.87) and (244.19,22.28) .. (248.62,22.28) -- (272.65,22.28) .. controls (277.08,22.28) and (280.67,25.87) .. (280.67,30.29) -- (280.67,75.27) .. controls (280.67,79.69) and (277.08,83.28) .. (272.65,83.28) -- (248.62,83.28) .. controls (244.19,83.28) and (240.61,79.69) .. (240.61,75.27) -- cycle ;
\draw   (169.4,89.03) .. controls (169.4,84.61) and (172.99,81.02) .. (177.41,81.02) -- (201.45,81.02) .. controls (205.87,81.02) and (209.46,84.61) .. (209.46,89.03) -- (209.46,135.27) .. controls (209.46,139.69) and (205.87,143.28) .. (201.45,143.28) -- (177.41,143.28) .. controls (172.99,143.28) and (169.4,139.69) .. (169.4,135.27) -- cycle ;
\draw   (240.61,131.29) .. controls (240.61,126.87) and (244.19,123.28) .. (248.62,123.28) -- (272.65,123.28) .. controls (277.08,123.28) and (280.67,126.87) .. (280.67,131.29) -- (280.67,176.82) .. controls (280.67,181.24) and (277.08,184.83) .. (272.65,184.83) -- (248.62,184.83) .. controls (244.19,184.83) and (240.61,181.24) .. (240.61,176.82) -- cycle ;
\draw   (582.44,85.44) .. controls (582.44,81.01) and (586.02,77.43) .. (590.45,77.43) -- (614.48,77.43) .. controls (618.91,77.43) and (622.5,81.01) .. (622.5,85.44) -- (622.5,177.31) .. controls (622.5,181.73) and (618.91,185.32) .. (614.48,185.32) -- (590.45,185.32) .. controls (586.02,185.32) and (582.44,181.73) .. (582.44,177.31) -- cycle ;
\draw  [line width=2pt]  (115,40) -- (190,40) -- (262,40) -- (333,60) -- (440,105) -- (495,105);
\draw [fill=uuuuuu] (115,40) circle (2.5pt);
\draw [fill=green] (190,40) circle (2.5pt);
\draw [fill=uuuuuu] (262,40) circle (2.5pt);
\draw [fill=uuuuuu] (333,60) circle (2.5pt);
\draw [fill=uuuuuu] (440,105) circle (2.5pt);
\draw [fill=uuuuuu] (495,105) circle (2.5pt);
%
\draw [fill=uuuuuu] (116,104) circle (1.5pt);
\draw [fill=uuuuuu] (116,163) circle (1.5pt);
\draw (108,49) node [anchor=north west][inner sep=0.75pt]   [align=left] {$v_1$};
\draw (184,49) node [anchor=north west][inner sep=0.75pt]   [align=left] {$u$};
\draw (109,110) node [anchor=north west][inner sep=0.75pt]   [align=left] {$v_2$};
\draw (109,170) node [anchor=north west][inner sep=0.75pt]   [align=left] {$v_3$};
\draw (177,0) node [anchor=north west][inner sep=0.75pt]   [align=left] {$N_{i_{\ast}}^1$};
\draw (248,0) node [anchor=north west][inner sep=0.75pt]   [align=left] {$N_{j_{\ast}}^1$};
\draw (177,102) node [anchor=north west][inner sep=0.75pt]   [align=left] {$N_{i_{\ast}}^2$};
\draw (248,147) node [anchor=north west][inner sep=0.75pt]   [align=left] {$N_{j_{\ast}}^3$};
\draw (107,195) node [anchor=north west][inner sep=0.75pt]   [align=left] {$V_1$};
\draw (181,195) node [anchor=north west][inner sep=0.75pt]   [align=left] {$V_{i_{\ast}}$};
\draw (251,195) node [anchor=north west][inner sep=0.75pt]   [align=left] {$V_{j_{\ast}}$};
\draw (483,195) node [anchor=north west][inner sep=0.75pt]   [align=left] {\text{small parts}};
%
\end{tikzpicture}
\caption{auxiliary figure for the proof of Claim~\ref{Claim:large-part}.}
\label{Fig:pf-of-large-parts}
\end{figure}
    \begin{claim}\label{Claim:large-part}
        We have $|\vartheta(V_i)|\le  2$ for every $i\in [q]$. 
    \end{claim}
    \begin{proof}[Proof of Claim~\ref{Claim:large-part}]
        Suppose to the contrary that there exists $i\in [q]$ such that $|\vartheta(V_i)|\ge 3$.
        By symmetry, we may assume that $i = 1$. 
        Let $v_1, v_2, v_3 \in V_1$ be three vertices such that $\vartheta(v_1), \vartheta(v_2), \vartheta(v_3)$ are pairwise distinct. 
        Let $N_{j}^{i} \coloneqq V_j \cap N_{\mathcal{H}}(v_{i})$ for $(i,j) \in [3] \times [k]$. 
        By Claim~\ref{CLAIM:k-large-induction}, there exists a unique pair $\{i_{\ast}, j_{\ast}\} \subseteq  [2,k]$ with $i_{\ast} \neq j_{\ast}$ such that $(\vartheta(N_{i_{\ast}}^{1}), \vartheta(N_{j_{\ast}}^{1})) = (\vartheta(v_2), \vartheta(v_3))$. In addition (by Claim~\ref{CLAIM:k-large-induction}~\ref{CLAIM:k-large-induction-2}), we have $N_{i_{\ast}}^{1} \cap N_{i_{\ast}}^2 = \emptyset$ and $N_{j_{\ast}}^{1} \cap N_{j_{\ast}}^3 = \emptyset$ (see Figure~\ref{Fig:pf-of-large-parts}).
        
        Fix a $(k-r)$-set $S \coloneqq \{i_1, \ldots, i_{k-r}\} \subseteq \overline{\{1,i_{\ast},j_{\ast}\}}$ such that 
        \begin{align*}
            \max\left\{|V_{j}| \colon j \in S \right\}
            \le \min\left\{|V_{j}| \colon j \in \overline{\{1,i_{\ast},j_{\ast}\}} \setminus S\right\}.
        \end{align*}
        Let $(m,\ell) \in \{(i_{\ast},1), (i_{\ast},2), (j_{\ast},1), (j_{\ast},3)\}$ be a member such that 
        \begin{align*}
            |N_{m}^{\ell}|
            = \min\left\{|N_{i_{\ast}}^{1}|, |N_{i_{\ast}}^{2}|, |N_{j_{\ast}}^{1}|, |N_{j_{\ast}}^{3}|\right\}. 
        \end{align*}
        In particular, $|N_{m}^{\ell}| \le (|V_{i_{\ast}}|+|V_{j_{\ast}}|)/4$. 
        
        Since there are $q \le r-2$ large parts and $V_1$ is large, each $V_{j}$ is small for $j \in S$.  
        By Proposition~\ref{PROP:Prelim-B}, there exists $e\in \mathcal{H}$ containing $v_{\ell}$ such that $\varphi(e) = \overline{S}$ (see Figure~\ref{Fig:Claim2-9} for the case $(\ell, m) = (1,i_{\ast})$). 
        Let $u$ denote the vertex in $e\cap V_{m}$, noting that $u \in N_{m}^{\ell}$.
        Observe that $N_{\mathcal{H}}(e \setminus \{u\}) \subseteq  N_{m}^{\ell} \cup V_{i_1} \cup \cdots \cup V_{i_{k-r}}$. 
        Hence, 
        \begin{align*}
            d_{\mathcal{H}}(e \setminus \{u\})
            \le |N_{m}^{\ell}| + |V_{i_1}| + \cdots + |V_{i_{k-r}}|
            \le \frac{|V_{i_{\ast}}|+|V_{j_{\ast}}|}{4} + |V_{i_1}| + \cdots + |V_{i_{k-r}}|.
        \end{align*}
        Let $I \coloneqq \left\{j \in \overline{\{1,i_{\ast},j_{\ast}\}} \colon V_j \text{ is small}\right\}$ and $p \coloneqq |I|$, noting that $p \ge k-2-q \ge k-r$ (recall that $V_1$ is large). 
        In addition, since $k\ge \frac{4r-2}{3}$, we have $p \le k-3 < 4k-4r$. 
        It follows from the definition of $S$ and a simple averaging argument that $|V_{i_1}| + \cdots + |V_{i_{k-r}}| \le \frac{k-r}{p} \cdot \sum_{j\in I}|V_j|$. 
        Therefore, the inequality above for $d_{\mathcal{H}}(e \setminus \{u\})$ continues as 
        \begin{align*}
            d_{\mathcal{H}}(e \setminus \{u\})
            \le \frac{|V_{i_{\ast}}|+|V_{j_{\ast}}|}{4} + \frac{k-r}{p} \cdot \sum_{j\in I}|V_j|, 
        \end{align*}
        which implies that $|V_{i_{\ast}}|+|V_{j_{\ast}}| + \frac{4(k-r)}{p} \sum_{j\in I}|V_j| \ge 4\cdot \delta_{r-1}^{+}(\mathcal{H}) > 4\cdot \frac{3k-3r+1}{3k-2} n$. 
        Consequently, 
        \begin{align*}
            n 
            & = \sum_{i\in [k]}|V_i| \\
            & = |V_1|
                +\left( |V_{i_{\ast}}|+|V_{j_{\ast}}| + \frac{4(k-r)}{p} \cdot \sum_{j\in I}|V_j| \right)
                - \frac{4(k-r)-p}{p} \cdot \sum_{j \in I} |V_j| + \sum_{j \in \overline{\{1,i_{\ast},j_{\ast}\}} \setminus I} |V_j| \\
            & > \frac{3 n}{3k-2} + 4\cdot \frac{3k-3r+1}{3k-2} n - \frac{4(k-r)-p}{p} \cdot p\cdot \frac{3 n}{3k-2} + \frac{3(k-3-p)}{3k-2}n 
            = n,
        \end{align*}
        a contradiction. 
    \end{proof}

    Claims~\ref{Claim:small-part} and~\ref{Claim:large-part}, together with Proposition~\ref{PROP:Final-compute}, imply that $\mathcal{H}$ is uniquely $k$-colorable, contradicting the assumption that $\mathcal{H}$ is not uniquely $k$-colorable. 
\end{proof}
Next, we prove Theorem~\ref{THM:main} for the case $k < \frac{4r-2}{3}$. The proof closely parallels that of Proposition~\ref{PROP:case-k-large}. 
\begin{proof}[Proof of Proposition~\ref{PROP:case-k-small}]
    By contradiction, let $r$ be the smallest integer for which Proposition~\ref{PROP:case-k-small} fails. 
    Note that the case $r=2$ is is vacuously true, as $2=r\le k<\frac{4r-2}{3}=2$. 
    So we may assume that $r \ge 3$. 
    Let $\mathcal{H}$ be an $n$-vertex $k$-partite $r$-graph without isolated vertices and with $\delta_{r-1}^{+}(\mathcal{H})>\frac{k-r+1}{k+2}n$ that is not uniquely $k$-colorable.

    Fix $\varphi \in \mathrm{Hom}(\mathcal{H},K_{k}^{r})$ and let $V_{i} \coloneqq \varphi^{-1}(i)$ for $i \in [k]$.
    Let $p \coloneqq |I_{\varphi}|$ and recall from Proposition~\ref{PROP:Prelim-D} that $p \ge r$. 
    By relabelling vertices in $K_{k}^{r}$ we may assume that $I_{\varphi} = [p]$, i.e. $V_1, \ldots, V_{p}$ are good sets.    
    
    Fix an arbitrary homomorphism $\vartheta \in \mathrm{Hom}(\mathcal{H}, K_{k}^{r})$.
    \begin{claim}\label{CLAIM:induction-b}
        For every $i \in [p]$ and $v \in V_i$, the following statements hold$\colon$ 
        \begin{enumerate}[label=(\roman*)]
            \item\label{CLAIM:induction-b-1} $|\vartheta(N_{\mathcal{H}}(v) \cap V_{j})|=1$ for every $j\in [k]\setminus \{i\}$, 
            \item\label{CLAIM:induction-b-2} $\vartheta(N_{\mathcal{H}}(v) \cap V_{j}) \neq \vartheta(v)$ for every $j\in [k]\setminus \{i\}$, and 
            \item\label{CLAIM:induction-b-3} $\vartheta(N_{\mathcal{H}}(v) \cap V_{j}) \neq \vartheta(N_{\mathcal{H}}(v) \cap V_{j'})$ for all distinct $j,j' \in [k]\setminus \{i\}$. 
        \end{enumerate}    
    \end{claim}
    \begin{proof}[Proof of Claim~\ref{CLAIM:induction-b}]
        The proof is similar to that of Claim~\ref{CLAIM:k-large-induction}. 
        Fix $i \in [q]$ and $v \in V_i$. 
        We may assume that $\vartheta(v)=i$.
        Let $N \coloneqq N_{\mathcal{H}}(v)$, noting that $N \subseteq  V \setminus V_i$. 
        Since $i \in [p]$, we have $|N| \le n - |V_i| \le n-\frac{n}{k+2}$. 
        Since both $k$ and $r$ are integers, the assumption $k < \frac{4r-2}{3}$ implies that $k \le \frac{4r-3}{3}$.
        Thus, we have $k-1 \le \frac{4(r-1)-2}{3}$. 
        
        We consider the link $L_{\mathcal{H}}(v)$ as an $(r-1)$-graph on $N$, and, in particular, $L_{\mathcal{H}}(v)$ contains no isolated vertices. 
        Notice that $L_{\mathcal{H}}(v)$ satisfies  
        \begin{align*}
            \delta_{r-2}^{+}(L_{\mathcal{H}}(v)) 
            \ge \delta_{r-1}^{+}(\mathcal{H})
            > \frac{k-r+1}{k+2}n 
             = \frac{k-r+1}{k+1} \left(n - \frac{n}{k+2} \right) 
             \ge \frac{k-r+1}{k+1}|N|.
        \end{align*}
        If $k-1 = \frac{4(r-1)-2}{3}$, then $\frac{3(k-1) - 3(r-1) +1}{3(k-1)-2} = \frac{(k-1)-(r-1)+1}{(k-1)+2}$, and it follows from Proposition~\ref{PROP:case-k-large} that $L_{\mathcal{H}}(v)$ is uniquely $(k-1)$-colorable. 
        If $k-1 < \frac{4(r-1)-2}{3}$, then it follows from the minimality of $r$ that $L_{\mathcal{H}}(v)$ is uniquely $(k-1)$-colorable.
        Thus, in both cases, $L_{\mathcal{H}}(v)$ is uniquely $(k-1)$-colorable. 
        Therefore, the induced maps $\left.\varphi\right|_{N}, \left.\vartheta\right|_{N}$ (viewed as homomorphisms from $L_{\mathcal{H}}(v)$ to $K_{k-1}^{r-1}$) are equivalent. 
        This implies Claim~\ref{CLAIM:induction-b}~\ref{CLAIM:induction-b-1},~\ref{CLAIM:induction-b-2}, and~\ref{CLAIM:induction-b-3}.
    \end{proof}
    Next, we show that $|\vartheta(V_i)|\le  2$ for every $i \in [k]$. 
    By Proposition~\ref{PROP:Final-compute}, this will complete the proof of Proposition~\ref{PROP:case-k-small}. 
    We will achieve this in two cases, addressed in Claims~\ref{Claim:small-part-b} and~\ref{Claim:large-part-b}, respectively. 
    \begin{claim}\label{Claim:small-part-b}
        We have $|\vartheta(V_i)|\le  2$ for every $i \in [p+1, k]$. 
    \end{claim}
    \begin{proof}[Proof of Claim~\ref{Claim:small-part-b}]
        The proof is similar to that of Claim~\ref{Claim:small-part}. 
        Fix $i \in [p+1, r]$ and $v \in V_i$. 
        By Proposition~\ref{PROP:Prelim-B}~\ref{PROP:Prelim-B-3}, there exists $e \in \mathcal{H}$ containing $v$ such that 
        \begin{align*}
            \min\left\{|[u]_{\varphi}| \colon u \in e \setminus \{v\}\right\}
            \ge 
            \max\left\{|V_j| \colon j \in \overline{\varphi(e)}\right\}.
        \end{align*}
        Since $|J_{\varphi}| \le r-2$ (by Proposition~\ref{PROP:Prelim-A}~\ref{PROP:Prelim-A-3}), we have $J_{\varphi} \subseteq  \varphi(e)$. 
        Thus, $\overline{\varphi(e)} \subseteq  \overline{J_{\varphi}}$.
        Let 
        \begin{align*}
            X 
            \coloneqq N_{\mathcal{H}}(e\setminus \{v\}) \cap V_i
            = N_{\mathcal{H}}(e\setminus \{v\}) \setminus \bigcup_{j \in \overline{\varphi(e)}} V_j. 
        \end{align*}
        Then simple calculations show that 
        \begin{align*}
            |X| 
            \ge d_{\mathcal{H}}(e\setminus \{v\}) - \sum_{j\in \overline{\varphi(e)}} |V_j|
            & > \frac{k-r+1}{k+2} n - (k-r) \cdot \frac{3n}{3k-2} \\
            & = \frac{1}{3} \cdot \frac{n}{k+2} + \frac{6\left(\frac{4r}{3}-\frac{2}{9} -k\right)}{(3k-2)(k+2)}n 
            > \frac{1}{3} \cdot \frac{n}{k+2}
            \ge \frac{|V_i|}{3}. 
        \end{align*}
        Fix $j \in \varphi(e) \cap I_\varphi$ and let $u$ denote the vertex in $V_j \cap e$. 
        The existence of such a $j$ is guaranteed by the definition of $e$; in fact, by Proposition~\ref{PROP:Prelim-D}, we have $\varphi(e) \setminus \{i\} \subseteq I_{\varphi}$. 
        Note that $X \subseteq  N_{\mathcal{H}}(u)$ and $v \in X$. 
        It follows from Claim~\ref{CLAIM:induction-b}~\ref{CLAIM:induction-b-1} that  $|\vartheta(X)|=1$. 
        In particular, $X \subseteq  [v]_{\vartheta}$. 
        Therefore, 
        \begin{align*}
            |[v]_{\vartheta} \cap V_{i}|
            \ge |X|
            > \frac{|V_{i}|}{3}. 
        \end{align*}
        Since $v$ was chosen arbitrarily, we conclude that $|\vartheta(V_i)| < \frac{|V_i|}{|V_i|/3} = 3$. 
    \end{proof}

    \begin{claim}\label{Claim:large-part-b}
        We have $|\vartheta(V_i)|\le  2$ for every $i \in [p]$. 
    \end{claim}
    \begin{proof}[Proof of Claim~\ref{Claim:large-part-b}]
        The proof is similar to that of Claim~\ref{Claim:large-part}. 
        Suppose to the contrary that there exists $i \in [p]$ such that $|\vartheta(V_i)|\ge 3$.  
        By symmetry, we may assume that $i = 1$. 
        Let $v_1, v_2, v_3 \in V_1$ be three vertices such that $\vartheta(v_1), \vartheta(v_2), \vartheta(v_3)$ are pairwise distinct.  
        Let $N_{j}^{i}\coloneqq V_{j} \cap N_{\mathcal{H}}(v_{i})$ for $(i,j) \in [3] \times [k]$. 
        By Claim~\ref{CLAIM:induction-b}, there exists a unique pair $\{i_{\ast}, j_{\ast}\} \subseteq  [2,k]$ with $s\neq t$ such that $(\vartheta(N_{i_{\ast}}^{1}), \vartheta(N_{j_{\ast}}^{1})) = (\vartheta(v_2), \vartheta(v_3))$. In addition (by Claim~\ref{CLAIM:induction-b}~\ref{CLAIM:induction-b-2}), we have $N_{i_{\ast}}^{1} \cap N_{i_{\ast}}^2 = \emptyset$ and $N_{j_{\ast}}^{1} \cap N_{j_{\ast}}^3 = \emptyset$ (see Figure~\ref{Fig:pf-of-large-parts}). 
        
        Let $(m,\ell) \in \{(i_{\ast},1), (i_{\ast},2), (j_{\ast},1), (j_{\ast},3)\}$ be a member such that 
        \begin{align*}
            |N_{m}^{\ell}|
            = \min\left\{|N_{i_{\ast}}^{1}|, |N_{i_{\ast}}^{2}|, |N_{j_{\ast}}^{1}|, |N_{j_{\ast}}^{3}|\right\}.
        \end{align*}
        In particular, $|N_{m}^{\ell}| \le (|V_{i_{\ast}}|+|V_{j_{\ast}}|)/4$. 
        
        By Proposition~\ref{PROP:Prelim-B}~\ref{PROP:Prelim-B-3}, there exists $e\in \mathcal{H}$ containing $v_{\ell}$ such that 
        \begin{align*}
            m \in \varphi(e)
            \quad\text{and}\quad 
            \min\left\{|\varphi^{-1}(j)| \colon j \in \varphi(e) \setminus \{1,m\}\right\}
            \ge 
            \max\left\{|\varphi^{-1}(j)| \colon j \in \overline{\varphi(e)} \right\}.
        \end{align*} 
        Assume that $\overline{\varphi(e)}=\{i_{1}, \ldots, i_{k-r}\}$. 
        Let $u$ denote the vertex in $e\cap V_{m}$, noting that $u \in N_{m}^{\ell}$.
        Observe that $N_{\mathcal{H}}(e \setminus \{u\}) \subseteq  N_{m}^{\ell} \cup V_{i_1} \cup \cdots \cup V_{i_{k-r}}$. 
        Therefore, 
        \begin{align*}
            d_{\mathcal{H}}(e \setminus \{u\})
            \le |N_{m}^{\ell}| + |V_{i_1}| + \cdots + |V_{i_{k-r}}|
            \le \frac{|V_{i_{\ast}}|+|V_{j_{\ast}}|}{4} + |V_{i_1}| + \cdots + |V_{i_{k-r}}|, 
        \end{align*}
        which implies that 
        \begin{align*}
            |V_{i_{\ast}}|+|V_{j_{\ast}}| + 4\left(|V_{i_1}| + \cdots + |V_{i_{k-r}}| \right)
            \ge 4 \cdot \delta^{+}_{r-1}(\mathcal{H})
            \ge \frac{4(k-r+1)}{k+2}n.
        \end{align*}
        Consequently, 
        \begin{align}\label{equ:Claim-case-small-n-lower}
            n 
             = \sum_{i\in [k]}|V_i| 
            & = |V_1| 
                + \left(|V_{i_{\ast}}|+|V_{j_{\ast}}| + 4 \sum_{j\in [k-r]}|V_{i_j}| \right) 
                - 4 \sum_{j\in [k-r]}|V_{i_j}|
                + \sum_{j\in \overline{\{1,i_{\ast},j_{\ast}\}}}|V_j|  \notag \\
            & > \frac{n}{k+2} + \frac{4(k-r+1)}{k+2}n 
            - 4 \sum_{j\in [k-r]}|V_{i_j}|
            + \sum_{j\in \overline{\{1,i_{\ast},j_{\ast}\}}}|V_j|.
        \end{align}
        Since both $k$ and $r$ are integers, the assumption $k < \frac{4r-2}{3}$ implies that $k \le \frac{4r-3}{3}$, which is equivalent to $4(k-r) \le k-3$. 
        Therefore, we can choose a set $S \subseteq  \overline{\{1,i_{\ast},j_{\ast}\}}$ of size $4k-4r$ such that 
        \begin{align*}
            \max\left\{|V_j| \colon j \in S\right\}
            \le \min\left\{|V_{j}| \colon j \in \overline{\{1,i_{\ast},j_{\ast}\}}\setminus S\right\}. 
        \end{align*} 
        Let $T \coloneqq \overline{\{1,i_{\ast},j_{\ast}\}}\setminus S$, noting that $|T| = 4r-3k-3 \le r-3 \le |I_{\varphi}| - 3$ (the second inequality is due to Proposition~\ref{PROP:Prelim-D}). 
        So, by the definition of $S$, we have $T \subseteq  I_{\varphi}$. Therefore, 
        \begin{align*}
            \sum_{j\in \overline{\{1,i_{\ast},j_{\ast}\}}}|V_j| - 4 \sum_{j\in [k-r]}|V_{i_j}|
            \ge \sum_{j\in \overline{\{1,i_{\ast},j_{\ast}\}}}|V_j| - \sum_{j\in S}|V_{j}|
            \ge \sum_{j\in T}|V_{j}|
            \ge \frac{4r-3k-3}{k+2}n. 
        \end{align*}
        Combining this with~\eqref{equ:Claim-case-small-n-lower}, we obtain  
        \begin{align*}
            n 
            > \frac{n}{k+2} + 4\cdot \frac{k-r+1}{k+2}n + \frac{4r-3k-3}{k+2}n
            = n, 
        \end{align*}
        a contradiction. 
    \end{proof}

    Claims~\ref{Claim:small-part-b} and~\ref{Claim:large-part-b}, together with Proposition~\ref{PROP:Final-compute}, imply that $\mathcal{H}$ is uniquely $k$-colorable, contradicting the assumption that $\mathcal{H}$ is not uniquely $k$-colorable. 
\end{proof}

\subsection{Proof of Proposition~\ref{PROP:case-all-small}}\label{SEC:proof-Prop-small}
We prove Proposition~\ref{PROP:case-all-small} in this subsection. 
The following structure will be crucial for the proof.

Let $m \ge r \ge 2$ be integers and $q \coloneqq \left\lfloor \frac{m-1}{r-1}\right\rfloor$. 
Given an ordered vertex set $(v_1, \ldots, v_m)$, the $r$-uniform \textbf{quasi-sunflower} on $(v_1, \ldots, v_m)$, denoted by $\mathcal{S}^{r}(v_1, \ldots, v_m)$, is the $r$-graph with edge set 
    \begin{align*}
        \left\{ \{v_1,\cdots,v_{r-1},v_m\}, \cdots, \{v_{(q-1)(r-1)+1}, \cdots, v_{q(r-1)},v_m\},\{v_{m-r+1}, \cdots, v_m\} \right\}.
    \end{align*}

\begin{figure}[htbp]
\centering

\tikzset{every picture/.style={line width=1pt}} 

\begin{tikzpicture}[x=0.75pt,y=0.75pt,yscale=-1,xscale=1,line cap=round,line join=round]

\draw    (133.34,62.86) -- (133.34,107.99) ;
\draw    (133.34,126.57) -- (133.34,171.7) ;
\draw    (78.23,217.71) -- (133.34,171.7) ;
\draw    (78.23,217.71) -- (133.34,126.57) ;
\draw    (78.23,217.71) -- (133.34,107.99) ;
\draw    (78.23,217.71) -- (133.34,62.86) ;
\draw [line width=1pt, fill=sqsqsq, fill opacity=0.1] (78.23,217.71) -- (133.34,62.86) -- (133.34,107.99) -- cycle;
\draw [line width=1pt, fill=sqsqsq, fill opacity=0.1] (78.23,217.71) -- (133.34,126.57) -- (133.34,171.7) -- cycle;
\draw [line width=1.3pt, color = red, fill = red] (133.34,62.86) -- (133.34,107.99);
\draw [line width=1.3pt, color = red, fill = red] (133.34,126.57) -- (133.34,171.7);
\draw [line width=1.3pt, color = red, fill = red] (78.23,217.71) -- (133.34,171.7);
\draw [fill=cyan] (78.23,217.71) circle (2.5pt);
\draw [fill=uuuuuu] (133.34,62.86) circle (1.5pt);
\draw [fill=uuuuuu] (133.34,107.99) circle (1.5pt);
\draw [fill=uuuuuu] (133.34,126.57) circle (1.5pt);
\draw [fill=uuuuuu] (133.34,171.7) circle (1.5pt);
%
\draw    (221.7,218.6) -- (276.8,172.59) ;
\draw [dashed, opacity=0.5]   (221.7,218.6) -- (276.8,127.46) ;
\draw  [dashed, opacity=0.5]  (221.7,218.6) -- (276.8,108.88) ;
\draw  [dashed, opacity=0.5]  (221.7,218.6) -- (276.8,63.75) ;
\draw    (276.8,63.75) -- (374.66,186.74) ;
\draw    (276.8,108.88) -- (374.66,186.74) ;
\draw    (276.8,127.46) -- (374.66,186.74) ;
\draw    (276.8,172.59) -- (374.66,186.74) ;
\draw [line width=1pt, fill=sqsqsq, fill opacity=0.1] (276.8,63.75) -- (276.8,108.88) -- (374.66,186.74) -- cycle;
\draw [line width=1pt, fill=sqsqsq, fill opacity=0.1] (276.8,127.46) -- (276.8,172.59) -- (374.66,186.74) -- cycle;
\draw [line width=1pt, fill=sqsqsq, fill opacity=0.1] (221.7,218.6) -- (276.8,172.59) -- (374.66,186.74) -- cycle;
\draw [line width=1.3pt, color = red, fill = red]   (276.8,63.75) -- (276.8,108.88) ;
\draw  [line width=1.3pt, color = red, fill = red]  (276.8,127.46) -- (276.8,172.59) ;
\draw  [line width=1.3pt, color = red, fill = red]  (221.7,218.6) -- (374.66,186.74) ;
\draw [fill=uuuuuu] (276.8,63.75) circle (1.5pt);
\draw [fill=uuuuuu] (276.8,108.88) circle (1.5pt);
\draw [fill=uuuuuu] (276.8,127.46) circle (1.5pt);
\draw [fill=uuuuuu] (276.8,172.59) circle (1.5pt);
\draw [fill=uuuuuu] (221.7,218.6) circle (1.5pt);
\draw [fill=cyan] (374.66,186.74) circle (2.5pt);
\draw  [dashed, opacity=0.5]  (428.81,219.48) -- (483.92,173.47) ;
\draw  [dashed, opacity=0.5]  (428.81,219.48) -- (483.92,128.34) ;
\draw  [dashed, opacity=0.5]  (428.81,219.48) -- (483.92,109.76) ;
\draw  [dashed, opacity=0.5]  (428.81,219.48) -- (483.92,64.63) ;
\draw  [dashed, opacity=0.5]  (483.92,64.63) -- (581.77,187.63) ;
\draw  [dashed, opacity=0.5]  (483.92,109.76) -- (581.77,187.63) ;
\draw  [dashed, opacity=0.5]  (483.92,128.34) -- (581.77,187.63) ;
\draw  [dashed, opacity=0.5]  (483.92,173.47) -- (581.77,187.63) ;
\draw    (483.92,64.63) -- (587.48,102.68) ;
\draw    (483.92,109.76) -- (587.48,102.68) ;
\draw    (483.92,128.34) -- (587.48,102.68) ;
\draw    (483.92,173.47) -- (587.48,102.68) ;
\draw    (428.81,219.48) .. controls (499.12,192.05) and (573.22,138.08) .. (587.48,102.68) ;
\draw [line width=1pt, fill=sqsqsq, fill opacity=0.1]  (587.48,102.68) -- (483.92,128.34) -- (483.92,173.47) -- cycle;
\draw [line width=1pt, fill=sqsqsq, fill opacity=0.1]  (587.48,102.68) -- (483.92,64.63) -- (483.92,109.76) -- cycle;
\draw [line width=1pt, fill=sqsqsq, fill opacity=0.1] (581.77,187.63) -- (428.81,219.48) .. controls (499.12,192.05) and (573.22,138.08) .. (587.48,102.68) -- (581.77,187.63);
\draw  [line width=1.3pt, color = red, fill = red]  (483.92,64.63) -- (483.92,109.76) ;
\draw  [line width=1.3pt, color = red, fill = red]  (483.92,128.34) -- (483.92,173.47) ;
\draw  [line width=1.3pt, color = red, fill = red]  (587.48,102.68) -- (581.77,187.63) ;
\draw  [line width=1.3pt, color = red, fill = red]  (428.81,219.48) -- (581.77,187.63) ;
\draw [fill=cyan] (587.48,102.68) circle (2.5pt);
\draw [fill=uuuuuu] (483.92,128.34) circle (1.5pt);
\draw [fill=uuuuuu] (483.92,173.47) circle (1.5pt);
\draw [fill=uuuuuu] (483.92,64.63) circle (1.5pt);
\draw [fill=uuuuuu] (483.92,109.76) circle (1.5pt);
\draw [fill=uuuuuu] (428.81,219.48) circle (1.5pt);
\draw [fill=uuuuuu] (581.77,187.63) circle (1.5pt);
\draw (140,55) node [anchor=north west][inner sep=0.75pt]   [align=left] {$1$};
\draw (140,100) node [anchor=north west][inner sep=0.75pt]   [align=left] {$2$};
\draw (140,125) node [anchor=north west][inner sep=0.75pt]   [align=left] {$3$};
\draw (140,165) node [anchor=north west][inner sep=0.75pt]   [align=left] {$4$};
\draw (60,210) node [anchor=north west][inner sep=0.75pt]   [align=left] {$5$};
\draw (380,180) node [anchor=north west][inner sep=0.75pt]   [align=left] {$6$};
\draw (592,95) node [anchor=north west][inner sep=0.75pt]   [align=left] {$7$};
\draw (588,180) node [anchor=north west][inner sep=0.75pt]   [align=left] {$6$};
\draw (412,210) node [anchor=north west][inner sep=0.75pt]   [align=left] {$5$};
\draw (205,210) node [anchor=north west][inner sep=0.75pt]   [align=left] {$5$};
\draw (259,55) node [anchor=north west][inner sep=0.75pt]   [align=left] {$1$};
\draw (259,100) node [anchor=north west][inner sep=0.75pt]   [align=left] {$2$};
\draw (259,125) node [anchor=north west][inner sep=0.75pt]   [align=left] {$3$};
\draw (259,165) node [anchor=north west][inner sep=0.75pt]   [align=left] {$4$};
\draw (466,55) node [anchor=north west][inner sep=0.75pt]   [align=left] {$1$};
\draw (466,100) node [anchor=north west][inner sep=0.75pt]   [align=left] {$2$};
\draw (466,125) node [anchor=north west][inner sep=0.75pt]   [align=left] {$3$};
\draw (466,165) node [anchor=north west][inner sep=0.75pt]   [align=left] {$4$};
\end{tikzpicture}
\caption{$\mathcal{S}^3(1,\ldots,5) \to \mathcal{S}^3(1,\ldots,6) \to \mathcal{S}^3(1,\ldots,7)$.} 
\label{Fig:generalized-sunflower}
\end{figure}
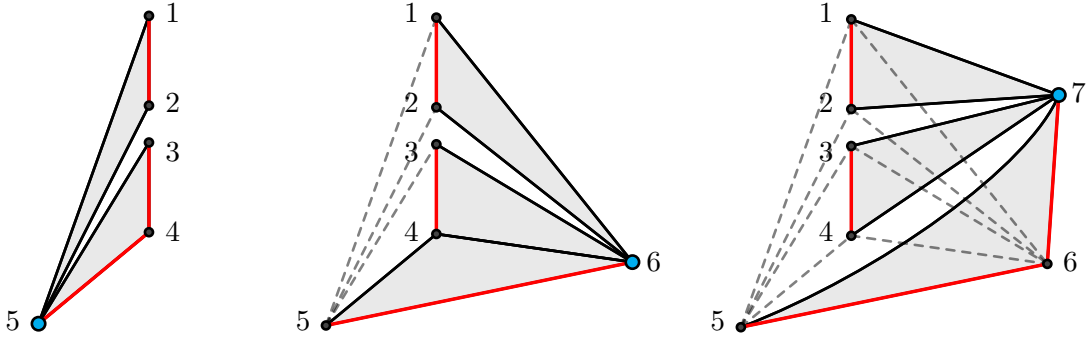

The following fact about quasi-sunflowers can be derived through a simple inductive argument (see Figure~\ref{Fig:generalized-sunflower}). 
\begin{fact}\label{FACT:quasi-sunflower-2-covered}
    Let $k \ge m \ge r \ge 2$ be integers and $V\coloneqq (v_1, \ldots, v_{m})$ be an ordered vertex set. 
    Suppose that $\mathcal{H}$ is an $r$-graph on $V$ such that $\mathcal{S}^{r}(v_1, \ldots, v_j) \subseteq \mathcal{H}$ for every $j \in [r, m]$. 
    Then every pair of vertices in $V$ is contained in some edge of $\mathcal{H}$. 
    In particular, $\varphi(v_1), \ldots, \varphi(v_m)$ are pairwise distinct for every homomorphism $\varphi \in \mathrm{Hom}(\mathcal{H}, K_{k}^r)$. 
\end{fact}
Now we are ready to prove Proposition~\ref{PROP:case-all-small}. 
\begin{proof}[Proof of Proposition~\ref{PROP:case-all-small}]
    Let $n \ge k \ge r \ge 2$ be integers. 
    Let $q$ and $s$ be integers satisfying $k-1 = q(r-1)+s$ and $0 \le s \le r-2$. 
    Assume that every homomorphism in $\mathrm{Hom}(\mathcal{H},K_{k}^{r})$ is small. 
    Fix $\vartheta, \varphi \in \mathrm{Hom}(\mathcal{H},K_{k}^{r})$ and let $V_{i}\coloneqq \varphi^{-1}(i)$ for $i\in [k]$.  
    Note from the assumption that $J_{\varphi} = J_{\vartheta} = \emptyset$. 
    The key step in this proof is to show that $|\vartheta(\bigcup_{i\in I}V_{i})|\le  r$ for every $(s+1)$-set $I \subseteq [k]$ (i.e.  Claim~\ref{Claim:colors-in-V-I}).
    \begin{claim}\label{CLAIM:find-quasi-sunflower}
        Let $I \subseteq [k]$ be a set of size $s+1$. 
        For every $i \in I$ and $v \in  V_i$, there exists an ordered vertex set $(v_{1},  \cdots, v_{k-s}) \subseteq \bigcup_{j\in \overline{I\setminus\{i\}}} V_j$ with $v_1 = v$  such that  
        \begin{align*}
            S^{r}(v_1,\cdots, v_{\ell}) \subseteq  \mathcal{H}
            \quad\text{for every}\quad 
            \ell \in [r, k-s]. 
        \end{align*}
        In particular, by Fact~\ref{FACT:quasi-sunflower-2-covered}, $\vartheta(v), \vartheta(v_2), \ldots, \vartheta(v_{k-s})$ are pairwise distinct.
    \end{claim}
    \begin{proof}[Proof of Claim~\ref{CLAIM:find-quasi-sunflower}]
        Fix $I \in \binom{[k]}{s+1}$, $i \in I$, and $v \in  V_i$. 
        We will find inductively the ordered vertex set $(v_{1},  \cdots, v_{k-s}) \subset \bigcup_{j\in \overline{I\setminus\{i\}}} V_j$  with $v_1 = v$ such that $S^{r}(v_1,\cdots, v_{\ell}) \subseteq  \mathcal{H}$ for every $\ell \in [r, k-s]$.
            
        The base case $\ell = r$ is guaranteed by 
        Proposition~\ref{PROP:Prelim-B}~\ref{PROP:Prelim-B-4}, so it suffices to focus on the inductive step. 
        Suppose that we have found an ordered vertex set $(v_{1},\dots,v_{\ell}) \in \bigcup_{j\in \overline{I\setminus\{i\}}} V_j$ with $v_1 = v$ for some $\ell \in [r,k-s-1]$ such that $\mathcal{S}^{r}(v_{1},\dots,v_{j}) \subseteq  \mathcal{H}$ for every $j \in [r,\ell]$. 
        Let $t \coloneqq \left\lfloor \frac{\ell-1}{r-1}\right\rfloor\le  q-1$.
        Let 
        \begin{align*}
            e_j 
            \coloneqq \left\{v_{(r-1)j+1}, \ldots, v_{(r-1)(j+1)} \right\}
            \text{~for~} j \in [0, t-1],
            \quad\text{and}\quad
            e_t 
            \coloneqq \left\{v_{\ell-r+2}, \ldots, v_{\ell} \right\}. 
        \end{align*}
        Since $\mathcal{S}^{r}(v_1, \ldots, v_{\ell}) \subseteq \mathcal{H}$, we have $e_j \in  \partial\mathcal{H}$ for every $j \in [0,t]$ (see the red pairs in Figure~\ref{Fig:generalized-sunflower} for the case $r=3$). 
        Let 
        \begin{align*}
            U 
            \coloneqq N_{\mathcal{H}}(e_0) \cap \cdots \cap N_{\mathcal{H}}(e_{t}).  
        \end{align*} 
        Consider the number of pairs $(e_j,u)$ with $j\in [t]$ and $u \in V(\mathcal{H})$ such that $e_j\cup \{u\} \in \mathcal{H}$.
        For each $e_j$, by the assumption on $\delta^{+}_{r-1}(\mathcal{H})$, the number of possible choices for $u$ is greater than $\frac{3k-3r+1}{3k-2} n$. 
        On the other hand, for each $u \in V(\mathcal{H}) \setminus U$, there are at most $t$ choices for $e_j$; and for each $u \in U$, there are at most $t+1$ choices for $e_j$. 
        Therefore, we obtain 
        \begin{align*}
            |U| \cdot (t+1) + \left(n-|U|\right) \cdot t 
            > (t+1) \cdot \frac{3k-3r+1}{3k-2} n.
        \end{align*}
        It follows that
        \begin{align}\label{equ:quasi-sunflower-U-lower}
            |U|
            > \frac{3k-3r+1}{3k-2} n - t \cdot \frac{3 (r-1)}{3 k-2} n
            \ge \frac{3k-3r+1}{3k-2} n - (q-1) \cdot \frac{3 (r-1)}{3 k-2} n 
            = \frac{3s+1}{3k-2} n,
        \end{align}
        which is greater than $s\cdot\frac{3n}{3k-2}\ge \sum_{j\in I\setminus\{i\}}|V_j|$.
        Therefore, $U \setminus \bigcup_{j\in I\setminus\{i\}} V_j \neq \emptyset$. 
        
        Fix a vertex $v_{\ell+1} \in U\setminus\bigcup_{j\in I\setminus\{i\}}V_{j}$. 
        It follows from the definition of $U$ that the set $\{e_j \cup \{v_{\ell+1}\} \colon j \in [0,t]\} \subseteq \mathcal{H}$ forms the quasi-sunflower $\mathcal{S}^r(v_1, \ldots, v_{\ell+1})$ (see Figure~\ref{Fig:generalized-sunflower}). 
        This concludes the proof of the inductive step, and hence, the proof of Claim~\ref{CLAIM:find-quasi-sunflower}. 
    \end{proof}
    %
\begin{figure}[htbp]
\centering

\tikzset{every picture/.style={line width=1pt}} 

\begin{tikzpicture}[x=0.75pt,y=0.75pt,yscale=-1,xscale=1, line cap=round,line join=round]
%
\draw [line width=1pt, fill=sqsqsq, fill opacity=0.1] (157.92,52.63) -- (157.92,97.76) -- (261.48,90.68) -- cycle; 
\draw [line width=1pt, fill=sqsqsq, fill opacity=0.1] (157.92,116.34) -- (157.92,161.47) -- (261.48,90.68) -- cycle; 
\draw [line width=1pt, fill=sqsqsq, fill opacity=0.1]  (102.81,207.48) .. controls (173.12,180.05) and (247.22,126.08) .. (261.48,90.68) -- (255.77,175.63) -- (102.81,207.48); 
\draw    (157.92,52.63) -- (157.92,97.76) ;
\draw  [line width = 1.3pt, color = yellow, fill = yellow]  (157.92,116.34) -- (157.92,161.47) ;
\draw  [line width = 1.3pt, color = yellow, fill = yellow]  (102.81,207.48) -- (255.77,175.63) ;
\draw    (157.92,52.63) -- (261.48,90.68) ;
\draw  [line width = 1.3pt, color = yellow , fill = yellow]  (157.92,97.76) -- (261.48,90.68) ;
\draw    (157.92,116.34) -- (261.48,90.68) ;
\draw    (157.92,161.47) -- (261.48,90.68) ;
\draw    (102.81,207.48) .. controls (173.12,180.05) and (247.22,126.08) .. (261.48,90.68) ;
\draw   (261.48,90.68) -- (255.77,175.63) ;
\draw [fill=cyan] (261.48,90.68) circle (2.5pt);
\draw [fill=uuuuuu] (157.92,97.76) circle (1.5pt);
\draw [fill=uuuuuu] (157.92,52.63) circle (1.5pt);
\draw [fill=uuuuuu] (157.92,116.34) circle (1.5pt);
\draw [fill=uuuuuu] (157.92,161.47) circle (1.5pt);
\draw [fill=uuuuuu] (102.81,207.48) circle (1.5pt);
\draw [fill=uuuuuu] (255.77,175.63) circle (1.5pt);
\draw  [dashed, opacity=0.5]  (442.92,53.63) -- (442.92,98.76) ;
\draw    (442.92,117.34) -- (442.92,162.47) ;
\draw   (387.81,208.48) -- (540.77,176.63) ;
\draw [dashed, opacity=0.5]   (442.92,53.63) -- (546.48,91.68) ;
\draw    (442.92,98.76) -- (546.48,91.68) ;
\draw  [dashed, opacity=0.5]  (442.92,117.34) -- (546.48,91.68) ;
\draw  [dashed, opacity=0.5]  (442.92,162.47) -- (546.48,91.68) ;
\draw  [dashed, opacity=0.5]  (387.81,208.48) .. controls (458.12,181.05) and (532.22,127.08) .. (546.48,91.68) ;
\draw  [dashed, opacity=0.5]  (546.48,91.68) -- (540.77,176.63) ;
\draw [line width=1pt, fill=sqsqsq, fill opacity=0.1] (337.19,60.28) .. controls (369.39,74.56) and (500.19,84.28) .. (546.48,91.68) -- (442.92,98.76) .. controls (413.19,85.28) and (369.39,74.56) .. (337.19,60.28); 
\draw [line width=1pt, fill=sqsqsq, fill opacity=0.1] (337.19,60.28) .. controls (378.19,88.28) and (400.19,101.28) .. (442.92,117.34) -- (442.92,162.47) .. controls (415.19,122.28) and (381.19,100.28) .. (337.19,60.28); 
\draw [line width=1pt, fill=sqsqsq, fill opacity=0.1] (337.19,60.28) .. controls (391.19,154.28) and (422.19,185.28) .. (540.77,176.63) -- (387.81,208.48) .. controls (387.19,179.28) and (358.19,111.28) .. (337.19,60.28); 

\draw    (337.19,60.28) .. controls (369.39,74.56) and (500.19,84.28) .. (546.48,91.68) ;
\draw    (337.19,60.28) .. controls (369.39,74.56) and (413.19,85.28) .. (442.92,98.76) ;
\draw    (337.19,60.28) .. controls (378.19,88.28) and (400.19,101.28) .. (442.92,117.34) ;
\draw    (337.19,60.28) .. controls (381.19,100.28) and (415.19,122.28) .. (442.92,162.47) ;
\draw    (337.19,60.28) .. controls (391.19,154.28) and (422.19,185.28) .. (540.77,176.63) ;
\draw    (337.19,60.28) .. controls (358.19,111.28) and (387.19,179.28) .. (387.81,208.48) ;
\draw [fill=cyan] (337.19,60.28) circle (2.5pt);
\draw [fill=uuuuuu] (442.92,53.63) circle (1.5pt);
\draw [fill=uuuuuu] (546.48,91.68) circle (1.5pt);
\draw [fill=uuuuuu] (442.92,98.76) circle (1.5pt);
\draw [fill=uuuuuu] (442.92,117.34) circle (1.5pt);
\draw [fill=uuuuuu] (442.92,162.47) circle (1.5pt);
\draw [fill=uuuuuu] (540.77,176.63) circle (1.5pt);
\draw [fill=uuuuuu] (387.81,208.48) circle (1.5pt);
\draw (141,47) node [anchor=north west][inner sep=0.75pt]   [align=left] {$1$};
\draw (141,90) node [anchor=north west][inner sep=0.75pt]   [align=left] {$2$};
\draw (141,112) node [anchor=north west][inner sep=0.75pt]   [align=left] {$3$};
\draw (141,153) node [anchor=north west][inner sep=0.75pt]   [align=left] {$4$};
\draw (87,202) node [anchor=north west][inner sep=0.75pt]   [align=left] {$5$};
\draw (261,172) node [anchor=north west][inner sep=0.75pt]   [align=left] {$6$};
\draw (268,85) node [anchor=north west][inner sep=0.75pt]   [align=left] {$7$};
\draw (427,47) node [anchor=north west][inner sep=0.75pt]   [align=left] {$1$};
\draw (427,97) node [anchor=north west][inner sep=0.75pt]   [align=left] {$2$};
\draw (448,116) node [anchor=north west][inner sep=0.75pt]   [align=left] {$3$};
\draw (448,156) node [anchor=north west][inner sep=0.75pt]   [align=left] {$4$};
\draw (371,202) node [anchor=north west][inner sep=0.75pt]   [align=left] {$5$};
\draw (545,172) node [anchor=north west][inner sep=0.75pt]   [align=left] {$6$};
\draw (550,85) node [anchor=north west][inner sep=0.75pt]   [align=left] {$7$};
\draw (330,43) node [anchor=north west][inner sep=0.75pt]   [align=left] {$u$};
\end{tikzpicture}
\caption{$\mathcal{S}^3(1,\ldots,7) \to \mathcal{S}^3(3,4,5,6,2,7,u)$.} 
\label{Fig:generalized-sunflower-2}
\end{figure}
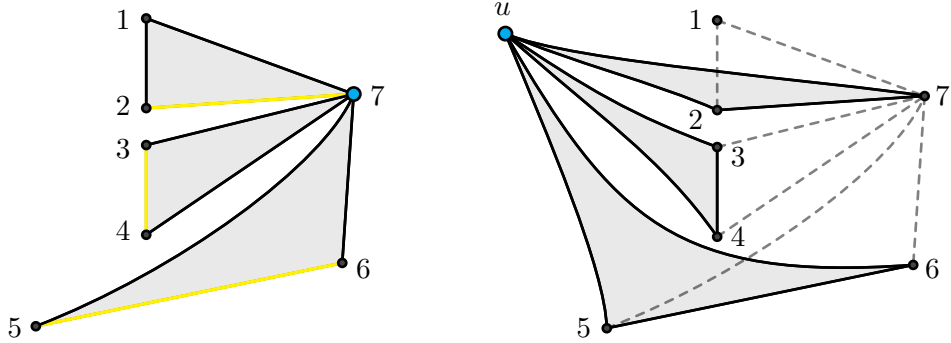
    \begin{claim}\label{CLAIM:quasi-sunflower-M}
        Let $I \subseteq [k]$ be a set of size $s+1$. 
        For every $v \in  \bigcup_{j \in I} V_j$, we have 
        \begin{align}\label{equ:vertices-with-same-color}
            \left|[v]_{\vartheta} \cap \left(\bigcup_{j\in I}V_{j}\right) \right|
                >\frac{n}{3k-2}. 
        \end{align}
    \end{claim}
    \begin{proof}[Proof of Claim~\ref{CLAIM:quasi-sunflower-M}]
        Fix $I \in \binom{[k]}{s+1}$, $i \in I$, and $v \in  V_i$. 
        Let $(v_{1},  \cdots, v_{k-s}) \subseteq \bigcup_{j\in \overline{I\setminus\{i\}}} V_j$ be an ordered vertex set, with $v_{1} = v$, as guaranteed by Claim~\ref{CLAIM:find-quasi-sunflower}. 
        %
        Let 
        \begin{align*}
            e'_j \coloneqq \{v_{(r-1)j+1}, \ldots, v_{(r-1)(j+1)}\}
            \text{ for } j \in [q-1]
            \quad\text{and}\quad 
            e'_q \coloneqq \{v_2, v_3, \ldots, v_{r-1}, v_{k-s}\}. 
        \end{align*}
        Since $\mathcal{S}^{r}(v_{1},  \cdots, v_{k-s}) \subseteq \mathcal{H}$, we have $e'_j \in \partial\mathcal{H}$ for $j \in [q]$ (see the yellow pairs in Figure~\ref{Fig:generalized-sunflower-2} for the case $r=3$). 
        Let 
        \begin{align*}
            M
            \coloneqq 
            N_{\mathcal{H}}(e'_1) \cap \cdots \cap N_{\mathcal{H}}(e'_q).
        \end{align*}
        It follows from the definition of $M$ that for every vertex $u \in M$, the set $\{e'_j \cup \{u\} \colon j \in [q]\} \subseteq \mathcal{H}$ forms the quasi-sunflower $\mathcal{S}^{r}(v_2, \ldots, v_{k-s}, u)$ (see Figure~\ref{Fig:generalized-sunflower-2}).  
        Therefore, similar to Claim~\ref{CLAIM:find-quasi-sunflower}, we obtain 
        \begin{align*}
            \vartheta(u) 
            \in [k] \setminus \left\{\vartheta(v_2), \ldots, \vartheta(v_{k-s}) \right\}
            \quad\text{for every}\quad 
            u \in M. 
        \end{align*}
        In particular, $\vartheta(M) \subseteq  [k]\setminus\{\vartheta(v_2), \ldots, \vartheta(v_{k-s})\}$ has size at most $s+1$. 
        
        A proof similar to that of~\eqref{equ:quasi-sunflower-U-lower} yields 
        \begin{align}\label{equ:|M|}
            |M|
            >\frac{3s+1}{3k-2} n.
        \end{align}
        On the other hand, since $\vartheta$ is small, we have 
        \begin{align}\label{equ:quasi-sunflower-M-colors}
            |\vartheta^{-1}(j) \cap M| 
            \le |\vartheta^{-1}(j)| \le \frac{3n}{3k-2}
            \quad\text{for every}\quad
            j \in \vartheta(M). 
        \end{align}
        Combining this with~\eqref{equ:|M|}, we obtain $\vartheta(M) \ge s+1$, and hence, $\vartheta(M) = s+1$. 
        This implies that $\vartheta(M) =  [k]\setminus\{\vartheta(v_2), \ldots, \vartheta(v_{k-s})\}$.
        Since $\vartheta(v_1) \in [k]\setminus\{\vartheta(v_2), \ldots, \vartheta(v_{k-s})\} = \vartheta(M)$, it follows from~\eqref{equ:quasi-sunflower-M-colors} that 
        \begin{align*}
            |[v_1]_{\vartheta} \cap M| 
            \ge |M| - s \cdot  \frac{3n}{3k-2}
            > \frac{n}{3k-2}.
        \end{align*}
        Recall that $v_1 = v$ and $M \subseteq  \bigcup_{j\in I}V_j$. 
        Therefore, the inequality above implies~\eqref{equ:vertices-with-same-color}. 
    \end{proof}
    \begin{claim}\label{Claim:colors-in-V-I}
        We have $|\vartheta(\bigcup_{j\in I}V_j)| \le r$ for every $(s+1)$-set $I \subseteq [k]$.
        In particular, $\left| \vartheta(V_{i})\right| \le  r$ for every $i\in [k]$.  
    \end{claim}
    \begin{proof}[Proof of Claim~\ref{Claim:colors-in-V-I}]
        Suppose to the contrary that $|\vartheta(\bigcup_{j\in I}V_j)| \ge r+1 \ge s+3$ for some $(s+1)$-set $I \subseteq [k]$. 
        Fix $i \in I$ and $v \in  V_i$. 
        Let $(v_{1},  \cdots, v_{k-s}) \subseteq \bigcup_{j\in \overline{I\setminus\{i\}}} V_j$ be an ordered vertex set, with $v_{1} = v$, as guaranteed by Claim~\ref{CLAIM:find-quasi-sunflower}.
        Let $M$ be be defined as in the proof of Claim~\ref{CLAIM:quasi-sunflower-M}. 
        Recall that $\vartheta(M) =  [k]\setminus\{\vartheta(v_2), \ldots, \vartheta(v_{k-s})\}$ has size $s+1$. 
        So, there exist two vertices $w_1, w_2 \in \bigcup_{j\in I}V_j$ with $\vartheta(w_1) \neq \vartheta(w_2)$ such that $\{\vartheta(w_1), \vartheta(w_2)\} \subseteq  \{\vartheta(v_2), \ldots, \vartheta(v_{k-s})\}$.
        In particular, $M \cap [w_1]_{\vartheta} = M \cap [w_2]_{\vartheta} = \emptyset$. 
        So, it follows from~\eqref{equ:vertices-with-same-color} and~\eqref{equ:|M|} that 
        \begin{align*}
            \left| \bigcup_{j\in I}V_{j} \right|
            \ge \left| [w_1]_{\vartheta} \cap \bigcup_{j\in I}V_{j} \right|
                +\left| [w_2]_{\vartheta} \cap \bigcup_{j\in I}V_{j} \right|+|M|
            >\frac{3(s+1)}{3k-2}n.  
        \end{align*}
        However, since $\vartheta$ is small, we have $\left| \bigcup_{j\in I}V_{j} \right| \le (s+1) \cdot \frac{3n}{3k-2}$, a contradiction. 
    \end{proof}
    By Proposition~\ref{PROP:case-all-small}, to finish the proof, it suffices to prove that $\left| \vartheta(V_{i})\right| \le  2$ for every $i\in [k]$. 
    Suppose to the contrary that $\left| \vartheta(V_{i})\right| = p \ge 3$ for some $i \in [k]$.
    Then, by averaging, there exists a vertex $v \in V_i$ such that $\left| [v]_{\vartheta} \cap V_i\right| \le \frac{|V_i|}{p}<\frac{3n}{(3k-2)p}\le \frac{n}{3k-2}$.
    Since $\left| \vartheta(V_{i})\right| \le  r$ (due to Claim~\ref{Claim:colors-in-V-I}), Proposition~\ref{PROP:Prelim-B}~\ref{PROP:Prelim-B-4} ensures the existence of an edge $e\in \mathcal{H}$ containing $v$ such that $\vartheta(V_i) \subseteq  \vartheta(e)$.
    This implies that $N_{\mathcal{H}}(e\setminus\{v\}) \cap V_i \subseteq [v]_{\theta} \cap V_i$.  
    Therefore, 
    \begin{align*}
        d_{\mathcal{H}}(e\setminus \{v\})
        \le |[v]_{\vartheta} \cap V_i| + \sum_{j \in \overline{\varphi(e)}} |V_j|
        < \frac{|V_i|}{p} + (k-r) \cdot \frac{3n}{3k-2}
        < \frac{3k-3r+1}{3k-2}n,
    \end{align*}
    contradicting Inequality~\eqref{equ:assumption-mincodegree}. 
\end{proof}

\subsection{Proof of Proposition~\ref{PROP:Final-compute}}\label{SEC:proof-final-compute}
We prove Proposition~\ref{PROP:Final-compute} in this subsection. The following results will be useful. 

Recall from Section~\ref{SEC:Intorduction} that the $(r-2)$-th shadow $\partial_{r-2}\mathcal{H}$ of an $r$-graph $\mathcal{H}$ is a graph, where $\{u,v\} \in \partial_{r-2}\mathcal{H}$  if and only if  it is contained in some edge of $\mathcal{H}$. 
\begin{proposition}\label{PROP:Prelim-C}
    Let $n \ge k \ge r \ge 2$ be integers and $\mathcal{H}$ be an $n$-vertex $r$-graph satisfying Assumption~\ref{assume:2}. 
    Let $\varphi \in \mathrm{Hom}(\mathcal{H},K_{k}^{r})$. 
    For every pair $\{u,v\} \in \partial_{r-2}\mathcal{H}$, there exists an edge $e\in \mathcal{H}$ containing $\{u,v\}$ such that 
    \begin{align*}
       \min\left\{|[w]_{\varphi}| \colon w\in e\setminus\{u,v\}\right\}
        \ge \max\left\{|\varphi^{-1}(i)| \colon i\in \overline{\varphi(e)} \right\}.
    \end{align*}
    In particular, by Proposition~\ref{PROP:Prelim-A}~\ref{PROP:Prelim-A-3}, $\varphi^{-1}(i)$ is small for every $i\in \overline{\varphi(e)}$. 
\end{proposition}
\begin{proof}[Proof of Proposition~\ref{PROP:Prelim-C}]
    The proof is analogous to that of  Proposition~\ref{PROP:Prelim-B}.
    Fix $\varphi \in \mathrm{Hom}(\mathcal{H},K_{k}^{r})$ and $\{u,v\} \in \partial_{r-2}\mathcal{H}$. 
    Recall from Proposition~\ref{PROP:Prelim-A}~\ref{PROP:Prelim-A-3} that $|J_{\varphi}| \le r-2$. 

    First, we claim that there exists an edge $\tilde{e} \in \mathcal{H}$ containing $\{u,v\}$ such that $J_{\varphi} \subseteq \varphi(\tilde{e})$. 
    Indeed, choose an edge $\tilde{e} \in \mathcal{H}$ containing $\{u,v\}$ such that $|\varphi(\tilde{e}) \cap J_{\varphi}|$ is maximized.
    Suppose to the contrary that $J_{\varphi} \not\subseteq  \varphi(\tilde{e})$. 
    Then there exists a vertex $w\in \tilde{e}$ such that $[w]_{\varphi}$ is small. 
    Similar to the proof of Proposition~\ref{PROP:Prelim-A}, it follows from the maximality of $\tilde{e}$ that 
    \begin{align*}
        d_{\mathcal{H}}(\tilde{e}\setminus\{w\})
        \le \sum_{j\in \overline{J_{\varphi} \cup \varphi(\tilde{e}\setminus \{w\})}} |\varphi^{-1}(j)|
        \le (k-r) \cdot \frac{3 n}{3k-2}
            < \frac{3k-3r+1}{3k-2} n,
    \end{align*}
    contradicting Inequality~\eqref{equ:assumption-mincodegree}.
    Therefore, we have $J_{\varphi} \subseteq  \varphi(\tilde{e})$. 

    Fix a set $I \subseteq  \overline{J_{\varphi} \cup \{\varphi(u), \varphi(v)\}}$ of size $r - |J_{\varphi} \cup \{\varphi(u), \varphi(v)\}|$ such that 
    \begin{align*}
        \min\left\{|\varphi^{-1}(i)| \colon i\in I\right\}
        \ge \max\left\{|\varphi^{-1}(j)| \colon j\in \overline{I \cup J_{\varphi} \cup \{\varphi(u), \varphi(v)\}} \right\}.
    \end{align*}
    It suffices to show that there exists an edge $e\in \mathcal{H}$ containing $\{u,v\}$ such that $\varphi(e) = I \cup J_{\varphi} \cup \{\varphi(u), \varphi(v)\}$.
    Let $e \in \mathcal{H}$ be an edge containing $J_{\varphi} \cup \{\varphi(u), \varphi(v)\}$ such that $|\varphi(e) \cap I|$ is maximized. 
    Suppose to the contrary that $|\varphi(e) \cap I| \le |I| - 1$. 
    Then there exists a vertex $w \in e$ such that  $\varphi(w) \in \overline{I \cup J_{\varphi} \cup \{\varphi(u), \varphi(v)\}}$. 
    In particular, $[w]_{\varphi}$ is small. 
    Therefore, similar to the argument above, we obtain 
    \begin{align*}
        d_{\mathcal{H}}(e\setminus\{w\})
        \le \sum_{j\in \overline{I \cup \varphi(e\setminus \{\tilde{w}\}})} |\varphi^{-1}(j)|
        \le (k-r) \cdot \frac{3 n}{3k-2}
        < \frac{3k-3r+1}{3k-2} n,
    \end{align*}
    contradicting Inequality~\eqref{equ:assumption-mincodegree}.
\end{proof}
    
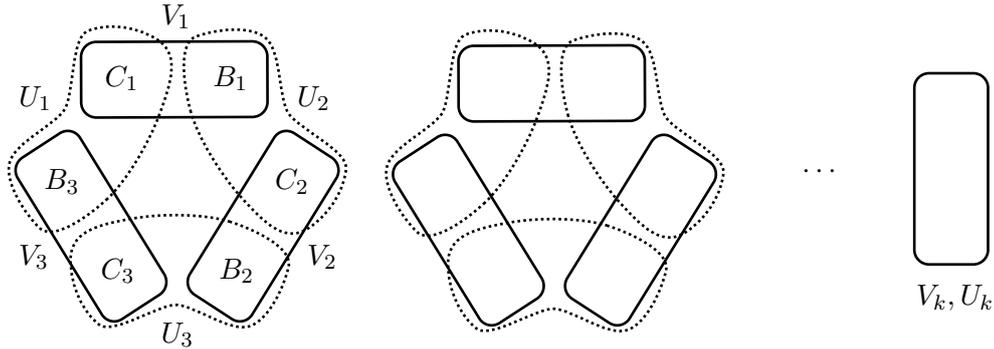
\begin{figure}[htbp]
\centering
\tikzset{every picture/.style={line width=1pt}} 

\begin{tikzpicture}[x=0.75pt,y=0.75pt,yscale=-1,xscale=1]

\draw   (254.82,122.74) .. controls (258.25,124.99) and (259.29,129.65) .. (257.15,133.15) -- (215.13,201.7) .. controls (212.98,205.19) and (208.46,206.2) .. (205.03,203.96) -- (186.39,191.74) .. controls (182.95,189.49) and (181.91,184.84) .. (184.05,181.34) -- (226.08,112.79) .. controls (228.22,109.29) and (232.74,108.28) .. (236.17,110.53) -- cycle ;
\draw   (129.71,72.2) .. controls (129.7,68) and (133.09,64.58) .. (137.29,64.57) -- (215.08,64.28) .. controls (219.28,64.27) and (222.7,67.66) .. (222.71,71.86) -- (222.79,94.69) .. controls (222.8,98.89) and (219.41,102.31) .. (215.21,102.33) -- (137.42,102.61) .. controls (133.22,102.63) and (129.8,99.23) .. (129.78,95.03) -- cycle ;
\draw   (150.22,203.96) .. controls (146.79,206.21) and (142.27,205.2) .. (140.13,201.7) -- (98.13,133.14) .. controls (95.99,129.64) and (97.03,124.98) .. (100.47,122.73) -- (119.11,110.52) .. controls (122.55,108.27) and (127.06,109.29) .. (129.21,112.78) -- (171.21,181.35) .. controls (173.35,184.85) and (172.3,189.5) .. (168.87,191.75) -- cycle ;
\draw  [dash pattern=on 1pt off 1.2pt] (203.75,154.91) .. controls (215.51,158.6) and (235.37,167.79) .. (233.32,177.55) .. controls (231.27,187.3) and (223.32,205.04) .. (213.35,207.7) .. controls (203.37,210.36) and (186.17,196.97) .. (177.14,196.84) .. controls (168.1,196.71) and (145.72,212.37) .. (138.57,206.66) .. controls (131.42,200.95) and (124.41,188.49) .. (124.5,173.98) .. controls (124.59,159.47) and (138.99,156.39) .. (154.92,153.49) .. controls (170.86,150.59) and (192,151.22) .. (203.75,154.91) -- cycle ;
\draw  [dash pattern=on 1pt off 1.2pt] (184.55,99.94) .. controls (181.15,87.74) and (177.8,65.51) .. (186.83,61.98) .. controls (195.85,58.45) and (214.58,55.68) .. (222.17,62.89) .. controls (229.75,70.1) and (228.25,92.3) .. (233.06,100.21) .. controls (237.87,108.12) and (262.76,119) .. (262.02,128.31) .. controls (261.28,137.62) and (254.97,150.49) .. (243.15,158.3) .. controls (231.33,166.12) and (220.99,155.31) .. (209.96,143.07) .. controls (198.94,130.83) and (187.95,112.14) .. (184.55,99.94) -- cycle ;
\draw  [dash pattern=on 1pt off 1.2pt] (142.94,143.4) .. controls (133.89,151.98) and (116.13,164.98) .. (109.05,158.19) .. controls (101.97,151.41) and (91.25,135.3) .. (94.11,125.07) .. controls (96.98,114.84) and (116.88,106.35) .. (121.59,98.37) .. controls (126.3,90.4) and (124.65,62.51) .. (133.04,59.05) .. controls (141.43,55.6) and (155.38,55.71) .. (167.41,63.17) .. controls (179.44,70.64) and (174.69,85.02) .. (168.99,100.68) .. controls (163.29,116.33) and (152,134.83) .. (142.94,143.4) -- cycle ;
\draw   (575.1,80.33) .. controls (579.17,80.33) and (582.46,83.63) .. (582.46,87.69) -- (582.46,169.12) .. controls (582.46,173.19) and (579.17,176.48) .. (575.1,176.48) -- (553.03,176.48) .. controls (548.96,176.48) and (545.67,173.19) .. (545.67,169.12) -- (545.67,87.69) .. controls (545.67,83.63) and (548.96,80.33) .. (553.03,80.33) -- cycle ;
\draw   (443.21,124.72) .. controls (446.64,126.97) and (447.68,131.62) .. (445.54,135.12) -- (403.51,203.67) .. controls (401.37,207.17) and (396.85,208.18) .. (393.42,205.93) -- (374.78,193.72) .. controls (371.34,191.47) and (370.3,186.81) .. (372.44,183.31) -- (414.47,114.76) .. controls (416.61,111.27) and (421.13,110.25) .. (424.56,112.5) -- cycle ;
\draw   (318.1,74.18) .. controls (318.08,69.97) and (321.48,66.56) .. (325.68,66.54) -- (403.47,66.26) .. controls (407.67,66.24) and (411.09,69.63) .. (411.1,73.84) -- (411.18,96.66) .. controls (411.19,100.87) and (407.8,104.29) .. (403.59,104.3) -- (325.81,104.59) .. controls (321.61,104.6) and (318.19,101.21) .. (318.17,97) -- cycle ;
\draw   (338.61,205.94) .. controls (335.18,208.18) and (330.66,207.17) .. (328.52,203.67) -- (286.52,135.11) .. controls (284.38,131.61) and (285.42,126.96) .. (288.85,124.71) -- (307.5,112.5) .. controls (310.93,110.25) and (315.45,111.26) .. (317.6,114.76) -- (359.6,183.32) .. controls (361.74,186.82) and (360.69,191.48) .. (357.26,193.73) -- cycle ;
\draw  [dash pattern=on 1pt off 1.2pt] (392.14,156.89) .. controls (403.89,160.58) and (423.76,169.77) .. (421.71,179.52) .. controls (419.66,189.27) and (411.71,207.02) .. (401.73,209.67) .. controls (391.76,212.33) and (374.56,198.95) .. (365.52,198.81) .. controls (356.49,198.68) and (334.11,214.35) .. (326.96,208.64) .. controls (319.8,202.93) and (312.8,190.46) .. (312.89,175.95) .. controls (312.98,161.45) and (327.38,158.36) .. (343.31,155.46) .. controls (359.24,152.56) and (380.39,153.2) .. (392.14,156.89) -- cycle ;
\draw  [dash pattern=on 1pt off 1.2pt] (372.94,101.91) .. controls (369.53,89.71) and (366.19,67.48) .. (375.21,63.95) .. controls (384.24,60.42) and (402.97,57.66) .. (410.55,64.87) .. controls (418.14,72.07) and (416.64,94.28) .. (421.45,102.18) .. controls (426.26,110.09) and (451.15,120.98) .. (450.41,130.29) .. controls (449.66,139.6) and (443.36,152.46) .. (431.54,160.28) .. controls (419.72,168.09) and (409.38,157.29) .. (398.35,145.04) .. controls (387.33,132.8) and (376.34,114.11) .. (372.94,101.91) -- cycle ;
\draw  [dash pattern=on 1pt off 1.2pt] (331.33,145.38) .. controls (322.28,153.96) and (304.52,166.95) .. (297.44,160.17) .. controls (290.36,153.38) and (279.63,137.27) .. (282.5,127.04) .. controls (285.37,116.81) and (305.27,108.32) .. (309.98,100.35) .. controls (314.68,92.37) and (313.03,64.48) .. (321.43,61.03) .. controls (329.82,57.57) and (343.77,57.68) .. (355.8,65.15) .. controls (367.83,72.61) and (363.08,87) .. (357.38,102.65) .. controls (351.68,118.31) and (340.39,136.8) .. (331.33,145.38) -- cycle ;

\draw (168,44) node [anchor=north west][inner sep=0.75pt]   [align=left] {$V_1$};
\draw (241,165) node [anchor=north west][inner sep=0.75pt]   [align=left] {$V_2$};
\draw (97,165) node [anchor=north west][inner sep=0.75pt]   [align=left] {$V_3$};
\draw (97,85) node [anchor=north west][inner sep=0.75pt]   [align=left] {$U_1$};
\draw (236,85) node [anchor=north west][inner sep=0.75pt]   [align=left] {$U_2$};
\draw (168,204) node [anchor=north west][inner sep=0.75pt]   [align=left] {$U_3$};
\draw (545,185) node [anchor=north west][inner sep=0.75pt]   [align=left] {$V_k, U_k$};
\draw (194,75) node [anchor=north west][inner sep=0.75pt]   [align=left] {$B_1$};
\draw (140,75) node [anchor=north west][inner sep=0.75pt]   [align=left] {$C_1$};
\draw (225,127) node [anchor=north west][inner sep=0.75pt]   [align=left] {$C_2$};
\draw (197,172) node [anchor=north west][inner sep=0.75pt]   [align=left] {$B_2$};
\draw (138,172) node [anchor=north west][inner sep=0.75pt]   [align=left] {$C_3$};
\draw (109.78,127) node [anchor=north west][inner sep=0.75pt]   [align=left] {$B_3$};
\draw (488,125) node [anchor=north west][inner sep=0.75pt]   [align=left] {$\cdots$};
\end{tikzpicture}
\caption{the intersection of two homomorphisms has cyclic structures.} 
\label{Fig:cyclic-structure}
\end{figure}

    \begin{proposition}\label{PROP:cyclic-structure}
        Let $n \ge k \ge r \ge 2$ be integers and $\mathcal{H}$ be an $n$-vertex $r$-graph satisfying Assumption~\ref{assume:2}. 
        Suppose that $\varphi, \vartheta \in \mathrm{Hom}(\mathcal{H}, K_{k}^{r})$ satisfy $|\vartheta(\varphi^{-1}(i))| \le 2$ for every $i \in [k]$. 
        Then the following statements hold. 
        \begin{enumerate}[label=(\roman*)]
            \item\label{PROP:cyclic-structure-1} 
                There exists a $t$-set $T \subseteq [k]$ for some $t \ge 2$, and by symmetry, we may assume that $T = [t]$, such that $\vartheta(V_i) = \{k_{i}, k_{i+1}\}$ for $i\in [t]$, where $\{k_1, \ldots, k_{t}\} \subseteq  [k]$ is a $t$-set, and $k_{t+1} \coloneqq k_1$ (see Figure~\ref{Fig:cyclic-structure}).  
            \item\label{PROP:cyclic-structure-2} 
                Every set $T \subseteq [k]$ satisfying the conclusion above must have $|T| \ge 3$. 
        \end{enumerate}
    \end{proposition}
    \begin{proof}[Proof of Proposition~\ref{PROP:cyclic-structure}]
        Let $\mathcal{H}, \varphi, \vartheta$ be as assumed in Proposition~\ref{PROP:cyclic-structure}. 
        Let $V_i \coloneqq \varphi^{-1}(i)$ and $U_i \coloneqq \vartheta^{-1}(i)$ for $i \in [k]$. 
        We begin by proving Proposition~\ref{PROP:cyclic-structure}~\ref{PROP:cyclic-structure-1}. 
        Note that it suffices to show that if $|\vartheta(V_{i})|=2$ for some $i\in [k]$ (which is ensured by the assumption that $\vartheta \ncong \varphi$), then there exists $j \in [k] \setminus \{i\}$ with $|\vartheta(V_{j})|=2$ such that $|\vartheta(V_{i})\cap \vartheta(V_{j})|\ge 1$. 
        
        Fix an $i \in [k]$ such that $\vartheta(V_{i})=\{p,q\}$ for some distinct $p,q \in [k]$.
        This is equivalent to saying that $V_i \cap U_p\neq \emptyset$ and $V_i \cap U_q\neq \emptyset$. 
        Fix a vertex $v \in  V_{i} \cap U_{q}$. 
        By Proposition~\ref{PROP:Prelim-B}~\ref{PROP:Prelim-B-3}, there exists an edge $e\in \mathcal{H}$ containing $v$ such that $p\in \vartheta(e)$, i.e. $e\cap U_p \neq \emptyset$. 
        Let $u$ denote the vertex in $e\cap U_p$. 
        Since $e \cap V_{i} \neq \emptyset$, the index $j \in [k]$ such that $u \in V_j$ must be different from $i$, i.e. $j \in [k] \setminus \{i\}$. 
        Note that $u \in V_j \cap U_p$ and $\vartheta(u) = p$. 
        So we obtain $p \in \vartheta (V_{i}) \cap \vartheta (V_{j})$, and hence, $|\vartheta (V_{i}) \cap \vartheta (V_{j})| \ge 1$.  
        
        Now fix a vertex $w \in V_{i}\cap U_p$. 
        By Proposition~\ref{PROP:Prelim-B}~\ref{PROP:Prelim-B-3}, there exists an edge $\tilde{e}\in \mathcal{H}$ containing $w$ such that $j \in \varphi(\tilde{e})$, i.e. $\tilde{e} \cap V_j \neq \emptyset$. 
        This implies that $\vartheta(\tilde{e}\cap V_{j}) \neq \vartheta(w) = p$.
        Therefore, $\vartheta(V_{j})$ has size at least two, and based on the assumption, we conclude that $|\vartheta(V_{j})| = 2$.
        This completes the proof for Proposition~\ref{PROP:cyclic-structure}~\ref{PROP:cyclic-structure-1}. 

        Next, we prove Proposition~\ref{PROP:cyclic-structure}~\ref{PROP:cyclic-structure-2}.
        Suppose to the contrary that there exists a set $T \subseteq [k]$ of size two satisfying Proposition~\ref{PROP:cyclic-structure}~\ref{PROP:cyclic-structure-1}. 
        Let $\{k_1, k_2\} \subseteq [k]$ be the corresponding set guaranteed by Proposition~\ref{PROP:cyclic-structure}~\ref{PROP:cyclic-structure-1}. 
        By relabeling the vertices in $K_{k}^{r}$, we may assume that $T = \{1,2\}$. 
        We may also assume that $(k_1,k_2) = (1,2)$, since otherwise, we can replace $\vartheta$ with $\eta \circ \vartheta$, where $\eta \in \mathrm{Aut}(K_{k}^{r})$ satisfies $\left(\eta(k_1), \eta(k_2) \right) = (1,2)$. 
        
\begin{figure}[htbp]
\centering

\tikzset{every picture/.style={line width=1pt}} 

\begin{tikzpicture}[x=0.75pt,y=0.75pt,yscale=-1,xscale=1]

\draw   (129.19,38.45) .. controls (129.19,33.03) and (133.58,28.65) .. (138.99,28.65) -- (168.39,28.65) .. controls (173.81,28.65) and (178.19,33.03) .. (178.19,38.45) -- (178.19,142.48) .. controls (178.19,147.89) and (173.81,152.28) .. (168.39,152.28) -- (138.99,152.28) .. controls (133.58,152.28) and (129.19,147.89) .. (129.19,142.48) -- cycle ;
\draw   (186.69,39.72) .. controls (186.69,34.16) and (191.2,29.65) .. (196.77,29.65) -- (227,29.65) .. controls (232.56,29.65) and (237.07,34.16) .. (237.07,39.72) -- (237.07,143.2) .. controls (237.07,148.77) and (232.56,153.28) .. (227,153.28) -- (196.77,153.28) .. controls (191.2,153.28) and (186.69,148.77) .. (186.69,143.2) -- cycle ;
\draw  [dash pattern=on 1pt off 1.2pt] (231.64,36.84) .. controls (237.47,36.84) and (242.19,41.57) .. (242.19,47.39) -- (242.19,79.05) .. controls (242.19,84.88) and (237.47,89.61) .. (231.64,89.61) -- (134.4,89.61) .. controls (128.57,89.61) and (123.85,84.88) .. (123.85,79.05) -- (123.85,47.39) .. controls (123.85,41.57) and (128.57,36.84) .. (134.4,36.84) -- cycle ;
\draw [dash pattern=on 1pt off 1.2pt]  (232.43,96.28) .. controls (237.82,96.28) and (242.19,100.65) .. (242.19,106.04) -- (242.19,135.32) .. controls (242.19,140.71) and (237.82,145.08) .. (232.43,145.08) -- (133.61,145.08) .. controls (128.22,145.08) and (123.85,140.71) .. (123.85,135.32) -- (123.85,106.04) .. controls (123.85,100.65) and (128.22,96.28) .. (133.61,96.28) -- cycle ;
\draw   (253.69,42.72) .. controls (253.69,37.16) and (258.2,32.65) .. (263.77,32.65) -- (294,32.65) .. controls (299.56,32.65) and (304.07,37.16) .. (304.07,42.72) -- (304.07,146.2) .. controls (304.07,151.77) and (299.56,156.28) .. (294,156.28) -- (263.77,156.28) .. controls (258.2,156.28) and (253.69,151.77) .. (253.69,146.2) -- cycle ;
\draw  (321.69,43.72) .. controls (321.69,38.16) and (326.2,33.65) .. (331.77,33.65) -- (362,33.65) .. controls (367.56,33.65) and (372.07,38.16) .. (372.07,43.72) -- (372.07,147.2) .. controls (372.07,152.77) and (367.56,157.28) .. (362,157.28) -- (331.77,157.28) .. controls (326.2,157.28) and (321.69,152.77) .. (321.69,147.2) -- cycle ;
\draw   (396.19,86.65) .. controls (396.19,82.03) and (399.94,78.28) .. (404.57,78.28) -- (429.7,78.28) .. controls (434.32,78.28) and (438.07,82.03) .. (438.07,86.65) -- (438.07,148.9) .. controls (438.07,153.53) and (434.32,157.28) .. (429.7,157.28) -- (404.57,157.28) .. controls (399.94,157.28) and (396.19,153.53) .. (396.19,148.9) -- cycle ;
\draw   (448.19,86.65) .. controls (448.19,82.03) and (451.94,78.28) .. (456.57,78.28) -- (481.7,78.28) .. controls (486.32,78.28) and (490.07,82.03) .. (490.07,86.65) -- (490.07,148.9) .. controls (490.07,153.53) and (486.32,157.28) .. (481.7,157.28) -- (456.57,157.28) .. controls (451.94,157.28) and (448.19,153.53) .. (448.19,148.9) -- cycle ;
\draw   (500.19,87.65) .. controls (500.19,83.03) and (503.94,79.28) .. (508.57,79.28) -- (533.7,79.28) .. controls (538.32,79.28) and (542.07,83.03) .. (542.07,87.65) -- (542.07,149.9) .. controls (542.07,154.53) and (538.32,158.28) .. (533.7,158.28) -- (508.57,158.28) .. controls (503.94,158.28) and (500.19,154.53) .. (500.19,149.9) -- cycle ;
\draw   (552.19,87.65) .. controls (552.19,83.03) and (555.94,79.28) .. (560.57,79.28) -- (585.7,79.28) .. controls (590.32,79.28) and (594.07,83.03) .. (594.07,87.65) -- (594.07,149.9) .. controls (594.07,154.53) and (590.32,158.28) .. (585.7,158.28) -- (560.57,158.28) .. controls (555.94,158.28) and (552.19,154.53) .. (552.19,149.9) -- cycle ;

\draw  [line width=2pt]  (155,60) -- (213,123) -- (280,60) -- (348,60) -- (417,104) -- (470,104) ;
\draw [fill=green] (155,60) circle (2.5pt);
\draw [fill=uuuuuu] (213,123) circle (2.5pt);
\draw [fill=uuuuuu] (280,60) circle (2.5pt);
\draw [fill=uuuuuu] (348,60) circle (2.5pt);
\draw [fill=uuuuuu] (417,104) circle (2.5pt);
\draw [fill=uuuuuu] (470,104) circle (2.5pt);
\draw (148,67) node [anchor=north west][inner sep=0.75pt]   [align=left] {$v$};
\draw (145,162) node [anchor=north west][inner sep=0.75pt]   [align=left] {$V_1$};
\draw (206,162) node [anchor=north west][inner sep=0.75pt]   [align=left] {$V_2$};
\draw (97,55) node [anchor=north west][inner sep=0.75pt]   [align=left] {$U_1$};
\draw (97,115) node [anchor=north west][inner sep=0.75pt]   [align=left] {$U_2$};
\end{tikzpicture}
\caption{auxiliary figure for the proof of Proposition~\ref{PROP:cyclic-structure}~\ref{PROP:cyclic-structure-2}.} 
\label{Fig:Claim2-19}
\end{figure}
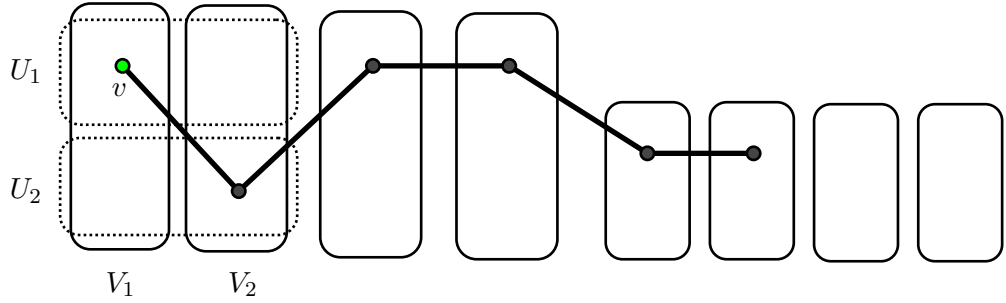
        
        \textbf{Case 1}$\colon$ $k \ge \frac{4r-2}{3}$. 
    
        Let $s \coloneqq |J_{\varphi}\setminus\{1,2\}| \le  |J_{\varphi}|\le  r-2$ (due to Proposition~\ref{PROP:Prelim-A}~\ref{PROP:Prelim-A-3}).
        By symmetry, we may assume 
        \begin{align*}
            |V_{1}\cap U_{1}|
            = \min\left\{|V_{i}\cap U_{j}| \colon (i,j) \in [2] \times [2] \right\}. 
        \end{align*}
        In particular, $|V_{1}\cap U_{1}| \le (|V_1|+|V_2|)/4$. 
        Fix a vertex $v \in  V_{1}\cap U_{1}$. 
        By Proposition~\ref{PROP:Prelim-B}~\ref{PROP:Prelim-B-3}, there exists an edge $e\in \mathcal{H}$ (see Figure~\ref{Fig:Claim2-19}) containing $v$ such that 
        \begin{align*}
            2 \in \varphi(e)
            \quad\text{and}\quad 
            \max\left\{|V_i| \colon i\in \overline{\varphi(e)} \right\} 
            \le \min\left\{|V_i| \colon i\in \varphi(e) \setminus \{1,2\} \right\}. 
        \end{align*}
        Since $|J_{\varphi} \cup \{1,2\}| \le r-2 + 2 = r$, it follows from the definition of $e$ that $J_{\varphi} \subseteq \varphi(e)$, which implies that $\overline{\varphi(e)} \subseteq \overline{J_{\varphi}\cup \{1,2\}}$. 
        Additionally, it follows from a simple averaging argument that 
        \begin{align*}
            \sum_{j\in \overline{\varphi(e)}}\left| V_{j}\right| 
            \le \frac{|\overline{\varphi(e)}|}{|\overline{J_{\varphi} \cup \{1,2\}}|} \cdot \sum_{j \in \overline{J_{\varphi} \cup \{1,2\}}}\left|V_{j}\right|
            = \frac{k-r}{k-s-2} \cdot \sum_{j \in \overline{J_{\varphi} \cup \{1,2\}}}\left|V_{j}\right|. 
        \end{align*}
        Therefore, 
        \begin{align*}
            \frac{3k-3r+1}{3k-2}n
            <d_{\mathcal{H}}\left( e\setminus\{v\}\right) 
            & \le  \left|V_{1}\cap U_1\right|
            +\sum_{j\in \overline{\varphi(e)}}\left| V_{j}\right|   \\
            & \le  \frac{|V_{1}|+|V_{2}|}{4}+\frac{k-r}{k-s-2} \cdot \sum_{j \in \overline{J_{\varphi} \cup \{1,2\}}}\left|V_{j}\right|. 
        \end{align*}
        Consequently, 
        \begin{align*}
            n
            & = \left|V_{1} \right|+\cdots+\left| V_{k}\right|\\
            & = 4\left(\frac{|V_{1}|+|V_{2}|}{4}
                +\frac{k-r}{k-s-2} \cdot \sum_{j \in \overline{J_{\varphi} \cup \{1,2\}}}\left|V_{j}\right| \right)  \\
            & \quad  - \frac{3k-(4r-2-s)}{k-s-2} \cdot \sum_{j \in \overline{J_{\varphi} \cup \{1,2\}}}\left|V_{j}\right|
                + \sum_{j \in J_{\varphi} \setminus \{1,2\}}\left|V_{j}\right| \\
            & > 4 \cdot \frac{3k-3r+1}{3k-2}n
                - \frac{3k-(4r-2-s)}{k-s-2} \cdot (k-s-2) \cdot \frac{3n}{3k-2}     
                + \frac{3s}{3k-2}n
            = n,  
        \end{align*}
        a contradiction. 
        Here we used the fact that $4 \cdot \frac{k-r}{k-s-2} - 1 = \frac{3k-(4r-2-s)}{k-s-2} \ge 0$, which follows from the assumption that $k \ge \frac{4r-2}{3}$. 

        \medskip 

        \textbf{Case 2}$\colon$ $k \le \frac{4r-2}{3}$. 
        
        Let $s \coloneqq |I_{\varphi} \setminus\{1,2\}| \ge |I_{\varphi}|-2 \ge r-2$ (due to Proposition~\ref{PROP:Prelim-D}).
        By symmetry, we may assume 
        \begin{align*}
            |V_{1}\cap U_{1}|
            = \min\left\{|V_{i}\cap U_{j}| \colon (i,j) \in [2] \times [2] \right\}. 
        \end{align*}
        In particular, $|V_{1}\cap U_{1}| \le (|V_1|+|V_2|)/4$. 
        Fix a vertex $v \in  V_{1}\cap U_{1}$. 
        By Proposition~\ref{PROP:Prelim-B}~\ref{PROP:Prelim-B-3}, there exists an edge $e\in \mathcal{H}$ containing $v$ such that 
        \begin{align*}
            2 \in \varphi(e)
            \quad\text{and}\quad 
            \max\left\{|V_i| \colon i \in \overline{\varphi(e)}\right\} 
            \le \min\left\{|V_i| \colon i \in \varphi(e)\setminus \{1,2\} \right\}.
        \end{align*} 
        Note that 
        \begin{align*}
            \frac{k-r+1}{k+2}n
            < d_{\mathcal{H}}\left( e\setminus\{v\}\right) 
            \le  \left|V_{1}\cap U_1 \right|+\sum_{i\in \overline{\varphi(e)}}\left| V_{i}\right|\le  \frac{|V_{1}|+|V_{2}|}{4}+\sum_{i\in \overline{\varphi(e)}}\left| V_{i}\right|. 
        \end{align*}
        Therefore, 
        \begin{align}\label{equ:main-lemma-n}
            n
             = \sum_{i\in [k]}|V_i|  
            & = 4\left(\frac{|V_{1}|+|V_{2}|}{4}+\sum_{i\in \overline{\varphi(e)}}\left| V_{i}\right| \right) - 4\sum_{i\in \overline{\varphi(e)}}\left| V_{i}\right| + \sum_{i \in [3,k]}  |V_i|  \notag \\
            & >4\cdot\frac{k-r+1}{k+2}n
            -4\sum_{i\in \overline{\varphi(e)}}\left| V_{i}\right|
            +\sum_{i \in [3,k]}  |V_i|.
        \end{align}
        Since both $k$ and $r$ are integers, the assumption $k < \frac{4r-2}{3}$ implies that $k \le \frac{4r-3}{3}$. 
        It follows that $k-3 \ge 4(k-r) = 4 \cdot |\overline{\varphi(e)}|$.
        The choice of $e$ implies that 
        \begin{align*}
            \sum_{i \in [3,k]}  |V_i| - 4\sum_{i\in \overline{\varphi(e)}}\left| V_{i}\right|
            \ge \max_{T \in \binom{[3,k]}{4r-2-3k}} \left\{ \sum_{i\in T} |V_i|  \right\} 
            \ge (4r-3k-2) \cdot \frac{n}{k+2},
        \end{align*}
        where the last inequality follows from the fact that $|I_{\varphi} \cap [3,k]| \ge r-2 \ge 4r-3k-2$. 
        Therefore, Inequality~\eqref{equ:main-lemma-n} continues as
        \begin{align*}
            n 
            >4\cdot\frac{k-r+1}{k+2}n+(4r-3k-2) \cdot \frac{n}{k+2}
            = n, 
        \end{align*}
        a contradiction.
    \end{proof}
%

Now we are ready to prove Proposition~\ref{PROP:Final-compute}, and we will proceed by considering two cases based on the value of $k$. 
\begin{proof}[Proof of Proposition~\ref{PROP:Final-compute} for $k \ge \frac{4r-2}{3}$]
    Let $n \ge k \ge r \ge 2$ be integers satisfying $k \ge \frac{4r-2}{3}$. 
    Let $\mathcal{H}$ be an $n$-vertex $k$-partite $r$-graph with no isolated vertices and with $\delta_{r-1}^{+}(\mathcal{H}) > \frac{3k-3r+1}{3k-2}n$. 
    Fix $\varphi, \vartheta \in \mathrm{Hom}(\mathcal{H},K_{k}^{r})$ such that $|\vartheta(\varphi^{-1}(i))| \le 2$ for every $i \in [k]$. 
    Let $V_{i}\coloneqq \varphi^{-1}(i)$ and $U_{i}\coloneqq \vartheta^{-1}(i)$ for $i\in [k]$. 
    Suppose to the contrary that $\varphi \ncong \vartheta$. 

    %
\begin{figure}[htbp]
\centering
\tikzset{every picture/.style={line width=1pt}} 

\begin{tikzpicture}[x=0.75pt,y=0.75pt,yscale=-1,xscale=1]

\draw   (236.52,105.68) .. controls (239.72,107.74) and (240.69,112.03) .. (238.67,115.27) -- (199.47,178.02) .. controls (197.45,181.25) and (193.21,182.2) .. (190,180.14) -- (172.58,168.94) .. controls (169.37,166.88) and (168.41,162.58) .. (170.43,159.35) -- (209.62,96.6) .. controls (211.64,93.37) and (215.88,92.42) .. (219.09,94.48) -- cycle ;
\draw   (119.62,59.34) .. controls (119.61,55.49) and (122.72,52.35) .. (126.57,52.34) -- (199.52,52.09) .. controls (203.37,52.07) and (206.5,55.18) .. (206.52,59.04) -- (206.59,79.96) .. controls (206.6,83.81) and (203.49,86.94) .. (199.64,86.96) -- (126.69,87.21) .. controls (122.84,87.23) and (119.71,84.11) .. (119.7,80.26) -- cycle ;
\draw   (138.79,180.14) .. controls (135.58,182.2) and (131.34,181.25) .. (129.33,178.02) -- (90.14,115.26) .. controls (88.12,112.02) and (89.08,107.73) .. (92.29,105.67) -- (109.72,94.47) .. controls (112.93,92.41) and (117.17,93.36) .. (119.19,96.6) -- (158.37,159.36) .. controls (160.39,162.59) and (159.43,166.88) .. (156.22,168.95) -- cycle ;
\draw  [dash pattern=on 1pt off 1.2pt] (188.81,135.17) .. controls (199.79,138.55) and (218.34,146.98) .. (216.43,155.92) .. controls (214.52,164.86) and (207.09,181.13) .. (197.77,183.57) .. controls (188.45,186.01) and (172.38,173.73) .. (163.93,173.61) .. controls (155.49,173.49) and (134.58,187.85) .. (127.9,182.62) .. controls (121.22,177.38) and (114.67,165.95) .. (114.76,152.65) .. controls (114.84,139.35) and (128.29,136.52) .. (143.18,133.86) .. controls (158.07,131.2) and (177.83,131.79) .. (188.81,135.17) -- cycle ;
\draw  [dash pattern=on 1pt off 1.2pt] (170.86,84.77) .. controls (167.68,73.59) and (164.55,53.2) .. (172.99,49.97) .. controls (181.42,46.73) and (198.92,44.2) .. (206.01,50.8) .. controls (213.1,57.41) and (211.69,77.77) .. (216.19,85.02) .. controls (220.68,92.27) and (243.94,102.25) .. (243.24,110.79) .. controls (242.55,119.32) and (236.66,131.11) .. (225.61,138.28) .. controls (214.57,145.44) and (204.91,135.54) .. (194.61,124.31) .. controls (184.31,113.09) and (174.04,95.95) .. (170.86,84.77) -- cycle ;
\draw  [dash pattern=on 1pt off 1.2pt] (131.99,124.62) .. controls (123.53,132.49) and (106.93,144.4) .. (100.32,138.18) .. controls (93.71,131.96) and (83.68,117.19) .. (86.36,107.81) .. controls (89.04,98.43) and (107.63,90.64) .. (112.03,83.33) .. controls (116.43,76.02) and (114.89,50.45) .. (122.73,47.28) .. controls (130.58,44.11) and (143.61,44.22) .. (154.85,51.06) .. controls (166.09,57.91) and (161.65,71.09) .. (156.32,85.45) .. controls (151,99.8) and (140.45,116.75) .. (131.99,124.62) -- cycle ;
\draw   (302.56,67.93) .. controls (306.36,67.93) and (309.44,71.01) .. (309.44,74.8) -- (309.44,149.2) .. controls (309.44,153) and (306.36,156.08) .. (302.56,156.08) -- (281.94,156.08) .. controls (278.14,156.08) and (275.06,153) .. (275.06,149.2) -- (275.06,74.8) .. controls (275.06,71.01) and (278.14,67.93) .. (281.94,67.93) -- cycle ;
\draw   (378.56,90.28) .. controls (382.36,90.28) and (385.44,93.36) .. (385.44,97.15) -- (385.44,149.54) .. controls (385.44,153.34) and (382.36,156.41) .. (378.56,156.41) -- (357.94,156.41) .. controls (354.14,156.41) and (351.06,153.34) .. (351.06,149.54) -- (351.06,97.15) .. controls (351.06,93.36) and (354.14,90.28) .. (357.94,90.28) -- cycle ;
\draw   (427.56,90.28) .. controls (431.36,90.28) and (434.44,93.36) .. (434.44,97.15) -- (434.44,149.54) .. controls (434.44,153.34) and (431.36,156.41) .. (427.56,156.41) -- (406.94,156.41) .. controls (403.14,156.41) and (400.06,153.34) .. (400.06,149.54) -- (400.06,97.15) .. controls (400.06,93.36) and (403.14,90.28) .. (406.94,90.28) -- cycle ;
\draw   (479.48,90.54) .. controls (483.28,90.54) and (486.36,93.61) .. (486.36,97.41) -- (486.36,149.8) .. controls (486.36,153.59) and (483.28,156.67) .. (479.48,156.67) -- (458.85,156.67) .. controls (455.05,156.67) and (451.98,153.59) .. (451.98,149.8) -- (451.98,97.41) .. controls (451.98,93.61) and (455.05,90.54) .. (458.85,90.54) -- cycle ;
\draw   (531.56,91.03) .. controls (535.36,91.03) and (538.44,94.11) .. (538.44,97.9) -- (538.44,150.29) .. controls (538.44,154.09) and (535.36,157.17) .. (531.56,157.17) -- (510.94,157.17) .. controls (507.14,157.17) and (504.06,154.09) .. (504.06,150.29) -- (504.06,97.9) .. controls (504.06,94.11) and (507.14,91.03) .. (510.94,91.03) -- cycle ;
\draw  [line width=2pt]  (109,115) -- (218,115) -- (292,92) -- (368,123) -- (417,123) -- (470,123) ;
\draw [fill=uuuuuu] (109,115) circle (2.5pt);
\draw [fill=uuuuuu] (218,115) circle (2.5pt);
\draw [fill=uuuuuu] (292,92) circle (2.5pt);
\draw [fill=uuuuuu] (368,123) circle (2.5pt);
\draw [fill=uuuuuu] (417,123) circle (2.5pt);
\draw [fill=green] (470,123) circle (2.5pt);
\draw (155,33.4) node [anchor=north west][inner sep=0.75pt]   [align=left] {$V_1$};
\draw (224,142) node [anchor=north west][inner sep=0.75pt]   [align=left] {$V_2$};
\draw (86,142) node [anchor=north west][inner sep=0.75pt]   [align=left] {$V_3$};
\draw (86,70) node [anchor=north west][inner sep=0.75pt]   [align=left] {$U_1$};
\draw (220,70) node [anchor=north west][inner sep=0.75pt]   [align=left] {$U_2$};
\draw (155,180) node [anchor=north west][inner sep=0.75pt]   [align=left] {$U_3$};
\draw (284,161) node [anchor=north west][inner sep=0.75pt]   [align=left] {$V_4$};
\draw (104,122) node [anchor=north west][inner sep=0.75pt]   [align=left] {$u$};
\draw (212,122) node [anchor=north west][inner sep=0.75pt]   [align=left] {$v$};
\draw (285,100) node [anchor=north west][inner sep=0.75pt]   [align=left] {$w_4$};
\draw (462,132) node [anchor=north west][inner sep=0.75pt]   [align=left] {$w$};

\end{tikzpicture}

\caption{auxiliary figure for the proof of Claim~\ref{Claim:C2_Bt}.} 
\label{Fig:Claim2-20}
\end{figure}
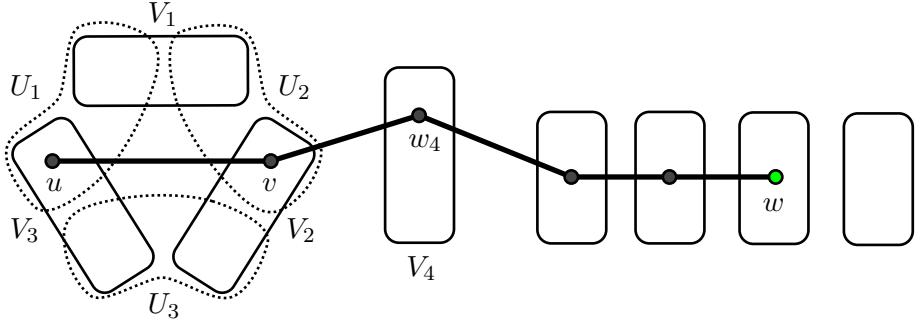
    %
    Let $T \subseteq [k]$ and $\{k_1, \ldots, k_t\} \subseteq [k]$ be the $t$-sets guaranteed by Proposition~\ref{PROP:cyclic-structure}, where $t \ge 3$. 
    Similar to the proof of Proposition~\ref{PROP:cyclic-structure}~\ref{PROP:cyclic-structure-2}, we may assume that $T = [t]$ and $(k_1, \ldots, k_t) = (1, \ldots, t)$. 
    For convenience, let $B_{i} \coloneqq V_{i}\cap U_{i+1}$ and $C_{i} \coloneqq V_{i}\cap U_i$ for $i\in [t]$, with the indices taken modulo $t$ (see Figure~\ref{Fig:cyclic-structure}).  
    \begin{claim}\label{Claim:C2_Bt}
        We have $\{v,u\} \not\in \partial_{r-2}\mathcal{H}$ for every pair $(v,u) \in C_2 \times B_{t}$.
    \end{claim}
    \begin{proof}[Proof of Claim~\ref{Claim:C2_Bt}]
        Suppose to the contrary that there exists a pair of vertices $(v,u) \in C_2 \times B_{t}$ such that $\{v,u\} \in \partial_{r-2}\mathcal{H}$.
        By Proposition~\ref{PROP:Prelim-C}, there exists an edge $e\in \mathcal{H}$ containing $\{v, u\}$ such that $V_j$ is small for every $j \in \overline{\varphi(e)}$. 
        Since $V_1 = (U_1 \cap V_1) \cup (U_2 \cap V_1)$ and $\{v, u\} \subseteq e$, it follows that $1 \not\in \varphi(e)$ (see Figure~\ref{Fig:Claim2-20}).
        If there exists a vertex $w \in e \setminus \{v,u\}$ such that $[w]_{\varphi}$ is small, then we would have 
        \begin{align*}
            d_{\mathcal{H}}(e\setminus\{w\}) 
            \le \sum_{j\in \overline{\varphi \left(e\setminus\{w\} \right)\cup\{1\}}}\left| V_{j}\right| 
            < (k-r) \cdot \frac{3 n}{3k-2}
            <\frac{3k-3r+1}{3k-2}n, 
        \end{align*}
        a contradiction. 
        Therefore $[w]_{\varphi}$ is large for every $w \in e \setminus \{v,u\}$. 
        Together with Proposition~\ref{PROP:Prelim-A}~\ref{PROP:Prelim-A-3}, this implies that  $|J_{\varphi}| = r-2$ and $\varphi(e) = J_{\varphi} \cup \{2,t\}$.  
        
        For every $i \in J_{\varphi}$, let $w_i$ denote the vertex in $e\cap V_i$. 
        From the inequality  
        \begin{align*}
            \frac{3k-3r+1}{3k-2}n
            < d_{\mathcal{H}}(e\setminus\{w_i\})
            \le |V_{i}|+\sum_{j\in \overline{\varphi(e)\cup\{1\}}}|V_{j}|
            < |V_{i}|+ (k-r-1) \cdot \frac{3n}{3k-2}, 
        \end{align*}
        it follows that $|V_i| > \frac{4n}{3k-2}$. 
        Combining this with Proposition~\ref{PROP:Prelim-A}~\ref{PROP:Prelim-A-1}, we obtain 
        \begin{align*}
            n
            = \sum_{j\in [k]} |V_j|
            = \sum_{j\in J_{\varphi}} |V_j| + \sum_{j\in \overline{J_{\varphi}}} |V_j| 
            & > (r-2) \cdot \frac{4n}{3k-2} + \frac{k-r+2}{k-r+1} \cdot \frac{3k-3r+1}{3k-2} n \\
            & = \left(1 + \frac{(r-2)k-(r^2-3r+4)}{(3k-2)(k-r+1)}\right) n. 
        \end{align*}
        Since $k \ge \left\lceil \frac{4r-2}{3} \right\rceil$, simple calculations show that $\frac{(r-2)k-(r^2-3r+4)}{(3k-2)(k-r+1)} \ge 0$ for all $r \ge 3$, implying that the inequality above is a contradiction. 
    \end{proof}

\begin{figure}[htbp]
\centering
\tikzset{every picture/.style={line width=1pt}} 

\begin{tikzpicture}[x=0.75pt,y=0.75pt,yscale=-1,xscale=1]

\draw   (236.52,105.68) .. controls (239.72,107.74) and (240.69,112.03) .. (238.67,115.27) -- (199.47,178.02) .. controls (197.45,181.25) and (193.21,182.2) .. (190,180.14) -- (172.58,168.94) .. controls (169.37,166.88) and (168.41,162.58) .. (170.43,159.35) -- (209.62,96.6) .. controls (211.64,93.37) and (215.88,92.42) .. (219.09,94.48) -- cycle ;
\draw   (119.62,59.34) .. controls (119.61,55.49) and (122.72,52.35) .. (126.57,52.34) -- (199.52,52.09) .. controls (203.37,52.07) and (206.5,55.18) .. (206.52,59.04) -- (206.59,79.96) .. controls (206.6,83.81) and (203.49,86.94) .. (199.64,86.96) -- (126.69,87.21) .. controls (122.84,87.23) and (119.71,84.11) .. (119.7,80.26) -- cycle ;
\draw   (138.79,180.14) .. controls (135.58,182.2) and (131.34,181.25) .. (129.33,178.02) -- (90.14,115.26) .. controls (88.12,112.02) and (89.08,107.73) .. (92.29,105.67) -- (109.72,94.47) .. controls (112.93,92.41) and (117.17,93.36) .. (119.19,96.6) -- (158.37,159.36) .. controls (160.39,162.59) and (159.43,166.88) .. (156.22,168.95) -- cycle ;
\draw  [dash pattern=on 1pt off 1.2pt] (188.81,135.17) .. controls (199.79,138.55) and (218.34,146.98) .. (216.43,155.92) .. controls (214.52,164.86) and (207.09,181.13) .. (197.77,183.57) .. controls (188.45,186.01) and (172.38,173.73) .. (163.93,173.61) .. controls (155.49,173.49) and (134.58,187.85) .. (127.9,182.62) .. controls (121.22,177.38) and (114.67,165.95) .. (114.76,152.65) .. controls (114.84,139.35) and (128.29,136.52) .. (143.18,133.86) .. controls (158.07,131.2) and (177.83,131.79) .. (188.81,135.17) -- cycle ;
\draw  [dash pattern=on 1pt off 1.2pt] (170.86,84.77) .. controls (167.68,73.59) and (164.55,53.2) .. (172.99,49.97) .. controls (181.42,46.73) and (198.92,44.2) .. (206.01,50.8) .. controls (213.1,57.41) and (211.69,77.77) .. (216.19,85.02) .. controls (220.68,92.27) and (243.94,102.25) .. (243.24,110.79) .. controls (242.55,119.32) and (236.66,131.11) .. (225.61,138.28) .. controls (214.57,145.44) and (204.91,135.54) .. (194.61,124.31) .. controls (184.31,113.09) and (174.04,95.95) .. (170.86,84.77) -- cycle ;
\draw  [dash pattern=on 1pt off 1.2pt] (131.99,124.62) .. controls (123.53,132.49) and (106.93,144.4) .. (100.32,138.18) .. controls (93.71,131.96) and (83.68,117.19) .. (86.36,107.81) .. controls (89.04,98.43) and (107.63,90.64) .. (112.03,83.33) .. controls (116.43,76.02) and (114.89,50.45) .. (122.73,47.28) .. controls (130.58,44.11) and (143.61,44.22) .. (154.85,51.06) .. controls (166.09,57.91) and (161.65,71.09) .. (156.32,85.45) .. controls (151,99.8) and (140.45,116.75) .. (131.99,124.62) -- cycle ;
\draw   (302.56,67.93) .. controls (306.36,67.93) and (309.44,71.01) .. (309.44,74.8) -- (309.44,149.2) .. controls (309.44,153) and (306.36,156.08) .. (302.56,156.08) -- (281.94,156.08) .. controls (278.14,156.08) and (275.06,153) .. (275.06,149.2) -- (275.06,74.8) .. controls (275.06,71.01) and (278.14,67.93) .. (281.94,67.93) -- cycle ;
\draw   (378.56,90.28) .. controls (382.36,90.28) and (385.44,93.36) .. (385.44,97.15) -- (385.44,149.54) .. controls (385.44,153.34) and (382.36,156.41) .. (378.56,156.41) -- (357.94,156.41) .. controls (354.14,156.41) and (351.06,153.34) .. (351.06,149.54) -- (351.06,97.15) .. controls (351.06,93.36) and (354.14,90.28) .. (357.94,90.28) -- cycle ;
\draw   (427.56,90.28) .. controls (431.36,90.28) and (434.44,93.36) .. (434.44,97.15) -- (434.44,149.54) .. controls (434.44,153.34) and (431.36,156.41) .. (427.56,156.41) -- (406.94,156.41) .. controls (403.14,156.41) and (400.06,153.34) .. (400.06,149.54) -- (400.06,97.15) .. controls (400.06,93.36) and (403.14,90.28) .. (406.94,90.28) -- cycle ;
\draw   (479.48,90.54) .. controls (483.28,90.54) and (486.36,93.61) .. (486.36,97.41) -- (486.36,149.8) .. controls (486.36,153.59) and (483.28,156.67) .. (479.48,156.67) -- (458.85,156.67) .. controls (455.05,156.67) and (451.98,153.59) .. (451.98,149.8) -- (451.98,97.41) .. controls (451.98,93.61) and (455.05,90.54) .. (458.85,90.54) -- cycle ;
\draw   (531.56,91.03) .. controls (535.36,91.03) and (538.44,94.11) .. (538.44,97.9) -- (538.44,150.29) .. controls (538.44,154.09) and (535.36,157.17) .. (531.56,157.17) -- (510.94,157.17) .. controls (507.14,157.17) and (504.06,154.09) .. (504.06,150.29) -- (504.06,97.9) .. controls (504.06,94.11) and (507.14,91.03) .. (510.94,91.03) -- cycle ;
\draw  [line width=2pt]  (133.19,155.28) -- (218,115) -- (292,92) -- (368,123) -- (417,123) -- (470,123) ;
\draw [fill=uuuuuu] (133.19,155.28) circle (2.5pt);
\draw [fill=green] (218,115) circle (2.5pt);
\draw [fill=uuuuuu] (292,92) circle (2.5pt);
\draw [fill=uuuuuu] (368,123) circle (2.5pt);
\draw [fill=uuuuuu] (417,123) circle (2.5pt);
\draw [fill=uuuuuu] (470,123) circle (2.5pt);
\draw (155,33.4) node [anchor=north west][inner sep=0.75pt]   [align=left] {$V_1$};
\draw (224,142) node [anchor=north west][inner sep=0.75pt]   [align=left] {$V_2$};
\draw (86,142) node [anchor=north west][inner sep=0.75pt]   [align=left] {$V_3$};
\draw (86,70) node [anchor=north west][inner sep=0.75pt]   [align=left] {$U_1$};
\draw (220,70) node [anchor=north west][inner sep=0.75pt]   [align=left] {$U_2$};
\draw (155,180) node [anchor=north west][inner sep=0.75pt]   [align=left] {$U_3$};
\draw (212,122) node [anchor=north west][inner sep=0.75pt]   [align=left] {$v$};
\draw (128,162) node [anchor=north west][inner sep=0.75pt]   [align=left] {$u$};

\end{tikzpicture}

\caption{auxiliary figure for the proof of Claim~\ref{Claim:size-of-Bi}.} 
\label{Fig:Claim2-21}
\end{figure}
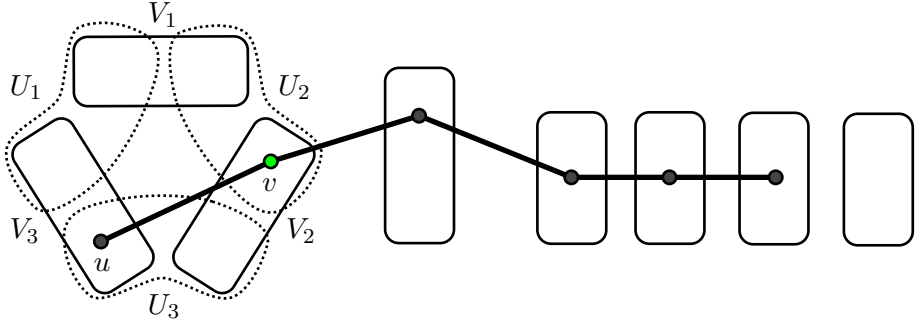
    %
    \begin{claim}\label{Claim:size-of-Bi}
        We have 
        \begin{align*}
            \min\left\{|B_{i}|+|B_{i+1}|,\  |C_{i}|+|C_{i+1}|\right\} 
            > \frac{4n}{3k-2}
            \quad\text{for every}\quad 
            i\in [t]. 
        \end{align*}
        In particular, 
        \begin{align*}
            |V_i|+|V_{i+1}|
            > \frac{8n}{3k-2}
            \quad\text{for every}\quad 
            i\in [t]. 
        \end{align*}
        Here, the indices are taken modulo $t$. 
    \end{claim}
    \begin{proof}[Proof of Claim~\ref{Claim:size-of-Bi}]
        Suppose to the contrary that this is not true. 
        By symmetry, we may assume that  $|C_{1}|+|C_{t}| \le \frac{4n}{3k-2}$. 
        Fix a vertex $v \in  C_{2}$.
        By Proposition~\ref{PROP:Prelim-B}~\ref{PROP:Prelim-B-3}, there exists an edge $e\in \mathcal{H}$ containing $v$ such that  
        \begin{align*}
            \text{$t \in \varphi(e)$ and $V_j$ is small for every $j \in \overline{\varphi(e) \cup \{1\}}$ (see Figure~\ref{Fig:Claim2-21}).}
        \end{align*}
        Additionally, we may assume that $1\not\in \varphi(e)$.  
        Indeed, suppose that $1 \in \varphi(e)$. 
        Then let $w$ denote the vertex in $e\cap V_1$. 
        Observe that $w$ must be contained in $C_1$ (see Figure~\ref{Fig:Claim2-21}). 
        Since $k \ge \left\lceil \frac{4r-2}{3} \right\rceil$, simple calculations show that $3k-3r+1 \ge 4$ for all integers $r \ge 3$. 
        Therefore, 
        \begin{align*}
            d_{\mathcal{H}}(e\setminus \{w\})
            > \frac{3k-3r+1}{3k-2}n 
            \ge \frac{4n}{3k-2}
            \ge |C_{1}|, 
        \end{align*}
        which means that there exists a vertex $\tilde{w} \in V(\mathcal{H}) \setminus C_1$ such that $(e\setminus\{w\}) \cup \{\tilde{w}\} \in \mathcal{H}$. 
        It is easy to see that $\tilde{w}$ cannot lie in $V_1$. 
        Thus, we can replace $e$ with $(e\setminus\{w\}) \cup \{\tilde{w}\}$, and this new edge will have empty intersection with $V_1$, as desired. 
        Therefore, we may initially assume that $1\not\in \varphi(e)$. 
        
        Let $u$ denote the vertex in $e\cap V_{t}$, noting that $u$ must lie in $C_t$ (see Figure~\ref{Fig:Claim2-21}). 
        Additionally, observe that $N_{\mathcal{H}}(e\setminus\{u\}) \cap V_1 \subseteq C_1$ and $N_{\mathcal{H}}(e\setminus\{u\}) \cap V_t \subseteq C_t$. 
        Therefore, similar to the proof of Claim~\ref{Claim:C2_Bt}, we obtain 
        \begin{align*}
            \frac{3k-3r+1}{3k-2}n
            < d_{\mathcal{H}}(e\setminus\{u\}) 
            & \le  |C_{1}|+|C_{t}|+\sum_{j\in \overline{\varphi(e)\cup\{1\}}}|V_{j}| \\
            & \le  \frac{4n}{3k-2}+ (k-r-1) \cdot \frac{3n}{3k-2}
            = \frac{3k-3r+1}{3k-2}n, 
        \end{align*}
        a contradiction. 
    \end{proof}

    \begin{claim}\label{Claim:many-large-parts}
        All but at most two sets in $\{V_{1}, \dots, V_{t}\}$ are large. 
        In particular, by Proposition~\ref{PROP:Prelim-A}~\ref{PROP:Prelim-A-3}, 
        \begin{align*}
            t \le |J_{\varphi}| + 2 \le r. 
        \end{align*}
    \end{claim}    
    \begin{proof}[Proof of Claim~\ref{Claim:many-large-parts}]
        Suppose to the contrary that there exist three sets in $\{V_{1}, \dots, V_{t}\}$ with size smaller than $\frac{3n}{3k-2}$. 
        By Claim~\ref{Claim:size-of-Bi}, $|V_{i}|+|V_{i+1}| > \frac{8n}{3k-2}$ for every $i\in [t]$.
        Therefore, there exist three sets  $V_{i_1}, V_{i_2}, V_{i_3}$ with size greater than $\frac{8n}{3k-2} - \frac{3n}{3k-2} = \frac{5n}{3k-2}$ (note that this implies that $|J_{\varphi}| \ge 3$, and hence, $r\ge 5$). 
        Combining this with Proposition~\ref{PROP:Prelim-A}~\ref{PROP:Prelim-A-1}, we obtian 
        \begin{align*}
            \sum_{j\in [k]}|V_j|
            & = |V_{i_1}| + |V_{i_2}| + |V_{i_3}| +\sum_{j \in \overline{\{i_1, i_2, i_3\}}} |V_j| \\
            & > 3\cdot \frac{5n}{3k-2} 
                + \frac{k-3}{k-r+1} \cdot \frac{3k-3r+1}{3k-2}n 
             = \left(1+ \frac{2 (3k - 4r +7)}{(3 k-2) (k-r+1)}\right) n 
            > n, 
        \end{align*}
        a contradiction. 
    \end{proof}

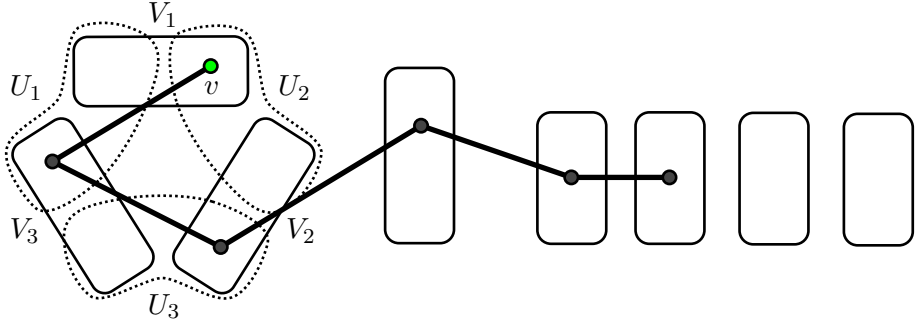
\begin{figure}[htbp]
\centering
\tikzset{every picture/.style={line width=1pt}} 

\begin{tikzpicture}[x=0.75pt,y=0.75pt,yscale=-1,xscale=1]

\draw   (236.52,105.68) .. controls (239.72,107.74) and (240.69,112.03) .. (238.67,115.27) -- (199.47,178.02) .. controls (197.45,181.25) and (193.21,182.2) .. (190,180.14) -- (172.58,168.94) .. controls (169.37,166.88) and (168.41,162.58) .. (170.43,159.35) -- (209.62,96.6) .. controls (211.64,93.37) and (215.88,92.42) .. (219.09,94.48) -- cycle ;
\draw   (119.62,59.34) .. controls (119.61,55.49) and (122.72,52.35) .. (126.57,52.34) -- (199.52,52.09) .. controls (203.37,52.07) and (206.5,55.18) .. (206.52,59.04) -- (206.59,79.96) .. controls (206.6,83.81) and (203.49,86.94) .. (199.64,86.96) -- (126.69,87.21) .. controls (122.84,87.23) and (119.71,84.11) .. (119.7,80.26) -- cycle ;
\draw   (138.79,180.14) .. controls (135.58,182.2) and (131.34,181.25) .. (129.33,178.02) -- (90.14,115.26) .. controls (88.12,112.02) and (89.08,107.73) .. (92.29,105.67) -- (109.72,94.47) .. controls (112.93,92.41) and (117.17,93.36) .. (119.19,96.6) -- (158.37,159.36) .. controls (160.39,162.59) and (159.43,166.88) .. (156.22,168.95) -- cycle ;
\draw  [dash pattern=on 1pt off 1.2pt] (188.81,135.17) .. controls (199.79,138.55) and (218.34,146.98) .. (216.43,155.92) .. controls (214.52,164.86) and (207.09,181.13) .. (197.77,183.57) .. controls (188.45,186.01) and (172.38,173.73) .. (163.93,173.61) .. controls (155.49,173.49) and (134.58,187.85) .. (127.9,182.62) .. controls (121.22,177.38) and (114.67,165.95) .. (114.76,152.65) .. controls (114.84,139.35) and (128.29,136.52) .. (143.18,133.86) .. controls (158.07,131.2) and (177.83,131.79) .. (188.81,135.17) -- cycle ;
\draw  [dash pattern=on 1pt off 1.2pt] (170.86,84.77) .. controls (167.68,73.59) and (164.55,53.2) .. (172.99,49.97) .. controls (181.42,46.73) and (198.92,44.2) .. (206.01,50.8) .. controls (213.1,57.41) and (211.69,77.77) .. (216.19,85.02) .. controls (220.68,92.27) and (243.94,102.25) .. (243.24,110.79) .. controls (242.55,119.32) and (236.66,131.11) .. (225.61,138.28) .. controls (214.57,145.44) and (204.91,135.54) .. (194.61,124.31) .. controls (184.31,113.09) and (174.04,95.95) .. (170.86,84.77) -- cycle ;
\draw  [dash pattern=on 1pt off 1.2pt] (131.99,124.62) .. controls (123.53,132.49) and (106.93,144.4) .. (100.32,138.18) .. controls (93.71,131.96) and (83.68,117.19) .. (86.36,107.81) .. controls (89.04,98.43) and (107.63,90.64) .. (112.03,83.33) .. controls (116.43,76.02) and (114.89,50.45) .. (122.73,47.28) .. controls (130.58,44.11) and (143.61,44.22) .. (154.85,51.06) .. controls (166.09,57.91) and (161.65,71.09) .. (156.32,85.45) .. controls (151,99.8) and (140.45,116.75) .. (131.99,124.62) -- cycle ;
\draw   (302.56,67.93) .. controls (306.36,67.93) and (309.44,71.01) .. (309.44,74.8) -- (309.44,149.2) .. controls (309.44,153) and (306.36,156.08) .. (302.56,156.08) -- (281.94,156.08) .. controls (278.14,156.08) and (275.06,153) .. (275.06,149.2) -- (275.06,74.8) .. controls (275.06,71.01) and (278.14,67.93) .. (281.94,67.93) -- cycle ;
\draw   (378.56,90.28) .. controls (382.36,90.28) and (385.44,93.36) .. (385.44,97.15) -- (385.44,149.54) .. controls (385.44,153.34) and (382.36,156.41) .. (378.56,156.41) -- (357.94,156.41) .. controls (354.14,156.41) and (351.06,153.34) .. (351.06,149.54) -- (351.06,97.15) .. controls (351.06,93.36) and (354.14,90.28) .. (357.94,90.28) -- cycle ;
\draw   (427.56,90.28) .. controls (431.36,90.28) and (434.44,93.36) .. (434.44,97.15) -- (434.44,149.54) .. controls (434.44,153.34) and (431.36,156.41) .. (427.56,156.41) -- (406.94,156.41) .. controls (403.14,156.41) and (400.06,153.34) .. (400.06,149.54) -- (400.06,97.15) .. controls (400.06,93.36) and (403.14,90.28) .. (406.94,90.28) -- cycle ;
\draw   (479.48,90.54) .. controls (483.28,90.54) and (486.36,93.61) .. (486.36,97.41) -- (486.36,149.8) .. controls (486.36,153.59) and (483.28,156.67) .. (479.48,156.67) -- (458.85,156.67) .. controls (455.05,156.67) and (451.98,153.59) .. (451.98,149.8) -- (451.98,97.41) .. controls (451.98,93.61) and (455.05,90.54) .. (458.85,90.54) -- cycle ;
\draw   (531.56,91.03) .. controls (535.36,91.03) and (538.44,94.11) .. (538.44,97.9) -- (538.44,150.29) .. controls (538.44,154.09) and (535.36,157.17) .. (531.56,157.17) -- (510.94,157.17) .. controls (507.14,157.17) and (504.06,154.09) .. (504.06,150.29) -- (504.06,97.9) .. controls (504.06,94.11) and (507.14,91.03) .. (510.94,91.03) -- cycle ;
\draw  [line width=2pt]  (188,67) -- (109,115) -- (193,158) -- (293,97) -- (368,123) -- (417,123) ;
\draw [fill=green] (188,67) circle (2.5pt);
\draw [fill=uuuuuu] (109,115) circle (2.5pt);
\draw [fill=uuuuuu] (193,158) circle (2.5pt);
\draw [fill=uuuuuu] (293,97) circle (2.5pt);
\draw [fill=uuuuuu] (368,123) circle (2.5pt);
\draw [fill=uuuuuu] (417,123) circle (2.5pt);
\draw (183,73) node [anchor=north west][inner sep=0.75pt]   [align=left] {$v$};
\draw (155,33.4) node [anchor=north west][inner sep=0.75pt]   [align=left] {$V_1$};
\draw (224,142) node [anchor=north west][inner sep=0.75pt]   [align=left] {$V_2$};
\draw (86,142) node [anchor=north west][inner sep=0.75pt]   [align=left] {$V_3$};
\draw (86,70) node [anchor=north west][inner sep=0.75pt]   [align=left] {$U_1$};
\draw (220,70) node [anchor=north west][inner sep=0.75pt]   [align=left] {$U_2$};
\draw (155,180) node [anchor=north west][inner sep=0.75pt]   [align=left] {$U_3$};

\end{tikzpicture}

\caption{auxiliary figure for the proof of Proposition~\ref{PROP:Final-compute}.} 
\label{Fig:Prop-2-4}
\end{figure}
    
    Let $\beta$ be the real number such that $\sum_{i\in [t]}|V_{i}|=\frac{\beta n}{3k-2}$.
    It follows from Claim~\ref{Claim:size-of-Bi} that $\beta \ge 4t$. 
    By symmetry, we may assume that $B_1$ is the smallest among $\left\{B_i \colon i \in [t]\right\} \cup \left\{C_i \colon i \in [t]\right\}$.
    In particular, $|B_{1}| \le  \frac{1}{2t}\sum_{i\in [t]}|V_{i}| = \frac{\beta n}{2t(3k-2)}$. 
    Fix a vertex $v \in  B_{1}$.
    By Proposition~\ref{PROP:Prelim-B}~\ref{PROP:Prelim-B-1},~\ref{PROP:Prelim-B-2}, and Claim~\ref{Claim:many-large-parts}, there exists an edge $e$ (see Figure~\ref{Fig:Prop-2-4}) containing $v$ such that 
    \begin{align*}
        [t] \subseteq  \varphi(e)
        \quad\text{and}\quad
        \max\left\{|V_j| \colon j\in \overline{\varphi(e)} \right\} 
        \le \min\left\{|V_j| \colon j \in \varphi(e\setminus \{v\})\right\}. 
    \end{align*}
    In particular, it follows from a simple averaging argument that 
    \begin{align*}
        \sum_{j\in \overline{\varphi(e)}}|V_{j}|
        \le \frac{k-r}{k-t} \cdot \sum_{j \in [t+1,k]} |V_j|
        \le \frac{k-r}{k-t} \cdot \left(n - \frac{\beta n}{3k-2}\right). 
    \end{align*}
    Since $B_1\cap e \neq \emptyset$ and $[t] \subseteq  \varphi(e)$, it is easy to see inductively that $V_{i} \cap e \subseteq  B_{i}$ for $i \in [2,t]$ (see Figure~\ref{Fig:Prop-2-4}).
    Therefore, $N_{\mathcal{H}}(e\setminus \{v\}) \cap V_1 \subseteq B_1$. 
    It follows that 
    \begin{align*}
        d_{\mathcal{H}}(e\setminus\{v\})
        \le  |B_{1}|+\sum_{j\in \overline{\varphi(e)}}|V_{j}|
        & \le  \frac{\beta n}{2t(3k-2)} + \frac{k-r}{k-t} \cdot \left(n - \frac{\beta n}{3k-2}\right) \\
        & = \frac{k-r}{k-t} n - \left(\frac{k-r}{k-t} - \frac{1}{2t}\right) \cdot \frac{\beta n}{3k-2}. 
    \end{align*}
    Since $k \ge \frac{4r-2}{3}$ and $t \ge 3$, we have 
    \begin{align*}
        \frac{k-r}{k-t} - \frac{1}{2t} 
        = \frac{(2k-2r+1)t-k}{2t(k-t)} 
        \ge \frac{3(2k-2r+1)-k}{2t(k-t)}
        = \frac{5k-6r+3}{2t(k-t)}
        > 0. 
    \end{align*}
    Therefore, the inequality above for $d_{\mathcal{H}}(e\setminus\{v\})$ continues as 
    \begin{align*}
        d_{\mathcal{H}}(e\setminus\{v\})
        & \le \frac{k-r}{k-t} n - \left(\frac{k-r}{k-t} - \frac{1}{2t}\right) \cdot \frac{4t n}{3k-2} \\
        & = \frac{3k-3r+1}{3k-2} n - \frac{(k-r+1) t + k-2 r}{(3 k-2) (k-t)} n \\
        & \le \frac{3k-3r+1}{3k-2} n - \frac{4 k-5 r+3}{(3 k-2) (k-t)} n
        < \frac{3k-3r+1}{3k-2} n
    \end{align*}
    a contradiction. 
    This completes the proof of Proposition~\ref{PROP:Final-compute} for the case $k \ge \frac{4r-2}{3}$.  
\end{proof}

Next, we consider the case $k < \frac{4r-2}{3}$. 

\begin{proof}[Proof of Proposition~\ref{PROP:Final-compute} for $k < \frac{4r-2}{3}$]
     Let $n \ge k \ge r \ge 2$ be integers satisfying $k < \frac{4r-2}{3}$. 
    Let $\mathcal{H}$ be an $n$-vertex $k$-partite $r$-graph with no isolated vertices and with $\delta_{r-1}^{+}(\mathcal{H}) > \frac{k-r+1}{k+2}n$. 
    Fix $\varphi, \vartheta \in \mathrm{Hom}(\mathcal{H},K_{k}^{r})$ such that $|\vartheta(\varphi^{-1}(i))| \le 2$ for every $i \in [k]$. 
    Let $V_{i}\coloneqq \varphi^{-1}(i)$ and $U_{i}\coloneqq \vartheta^{-1}(i)$ for $i\in [k]$.  
    Suppose to the contrary that $\varphi \ncong \vartheta$. 

    Let $T \subseteq [k]$ and $\{k_1, \ldots, k_t\} \subseteq [k]$ be the $t$-sets guaranteed by Proposition~\ref{PROP:cyclic-structure}, where $t \ge 3$. 
    Similar to the proof of Proposition~\ref{PROP:cyclic-structure}~\ref{PROP:cyclic-structure-2}, we may assume that $T = [t]$ and $(k_1, \ldots, k_t) = (1, \ldots, t)$. 
    For convenience, let $B_{i} \coloneqq V_{i}\cap U_{i+1}$ and $C_{i} \coloneqq V_{i}\cap U_i$ for $i\in [t]$, with the indices taken modulo $t$ (see Figure~\ref{Fig:cyclic-structure}). 
    \begin{claim}\label{Claim:C2_Bt-b}
        We have $\{v,u\} \not\in \partial_{r-2}\mathcal{H}$ for every pair $(v,u) \in C_2 \times B_t$.
    \end{claim}
    \begin{proof}[Proof of Claim~\ref{Claim:C2_Bt-b}]
        Suppose to the contrary that there exists a pair $(v,u) \in C_2 \times B_t$ such that $\{u,v\} \in \partial_{r-2}\mathcal{H}$. 
        Notice that every edge containing $\{u,v\}$ must have empty intersection with $V_1$. 
        Therefore, we have $k \ge r+1$. 
        By Proposition~\ref{PROP:Prelim-C}, there exists an edge $e\in \mathcal{H}$ containing $\{u,v\}$ such that 
        \begin{align*}
            \max\left\{|V_i| \colon i \in \overline{\varphi(e)} \right\}
            \le \min\left\{|V_i| \colon i \in \varphi(e) \setminus \{2,t\}\right\}. 
        \end{align*}
        Fix $i_{\ast} \in \varphi(e)\setminus \{2,t\}$ such that 
        \begin{align*}
            |V_{i_{\ast}}|
            = \min\left\{|V_j| \colon j \in \varphi(e) \setminus \{2,t\} \right\}. 
        \end{align*}
        Let $w$ denote the vertex in $e\cap V_{i_{\ast}}$. 
        Let  $K \coloneqq [k]\setminus \{1,2,t\}$ and $T \coloneqq \overline{\varphi(e\setminus\{w\}) \cup \{1\}} \subseteq K$. 
        Since $N_{\mathcal{H}}(e\setminus\{w\}) \cap V_1 = \emptyset$, we obtain 
        \begin{align*}
            \frac{k-r+1}{k+2}n 
            <  d_{\mathcal{H}}\left( e\setminus\{w\}\right)
            \le  \sum_{j\in T}\left| V_{j}\right|. 
        \end{align*}
        %
        %
        Choose a set $S \subseteq [k]$ of size $2(k-r+1)$ such that $\{1,2,t\} \subseteq S$. 
        Note that this is possible since $k \ge r+1$ (implying $2(k-r+1) \ge 4$) and $k < \frac{4r-2}{3}$ (implying $2(k-r+1) \le k$). 
        By Proposition~\ref{PROP:Prelim-A}~\ref{PROP:Prelim-A-1}, we have $\sum_{j \in S} |V_j| \ge 2\cdot \delta_{r-1}^{+}(\mathcal{H}) \ge 2\cdot \frac{k-r+1}{k+2}n$. 
        On the other hand, it follows from the definition of $e$ and $w$ that 
        \begin{align*}
            \max\left\{|V_j| \colon j \in T\right\}
            \le \max\left\{|V_j| \colon j \in K\setminus T\right\},  
        \end{align*}
        Since $\overline{S} \subseteq K$ and $|\overline{S}| = k-2(k-r+1) \ge k-r = |T|$, a simple averaging argument shows that 
        \begin{align*}
            \sum_{j \in \overline{S}} |V_j|
            \ge \frac{k-2(k-r+1)}{k-r} \cdot \sum_{j\in T}\left| V_{j}\right|
            > \frac{2r-2-k}{k-r} \cdot \frac{k-r+1}{k+2}n. 
        \end{align*}
        Consequently, 
        \begin{align*}
            \sum_{j\in [k]}|V_j|
             = \sum_{j\in S}|V_j| + \sum_{j\in \overline{S}}|V_j|
            & > 2\cdot \frac{k-r+1}{k+2}n + \frac{2r-2-k}{k-r} \cdot \frac{k-r+1}{k+2}n  \\
            & = \left(1 + \frac{4 r-2 -3 k}{(k+2) (k-r)}\right)n 
            \ge n, 
        \end{align*}
        a contradiction. 
    \end{proof}

     %
    Fix a vertex $v \in  B_t \subseteq V_t$.
    By Proposition~\ref{PROP:Prelim-B}~\ref{PROP:Prelim-B-3}, there exists an edge $e\in \mathcal{H}$ containing $v$ such that 
    \begin{align}\label{equ:main-lemma-e}
        2 \in \varphi(e)
        \quad\text{and}\quad 
        \max\left\{|V_i| \colon i \in \overline{\varphi(e)}\right\} 
        \le \min\left\{|V_i| \colon i \in \varphi(e)\setminus \{2,t\} \right\}.  
    \end{align}
    Let $u$ denote the vertex in $e\cap V_2$.
    It follows from Claim~\ref{Claim:C2_Bt-b} that $u \in B_2$. 
    Let $K \coloneqq [k]\setminus \{1,2,t\}$. 
    Note that $\{2,t\} \subseteq \varphi(e)$. 
    \begin{claim}\label{CLAIM:C1-C2-J}
        There exists a $(k-r-1)$-set $I \subseteq \overline{\varphi(e) \cup \{1\}} \subseteq K$ such that 
        \begin{align}\label{equ:B1-B2-I-lower}
            |B_1| + |B_2| + \sum_{j \in I} |V_j|
            > \frac{k-r+1}{k+2}n. 
        \end{align}
    \end{claim}
    \begin{proof}[Proof of Claim~\ref{CLAIM:C1-C2-J}]
    Define an edge $\tilde{e} \in \mathcal{H}$ and a $(k-r-1)$-set $I \subseteq  [k]\setminus \{1,2\}$ according to the following rules$\colon$
    \begin{itemize}
        \item If $1 \notin \varphi(e)$, then let $\tilde{e}\coloneqq e$ and $I \coloneqq \overline{\varphi(e) \cup \{1\}}$.
        \item If $1 \in \varphi(e)$ and $|B_1| > \frac{k-r+1}{k+2}n$, then let $\tilde{e} \coloneqq e$ and $I \subseteq \overline{\varphi(e)}$ be an arbitrary set of size $k-r-1$.
        \item Suppose that $1 \in \varphi(e)$ but $|B_1| \le \frac{k-r+1}{k+2}n$. 
        Then let $w$ denote the vertex in $e \cap V_1$.
        Note that $w$ is contained in $B_1$ since $e\cap B_t \neq \emptyset$. 
        For the same reason, we have $N_{\mathcal{H}}(e\setminus \{w\}) \cap V_1 \subseteq  B_1$. 
        Since $d_{\mathcal{H}}(e\setminus\{w\}) > \frac{k-r+1}{k+2} \ge |B_1|$, the set $N_{\mathcal{H}}(e\setminus \{u\}) \setminus B_1$ is nonempty.
        Fix a vertex $\tilde{w} \in  N_{\mathcal{H}}(e\setminus \{u\}) \setminus B_1$, noting that $\tilde{w} \not\in V_1$. 
        Let $\tilde{e} \coloneqq (e\setminus \{w\}) \cup \{\tilde{w}\}$ and $I \coloneqq \overline{\varphi(\tilde{e}) \cup \{1\}} \subseteq \overline{\varphi(e)}$, noting that $1\not\in \varphi(\tilde{e})$.   
    \end{itemize}
    Note that in all three cases we have $1\not\in I$ and $I \subseteq \overline{\varphi(e)}$, and hence, $I \subseteq \overline{\varphi(e) \cup \{1\}}$. 
    Additionally, in all three cases we have $\{v\} = \tilde{e} \cap B_t$. 
    Therefore, $N_{\mathcal{H}}(\tilde{e}\setminus\{u\})\cap V_{1} \subseteq  B_{1}$. 
    Furthermore, by Claim~\ref{Claim:C2_Bt-b}, we have $N_{\mathcal{H}}(\tilde{e}\setminus\{u\}) \cap V_{2}\subseteq  B_{2}$. 
    Therefore, it follows from the definition of $\tilde{e}$ and $I$ that  
    \begin{align*}
        \mathrm{either}\quad 
        |B_1| 
        > \frac{k-r+1}{k+2}n
        \quad\mathrm{or}\quad
        \frac{k-r+1}{k+2}n
        < d_{\mathcal{H}}(e\setminus \{u\})
        \le |B_1|+|B_2| +\sum_{i\in I}|V_i|.
    \end{align*} 
    In both cases, we have $\frac{k-r+1}{k+2}n < |B_1|+|B_2| +\sum_{i\in I}|V_i|$.  
    \end{proof}

    Let $I \subseteq \overline{\varphi(e) \cup \{1\}} \subseteq K$ be the $(k-r-1)$-set guaranteed by Claim~\ref{CLAIM:C1-C2-J}. 
    By symmetry, there exists a $(k-r-1)$-set $J \subseteq  K$ such that 
    \begin{align}\label{equ:C1-C2-J-lower}
        |C_1|+|C_2|+\sum_{j\in J} |V_j|
        > \frac{k-r+1}{k+2}n. 
    \end{align}
    \begin{claim}\label{CLAIM:Q-vs-I}
        There exists a $(k-r-1)$-set $Q \subseteq K\setminus J$ such that $\sum_{j\in Q}|V_j| \ge \sum_{j\in I}|V_j|$. 
    \end{claim}
    \begin{proof}[Proof of Claim~\ref{CLAIM:Q-vs-I}]
        Note that $|K\setminus J| \ge k-3 - (k-r-1) = r-2 \ge k-r$ (since $k < \frac{4r-2}{3}$). 
        So there exists a $(k-r)$-set $\hat{Q} \subseteq K\setminus J$.  
        %
        Fix $i_{\ast}\in \hat{Q}$ such that $|V_{i_{\ast}}| = \min \left\{ |V_{j}| \colon j\in \hat{Q} \right\}$. 
        Fix $j_{\ast} \in \overline{\varphi(e)}$ such that $|V_{j_{\ast}}| = \min \left\{ |V_{j}| \colon j\in \overline{\varphi(e)} \right\}$. 
        Let $Q \coloneqq \hat{Q}\setminus \{i_{\ast}\}$. 
        Since $\overline{\varphi(e)} \subseteq K$, it follows from~\eqref{equ:main-lemma-e} that 
        \begin{align*}
            \sum_{j\in Q}|V_j| 
            \ge \sum_{j\in \overline{\varphi(e)}\setminus \{j_{\ast}\}}|V_j| 
            = \sum_{j\in \overline{\varphi(e)}}|V_j| - \min\left\{|V_j| \colon j \in \overline{\varphi(e)}\right\}. 
        \end{align*}
        Since $I \subseteq \overline{\varphi(e)}$ has size $|\overline{\varphi(e)}|-1$, we have 
        \begin{align*}
            \sum_{j\in I}|V_j|
            \le \sum_{j\in \overline{\varphi(e)}}|V_j| - \min\left\{|V_j| \colon j \in \overline{\varphi(e)}\right\}, 
        \end{align*}
        which implies that $\sum_{j\in Q}|V_j| \ge \sum_{j\in I}|V_j|$. 
    \end{proof}
     
    Let $Q \subseteq K\setminus J$ be the $(k-r-1)$-set ensured by Claim~\ref{CLAIM:Q-vs-I}. 
    Let $R \coloneqq \overline{J \cup Q\cup \{1,2\}}$. 
    Since $|R| = k-2-(k-r-1)-(k-r-1) = 2r-k \ge k-r-1$ (by $k < \frac{4r-2}{3}$), it follows from Proposition~\ref{PROP:Prelim-A}~\ref{PROP:Prelim-A-1} that 
    \begin{align*}
        \sum_{j\in R} |V_j|
        \ge \frac{2r-k}{k-r+1} \cdot \frac{k-r+1}{k+2} n
        = \frac{2r-k}{k+2}n. 
    \end{align*}
    Combining this with~\eqref{equ:B1-B2-I-lower},~\eqref{equ:C1-C2-J-lower}, and Claim~\ref{CLAIM:Q-vs-I}, we obtain  
    \begin{align*}
        \sum_{j\in [k]}|V_j|
        & = \left(|B_1|+|B_2| + \sum_{i\in Q} |V_i|\right)
            +\left(|C_1|+|C_2| + \sum_{i\in J} |V_i|\right) 
            + \sum_{i\in R} |V_i| \\ 
        & \ge \left(|B_1|+|B_2| + \sum_{i\in I} |V_i|\right)
            +\left(|C_1|+|C_2| + \sum_{i\in J} |V_i|\right) 
            + \sum_{i\in R} |V_i| \\
        &> \frac{k-r+1}{k+2}n + \frac{k-r+1}{k+2}n + \frac{2r-k}{k+2}n
        =n, 
    \end{align*}
    a contradiction. 
    This completes the proof of Proposition~\ref{PROP:Final-compute} for the case $k < \frac{4r-2}{3}$. 
\end{proof}

\section{Proofs for Theorem~\ref{THM:s-degree-beta} and Corollary~\ref{CORO:i-degree-tight-bound}}\label{SEC:proof-i-degree}
We prove Theorem~\ref{THM:s-degree-beta} and Corollary~\ref{CORO:i-degree-tight-bound} in this section. 

We will use the following result, proved by Frankl--F\"{u}redi--Kalai~\cite{FFK88} and later extended in~\cite{LM21,LMuk23}. 
\begin{theorem}[\cite{FFK88},~{\cite[Theorem~1.17]{LM21}}]\label{THM:k-partite-kruskal-katona}
    Let $k \ge r > i \ge 1$ be integers.
    Suppose that $\mathcal{H}$ is a $k$-partite $r$-graph. Then 
    \begin{align*}
        \left(\frac{|\mathcal{H}|}{\binom{k}{r}}\right)^{1/r}
        \le \left(\frac{|\partial_{i}\mathcal{H}|}{\binom{k}{r-i}}\right)^{1/(r-i)}.
    \end{align*}
\end{theorem}

\begin{proof}[Proof of Theorem~\ref{THM:s-degree-beta}]
    Let $n \ge k \ge r > i \ge 1$ be integers. 
    It suffices to show that 
    \begin{align*}
        \left ( \frac{\Phi_{k,r,i}}{\binom{k-i}{r-i}} \right ) ^{1/(r-i)} 
        \le  \left ( \frac{\Phi_{k,r-1,i}}{\binom{k-i}{r-1-i}} \right ) ^{1/(r-1-i)}. 
    \end{align*}
    Suppose to the contrary that the inequality above is not true. 
    Then there exists a $k$-partite $n$-vertex $r$-graph $\mathcal{H}$ with 
    \begin{align*}
        \delta^{+}_{i}(\mathcal{H}) 
        > \binom{k-i}{r-i} \left ( \frac{\Phi_{k,r-1,i}}{\binom{k-i}{r-1-i}} \right ) ^{\frac{r-i}{r-1-i}} n^{r-i}
    \end{align*}
    and without isolated vertices that is not uniquely $k$-colorable. 
    
    Take an arbitrary $e \in \partial_{r-i} \mathcal{H}$. 
    Since $\mathcal{H}$ is $k$-partite, the link $L_{\mathcal{H}}(e)$ is $(k-i)$-partite. 
    Note that $\partial L_{\mathcal{H}}(e) = L_{\partial\mathcal{H}}(e)$, so it follows from Theorem~\ref{THM:k-partite-kruskal-katona} that 
    \begin{align*}
        d_{\partial\mathcal{H}}(e)
        = |L_{\partial\mathcal{H}}(e)|
        = |\partial L_{\mathcal{H}}(e)| 
        & \ge \binom{k-i}{r-i-1} \left(\frac{|L_{\mathcal{H}}(e)|}{\binom{k-i}{r-i}}\right)^{\frac{r-i-1}{r-i}} \\
        & \ge \binom{k-i}{r-i-1} \left(\frac{\delta_{i}^{+}(\mathcal{H})}{\binom{k-i}{r-i}}\right)^{\frac{r-i-1}{r-i}} \\ 
        & > \binom{k-i}{r-i-1} \left( \left ( \frac{\Phi_{k,r-1,i}}{\binom{k-i}{r-1-i}} \right ) ^{\frac{r-i}{r-1-i}} n^{r-i} \right)^{\frac{r-i-1}{r-i}} \\
        & = \Phi_{k,r-1,i} \cdot n^{r-1-i}. 
    \end{align*}
    This implies that $\partial\mathcal{H}$ is a $k$-partite $n$-vertex $(r-1)$-graph with $\delta^{+}_{i}(\partial\mathcal{H}) > \Phi_{k,r-1,i} \cdot n^{r-1-i}$. 
    Since $\mathcal{H}$ has no isolated vertices, $\partial\mathcal{H}$ has no isolated vertices as well. 
    Therefore, $\partial\mathcal{H}$ is uniquely $k$-colorable. Consequently, $\mathcal{H}$ is uniquely $k$-colorbale, a contradiction. 
\end{proof}

\begin{proof}[Proof of Corollary~\ref{CORO:i-degree-tight-bound}]
    Applying Theorem~\ref{THM:s-degree-beta} with $(r_1, r_2) = (r, i+1)$, we obtain 
    \begin{align*}
        \Phi_{k,r,i} 
        \le \binom{k-i}{r-i} \left(\frac{\Phi_{k,i+1,i}}{k-i}\right)^{r-i}
        = \binom{k-i}{r-i} \cdot \left(\frac{1}{k-i} \cdot \max \left \{ \frac{k-i}{k+2} ,\ \frac{3k-3i-2}{3k-2} \right \}\right)^{r-i}.
    \end{align*}
    On the other hand, simply calculations show that the construction $\mathcal{H}_{k,r}(1, m)$ defined in Section~\ref{SEC:Intorduction} satisfies 
    \begin{align*}
        \delta_{i}^{+}(\mathcal{H}_{k,r}(1, m))
        = \binom{k-i}{r-i}\left(\frac{1}{k+2}\right)^{r-i} \cdot \left(v(\mathcal{H}_{k,r}(1, m))\right)^{r-i}. 
    \end{align*}
    Since $\mathcal{H}_{k,r}(1, m)$ is not uniquely $k$-colorable, we have 
    $\Phi_{k,r,i} \ge \binom{k-i}{r-i}\left(\frac{1}{k+2}\right)^{r-i}$ for all $k \ge r> i \ge 1$. 
    Therefore, $\Phi_{k,r,i} = \binom{k-i}{r-i}\left(\frac{1}{k+2}\right)^{r-i}$ for all $k \ge r> i \ge 1$ that satisfy $\frac{k-i}{k+2} \ge \frac{3k-3i-2}{3k-2}$, which is equivalent to $k \le \frac{4i+2}{3}$. 
\end{proof}

\section{Concluding remarks}\label{SEC:Remarks}
In general, given two $r$-graphs $\mathcal{H}$ and $\mathcal{G}$, let $\mathrm{SHom}(\mathcal{H},\mathcal{G})$ denote the collection of all surjective homomorphisms from $\mathcal{H}$ to $\mathcal{G}$. 
We say $\mathcal{H}$ is \textbf{surjectively $\mathcal{G}$-colorable} if $\mathrm{SHom}(\mathcal{H},\mathcal{G}) \neq \emptyset$. 
Similarly, an $r$-graph $\mathcal{H}$ is \textbf{uniquely surjectively $\mathcal{G}$-colorable} if for every pair $\psi_1, \psi_2 \in \mathrm{SHom}(\mathcal{H}, \mathcal{G})$ there exists an automorphism $\eta\in \mathrm{Aut}(\mathcal{G})$ such that $\eta \circ \psi_1 = \psi_2$.

As one way to extend Theorem~\ref{THM:Bollobas-clique}, Lai~\cite{Lai87a,Lai89b} considered the minimum degree constraint that forces a graph to be uniquely surjectively $C_{2k+1}$-colorable.
\begin{theorem}[Lai~\cite{Lai87a}]\label{THM:Lai-odd-cycle-1}
    Let $k \ge 2$ and $n\ge 2k+1$ be integers.
    Suppose that $G$ is a surjectively $C_{2k+1}$-colorable graph on $n$ vertices with 
    \begin{align*}
        \delta(G) \ge \frac{n}{k+1}.
    \end{align*}
    Then $G$ is uniquely surjectively $C_{2k+1}$-colorable. Moreover, the constant $\frac{1}{k+1}$ is optimal. 
\end{theorem}
Extensions of Theorems~\ref{THM:Bollobas-clique} and~\ref{THM:Lai-odd-cycle-1} to general graphs were explored in recent work~\cite{HLZ24}, but $K_{k}$ and $C_{2k+1}$ remain the only two classes of graphs for which tight bounds are known. 

One could also explore extensions of Theorems~\ref{THM:Bollobas-clique} and~\ref{THM:Lai-odd-cycle-1} to $r$-graphs with $r\ge 3$. In particular, the following question concerning the tight cycles seems durable. 
\begin{problem}
    Let $k \ge r \ge 3$ and $1 \le i \le r-1$ be integers. 
    Determine the minimum real number $\Psi_{k,r,i}$ such that for $n \ge k$, every surjectively $C_{k}^{r}$-colorable $r$-graph without isolated vertices, and satisfying $\delta^{+}_{i}(\mathcal{H}) > \Psi_{k,r,i}\cdot n^{r-i}$, is uniquely surjectively $C_{k}^{r}$-colorable. 
    Here, $C_{k}^{r}$ denotes the $r$-uniform tight cycle on $k$ vertices. 
\end{problem}
Corollary~\ref{CORO:i-degree-tight-bound} suggests that the extremal constructions for $\Phi_{k,r,i}$, when $k >\frac{4i+2}{3}$, might come from $r$-graphs $\mathcal{H}_{k,r}(\alpha, m)$.
Hence, we propose the following conjecture. 

\begin{conjecture}\label{CONJ:i-degree}
    Suppose that $k \ge r > i \ge 1$ are integers satisfying $k > \frac{4i+2}{3}$. 
    Then 
    \begin{align*}
        \Phi_{k,r,i}
        = \sup\left \{ \frac{\delta_{i}^{+}(\mathcal{H}_{k,r}(\alpha, m))}{\left(v(\mathcal{H}_{k,r}(\alpha, m))\right)^{r-i}}
        \colon \alpha>0 ~\mathrm{and}~ m\in \mathbb{N}^{+} \right \}. 
    \end{align*}
\end{conjecture}
Heuristically, Conjecture~\ref{CONJ:i-degree} becomes more challenging as $i$ decreases, with the hardest case likely being $i=1$.
In the following theorem, we determine $\Phi_{3,3,1}$, providing weak evidence in support of Conjecture~\ref{CONJ:i-degree}.
\begin{theorem}\label{THM:Phi-331}
    We have $\Phi_{3,3,1}=\frac{1}{18}$.
\end{theorem}
\begin{proof}[Proof of Theorem~\ref{THM:Phi-331}]
    The lower bound $\Phi_{3,3,1}\ge \frac{1}{18}$ comes from the construction $\mathcal{H}_{3,3}(2,m)$. 
    So it suffices to prove the upper bound. 
    Let $\mathcal{H}$ be an $n$-vertex $3$-partite $3$-graph  with $\delta(\mathcal{H})>\frac{n^{2}}{18}$. We will show that $\mathcal{H}$ is uniquely $3$-colorable. 
    
    Suppose to the contrary that there exist two non-equivalent homomorphisms $\varphi, \vartheta\in \mathrm{Hom}(\mathcal{H},K_{3}^{3})$. 
    By symmetry, we may assume that 
    \begin{align*}
        \max\left\{|\varphi^{-1}(i)| \colon i\in [3] \right\} 
        \ge \max\left\{|\vartheta^{-1}(i)| \colon i\in [3] \right\}.
    \end{align*}
    Let $V_{i}\coloneqq \varphi^{-1}(i)$ for $i\in [3]$ and by relabelling the vertices in $K_{3}^{3}$, we may assume $|V_{1}|\ge |V_{2}|\ge |V_{3}|$. So $|V_1|\ge\frac{n}{3}.$
    \begin{claim}\label{Claim:(3,3,1)-a}
        For every pair of distinct vertices $v_{1}, v_{2}\in V_{1}$, there exist vertices $u \in V_2$ and $w\in V_3$ such that $\{v_1, u, w\} \in \mathcal{H}$ and $\{v_2, u, w\} \in \mathcal{H}$.
    \end{claim}
    \begin{proof}[Proof of Claim~\ref{Claim:(3,3,1)-a}]
        Suppose to the contrary that this is not true. 
        Then for every $u \in V_2$ we have $N_{\mathcal{H}}(v_1 u) \cap N_{\mathcal{H}}(v_2 u) = \emptyset$.
        Notice that $N_{\mathcal{H}}(v_1 u) \subseteq  V_3$ and $N_{\mathcal{H}}(v_2 u) \subseteq  V_3$, hence, 
        \begin{align*}
            2 \cdot \delta(\mathcal{H})
            \le d_{\mathcal{H}}(v_{1})+d_{\mathcal{H}}(v_{2}) 
            & = \sum_{u\in V_{2}} \left( |N_{\mathcal{H}}(v_1 u)| + |N_{\mathcal{H}}(v_2 u)|\right) \\
            & \le \sum_{u\in V_{2}}|V_{3}|
            =|V_{2}|\cdot |V_{3}|
            \le \left(\frac{n-|V_1|}{2}\right)^2
            \le \frac{n^{2}}{9}, 
        \end{align*}
        a contradiction. 
    \end{proof}
    It follows from Claim~\ref{Claim:(3,3,1)-a} that  $\vartheta(v_{1})=\vartheta(v_{2})$ for every pair $\{v_{1}, v_{2}\} \subseteq  V_{1}$, i.e. $|\vartheta(V_{1})|=1$. 
    By symmetry, we may assume $\vartheta(V_{1})=\{1\}$.
    Let $V_{i,j}\coloneqq V_{i}\cap \vartheta^{-1}(j)$ and $n_{i,j} \coloneqq |V_{i,j}|$ for $(i,j) \in \{2,3\} \times [3]$. 
    Note that $n_{2,1}=n_{3,1}=0$, since otherwise we would have $\max\{|\varphi^{-1}(i)| \colon i\in [3]\} < \max\{|\vartheta^{-1}(i)| \colon i\in [3]\}$, a contradiction.
    So $\phi^{-1}(1)=V_1=\vartheta^{-1}(1)$.
    \begin{claim}\label{Claim:(3,3,1)-b}
        We have $n_{i, j}>0$ for every $(i,j) \in \{2,3\} \times \{2,3\}$. 
    \end{claim}
    \begin{proof}[Proof of Claim~\ref{Claim:(3,3,1)-b}]
        Let us present only the proof for the case $(i,j) = (2,2)$ since the other cases are analogous. 
        Suppose to the contrary that $n_{2, 2} = 0$. 
        Then we must have $n_{3, 3} > 0$ since otherwise we would have $\varphi\cong \vartheta$, a contradiction. 
        Let $v \in V_{3, 3}$ be a vertex. 
        Suppose that $\{u,w\}$ is an element in $L_{\mathcal{H}}(v)$. 
        Then either $(u,w) \in V_1 \times V_2$ or $(w,u) \in V_1 \times V_2$. 
        By symmetry, we may assume that $(u,w) \in V_1 \times V_2$. 
        Then $\{\vartheta(w)\} = [3]\setminus \{\vartheta(v), \vartheta(u)\} = \{2\}$, meaning that $w \in V_2\cap W_2$, contradicting the assumption that $n_{2, 2} = 0$. 
    \end{proof}
    Let $n_{1} \coloneqq |V_{1}|$. 
    Fix a vertex $u \in V_{2,2}$.
    Note that $\delta(\mathcal{H})
        \le d_{\mathcal{H}}(u)
        \le n_1 \cdot n_{3, 3}$.
    %
    Similarly, we have 
    \begin{align}\label{equ:delta-H-lower}
        \delta(\mathcal{H})
        \le n_1 \cdot n_{i, j}
        \quad\text{for every}\quad (i,j) \in \{2,3\} \times \{2,3\}.
    \end{align}
    Summing them up, we obtain 
    \begin{align*}
        \delta(\mathcal{H})
        \le n_1 \cdot \frac{n_{2,2}+n_{2,3}+n_{3,2}+n_{3,3}}{4}
        = \frac{n_1 (n-n_1)}{4}
        \le \frac{n^2}{16}.
    \end{align*}
    Fix a vertex $v\in V_{1}$.
    Then 
    \begin{align*}
        \delta(\mathcal{H})
        \le d_{\mathcal{H}}(v)
        \le n_{2,2}\cdot n_{3,3}+n_{2,3}\cdot n_{3,2}
        \le \left(\frac{n_{2,2}+n_{3,3}}{2}\right)^{2}+\left(\frac{n_{2,3}+n_{3,2}}{2}\right)^{2}. 
    \end{align*}
    Let $x \coloneqq \min \left\{\frac{n_{2,2}+n_{3,3}}{2}, \frac{n_{2,3}+n_{3,2}}{2}\right\}$ and $A \coloneqq \frac{n-n_1}{2}$. 
    Note that $x \le \frac{n-n_1}{4}$ and $x \ge \frac{\delta(\mathcal{H})}{n_1}$ (due to~\eqref{equ:delta-H-lower}). 
    Since the function $x^2+(A-x)^2$ is decreasing when $x\leq \frac{A}{2} = \frac{n-n_1}{4}$, we obtain 
    \begin{align*}
        \delta(\mathcal{H})
        \le x^2 + (A-x)^2
        \le \left( \frac{\delta(\mathcal{H})}{n_{1}} \right)^{2}+\left( \frac{n-n_{1}}{2}-\frac{\delta(\mathcal{H)}}{n_{1}}\right)^{2}.
    \end{align*}
    Rearrange this inequality and using the fact $n_1\ge \frac{n}{3}$, we obtain 
    \begin{align*}
        0 
        &\le \frac{(n-n_1)^2 n_1}{4} + \frac{2\cdot \delta^2(\mathcal{H})}{n_1} - n\cdot \delta(\mathcal{H})\\ 
        &\le \frac{n^3}{27} + \frac{6\cdot \delta^2(\mathcal{H})}{n} - n\cdot \delta(\mathcal{H})  = \frac{6}{n}\left(\delta(\mathcal{H}) - \frac{n^2}{18}\right)\left(\delta(\mathcal{H}) - \frac{n^2}{9}\right)
        <0, 
    \end{align*}
    a contradiction. 
    Here the last inequality is due to the fact $\frac{n^2}{18} < \delta(\mathcal{H})\le \frac{n^2}{16}$. 
    This finishes the proof of Theorem~\ref{THM:Phi-331}.
\end{proof}

\textbf{Remark.} From the above proof, we see that the unique extremal $3$-graph achieving $\Phi_{3,3,1}=\frac{1}{18}$ is the $3$-graph $\mathcal{H}_{3,3}(2,m)$.

\bibliographystyle{alpha}
\bibliography{uniquecolorable}

\begin{thebibliography}{FFK88}

\bibitem[BLP21]{BLP21}
Jozsef Balogh, Nathan Lemons, and Cory Palmer.
\newblock Maximum size intersecting families of bounded minimum positive
  co-degree.
\newblock {\em SIAM J. Discrete Math.}, 35(3):1525--1535, 2021.

\bibitem[Bol78]{Bol78}
B\'{e}la Bollob\'{a}s.
\newblock Uniquely colorable graphs.
\newblock {\em J. Combin. Theory Ser. B}, 25(1):54--61, 1978.

\bibitem[FFK88]{FFK88}
Peter Frankl, Zolt\'an F\"uredi, and Gil Kalai.
\newblock Shadows of colored complexes.
\newblock {\em Math. Scand.}, 63(2):169--178, 1988.

\bibitem[HLZ24]{HLZ24}
Jianfeng Hou, Xizhi Liu, and Hongbin Zhao.
\newblock A criterion for {A}ndr{\' a}sfai--{E}rd{\H o}s--{S}{\' o}s type
  theorems and applications.
\newblock {\em arXiv preprint arXiv:2401.17219}, 2024.

\bibitem[Lai87]{Lai87a}
Hong-Jian Lai.
\newblock Unique graph homomorphisms onto odd cycles.
\newblock {\em Utilitas Math.}, 31:199--208, 1987.

\bibitem[Lai89]{Lai89b}
Hong-Jian Lai.
\newblock Unique graph homomorphisms onto odd cycles. {II}.
\newblock {\em J. Combin. Theory Ser. B}, 46(3):363--376, 1989.

\bibitem[LM21]{LM21}
Xizhi Liu and Dhruv Mubayi.
\newblock The feasible region of hypergraphs.
\newblock {\em J. Combin. Theory Ser. B}, 148:23--59, 2021.

\bibitem[LM23]{LMuk23}
Xizhi Liu and Sayan Mukherjee.
\newblock Stability theorems for some {K}ruskal-{K}atona type results.
\newblock {\em European J. Combin.}, 110:Paper No. 103666, 20, 2023.

\end{thebibliography}
\end{document}